\documentclass[10pt]{amsart}
\usepackage{amssymb}
\usepackage{amsmath}
\usepackage{pstricks}
\usepackage[all]{xy}
\usepackage{latexsym}

\newtheorem{theorem}{Theorem}
\newtheorem*{conv*}{Conventions}{\bf}{\it}
\newtheorem{corollary}{Corollary}
\newtheorem{lemma}{Lemma}
\newtheorem{proposition}{Proposition}
\theoremstyle{definition}
\newtheorem{definition}{Definition}
{\bf}{\rm}
\newtheorem{remark}{Remark}
\newtheorem*{ack*}{Acknowledgments}
\newtheorem{example}{Example}

\newcommand{\com}{{\scriptstyle\bullet}}

\newcommand{\R}{\mathbb R}

\newcommand{\Z}{\mathbb Z}
\newcommand{\N}{\mathbb N}

\newcommand{\Sets}{\mathcal Sets}

\newcommand{\colim}{\mathop{\lim\limits_{\textstyle\longrightarrow}}\limits}

\newcommand{\hsset}{sSet_\infty}
\newcommand{\hcdga}{CDGA_\infty}
\newcommand{\sset}{sSet}
\newcommand{\cdga}{CDGA}
\newcommand{\hkmod}{k\textit{-}Mod_\infty}
\newcommand{\cga} {CGA}
\newcommand{\bordfr}{Bord^{fr}}
\newcommand{\CSS}{\mathcal{CS}e\mathcal{S}p}
\newcommand{\SeSp}{\mathcal{S}e\mathcal{S}p}
\newcommand{\pt}{\mathop{pt}_\com}
\newcommand{\Ball}{\mathcal{B}\mathit{all}}

\newcommand{\noprint}[1]{}

\title[Derived higher Hochschild, TCH,  Factorization algebras]{Derived Higher Hochschild Homology,  Topological Chiral Homology and Factorization algebras}

\author[G.~Ginot]{Gr\'egory Ginot}
\address{Gr\'egory Ginot, UPMC - Universit\'e Pierre et Marie Curie, \\
Institut  Math\'ematiques de Jussieu,   Case 247\\  4, place Jussieu, 75252 Paris Cedex 05, France \\ and   DMA - Ecole Normale Sup\'erieure,\\
45 rue d'Ulm, 75230 Paris, France}
\email{ginot@math.jussieu.fr}

\author[T.~Tradler]{Thomas~Tradler}
\address{Thomas Tradler, Department of Mathematics, College of Technology, City University of New York, 300 Jay Street, Brooklyn, NY 11201, USA,}
\email{ttradler@citytech.cuny.edu}

\author[M.~Zeinalian]{Mahmoud~Zeinalian}
\address{Department of Mathematics, C.W. Post Campus
of Long Island University, 720 Northern Boulevard, Brookville, NY
11548, USA} 
\email{mzeinalian@liu.edu}

\begin{document}

\begin{abstract}  
We study the higher Hochschild chain functor, factorization algebras  and their relationship with topological chiral homology. 
To this end, we emphasize that the higher Hochschild complex is a $(\infty, 1)$-functor $\hsset \times \hcdga$  where $\hsset$ and $\hcdga$ are the $(\infty,1)$-categories of simplicial sets and commutative differential graded algebras, and give 
an axiomatic characterization of this functor. 
From the axioms, we deduce several properties and computational tools for this functor. 
We study the relationship between the higher 
Hochschild functor and factorization algebras by showing that, in good cases, 
the Hochschild functor determines a constant commutative factorization algebra. 
Conversely,  every constant commutative factorization algebra is naturally equivalent to a Hochschild
 chain factorization algebra. Similarly, we study the relationship between the above concepts
 and topological chiral 
homology. In particular, we show that on their common domains of definition, the higher 
Hochschild functor is naturally equivalent to topological chiral homology.  
A byproduct of our proof yields  a  similar statement about the relationship between
 the blob complex and topological chiral homology of unoriented manifolds.
 Finally, we prove that topological chiral homology 
 determines a locally constant factorization algebra and, further, 
that this functor induces an equivalence between locally constant factorization algebras on a 
manifold and (local system of) $E_n$-algebras.We also deduce that Hochschild chains 
and topological chiral homology satisfies an exponential law, \emph{i.e.}, a Fubini type 
Theorem to compute them on products of manifolds.
\end{abstract}
\maketitle
\setcounter{tocdepth}{2}
\tableofcontents

\section{Introduction} 
In this paper, we study the higher Hochschild chain\footnote{Note, that in this paper we are using a \emph{cohomological} grading for our differential graded modules, see Convention (7) on page \pageref{convention7}. 
Also, this paper only deals with Hochschild chains and Hochschild homology, and never with Hochschild cochains or Hochschild cohomology.} complex $CH^{\com}_{X_\bullet}(A)$, functorially assigned to a simplicial set $X_\bullet$ (or a topological space), and a commutative differential graded algebra (CDGA) $A$, from an axiomatic point of view. Recently, motivated by topological quantum field theories, several concepts integrating (higher) categories of spaces or manifolds with those of algebras of different types have arisen. We also study the relationship between higher Hochschild chains, factorization algebras~\cite{CG,Co}, topological chiral homology~\cite{L-TFT,L-VI} and the blob complex~\cite{MW}.  
 Higher Hochschild homology was first introduced by Pirashvili in \cite{P}.
 The  higher Hochschild complexes (as well as other aforementioned concepts) are a generalization of the well-known and classical Hochschild complex. In fact,  for the case of the standard simplicial set model $X_\bullet=S^1_\bullet$ for the circle, $CH^{\com}_{S^1_\bullet}(A)$  reduces to the standard Hochschild complex $CH_{\com}(A)= A^{\otimes \com+1}$, see \cite{H,H2}.

In contrast with most other generalizations, higher Hochschild chains are defined over \emph{any} (simplicial set model of a) space and not only (stratified) manifolds. However, this forces us to restrict our attention to CDGAs or at best to $E_\infty$-algebras. More precisely, the higher Hochschild chains form a bifunctor $CH: \sset\times \cdga \to \cdga$ from the categories of simplicial sets and differential graded commutative algebras to the latter category. The functoriality with respect to spaces (and not merely manifold embeddings) is a key feature which allows us to derive algebraic operations on the higher Hochschild chain complexes from maps of topological spaces. For instance, it was crucially used  to study the Hodge decomposition of Hochschild homology  (Pirashvili~\cite{P}) or to give and study models of (higher) string topology~\cite{G,GTZ}. Also, its underlying  combinatorial properties  allow a generalization of Chen's iterated integral~\cite{GTZ}. Higher Hochschild is also a convenient setting to study holonomy of (higher) gerbes (for instance see~\cite{TWZ}) or compute  the observables of classical topological field theories, see~\cite{CG} and \S~\ref{S:Factorization}.      The higher Hochschild homology satisfies many axioms similar to those of Eilenberg-Steenrod for singular homology: naturality in each variable, commutations with coproducts in both variable, homotopy invariance and the dimension axiom, see  Corollary~\ref{C:Homologyfunctor}. 

However, to fully appreciate the higher Hochschild functor, one needs to go beyond mere homology and consider the higher Hochschild chains in a \emph{derived setting} which allows to formulate the analogue of the \emph{excision axiom}. This  axiom, reminiscent of the locality axioms of topological field theories,  asserts that Hochschild chains maps the homotopy pushout of simplicial sets to the derived tensor product of algebras, \emph{i.e.} homotopy pushout of CDGAs. This gluing property together with the homotopy invariance allow to build many examples of Hochschild chain complexes and to do computations as demonstrated in~\cite{GTZ}. Further, such an enhancement is needed in order to correctly compare the higher Hochschild functor with more sophisticated concepts, such as topological chiral homology, which naturally lies in a homotopical setting.  More precisely, we  interpret the higher Hochschild chains as a (derived) bifunctor from the $(\infty,1)$-categories $\hsset$ of simplicial sets  and  $\hcdga$ of CDGAs, which are  suitable localizations of the categories of simplicial sets and CDGAs, with respect to (weak) homotopy equivalences and quasi-isomorphisms. This framework (instead of simply homology) is also needed to keep track of the topology of topological spaces modeled by the simplicial sets; for instance the usual Hochschild complex $CH{_\ast}(A)$ interpretated in an $(\infty,1)$-category retains a circle action governing cylic homology as shown in~\cite{L-TFT,ToVe3}.  Here, following Rezk and Lurie \cite{Re,L-TFT}, an $(\infty,1)$-category means a complete Segal space. In our context, the $(\infty,1)$-categories we considered are obtained by a Dwyer-Kan localization process from standard model categories, though the results of this paper should not depend on the particular chosen approach to $(\infty,1)$-categories, see also Remark \ref{R:non-unique}.

\smallskip

 Our first main result is the following theorem.
\setcounter{section}{3} \setcounter{subsection}{2} \setcounter{theorem}{0}
\begin{theorem}\label{T:introderivedfunctor}
The Hochschild chains lift as a functor of $(\infty,1)$-categories $CH: \hsset \times \hcdga \to \hcdga$ which satisfies the following axioms
\begin{enumerate}
\item {\bf value on a point:} there is a natural equivalence  of CDGAs $CH^\com_{pt}(A)\cong A$.
\item {\bf monoidal:} there are natural equivalences of CDGAs
$$CH_{\coprod {X_i}_\com}^{\com}(A)\cong \bigotimes CH_{{X_i}_\com}^{\com}(A)$$
\item {\bf homotopy gluing/pushout:} $CH$ sends homotopy pushout in $\hsset$ to homotopy pushout in $\hcdga$, \emph{i.e.}  there is a natural equivalence of CDGAs
$$CH^{\com}_{X_\com \cup_{Z_\com}^{h} Y_\com}(A)\cong CH^{\com}_{X_\com}(A)\otimes_{CH^{\com}_{Z_\com}(A)}^{\mathbb{L}} CH^{\com}_{Y_\com}(A).$$   
\end{enumerate}
\end{theorem}

Furthermore, the above axioms actually \emph{define} the (derived) higher Hochschild chains: 
indeed our second main result, Theorem~\ref{T:deriveduniqueness}  can be rephrased as \setcounter{theorem}{1}
\begin{theorem}
The Hochschild chains is the \emph{unique}  bifunctor 
$\hsset \times \hcdga \to \hcdga$ satisfying the axioms (1), (2), (3) in 
Theorem~\ref{T:introderivedfunctor}.
\end{theorem}
These two results actually follow from the fact that $\cdga$ is tensored over simplicial sets and the general formalism of $(\infty,1)$-categories as in~\cite{Lu11,L-VI} and allow to interpret the Hochschild functor as a (derived) mapping stack in the context of~\cite{ToVe}, see Corollary~\ref{C:mappingstack}.   We also show that the derived Hochschild functor $CH:\hsset\times \hcdga\to \hcdga$ has many good formal properties: for instance it commutes with finite (homotopy) colimits in both arguments and with finite products of simplicial spaces (Corollary~\ref{C:hocolim} and Proposition~\ref{P:product}). Further, the locality axioms leads to an Eilenberg-Moore spectral sequence computing the higher Hochschild homology (Corollary~\ref{C:HlocalitySpecSeq}).

We also deal with the pointed versions of higher Hochschild chains, which allows to define Hochschild chains over a pointed simplicial set $X_\com$ of a CDGA $A$ with coefficient in an $A$-module $M$ and establish similar results for this theory.

By homotopy invariance, we can define $CH^\com_X(A)$ for a topological space $X$, 
generalizing the concept for a simplicial set $X_\com$, in such a way that  all of the above 
properties still hold. With this, we can now  offer interpretations of $CH_{X}^\bullet(A)$ in various contexts. 
First, in Section \ref{S:factor-alg}, we use these properties to give an interpretation of Hochschild chains over spaces of a CDGA $A$ as a factorization algebra in any dimension. The concept of \emph{factorization algebras} (see~\cite{CG,Co}) is inspired by Topological Quantum Field Theory, in which they appear naturally to encode observables. They were inspired by the work of Beilinson and Drinfeld~\cite{BD} (in an algebraic-geometry framework). Roughly speaking a factorization algebra $\mathcal{F}$ is a rule which (covariantly) associate cochain complexes to open subsets   of   a space $X$  together with multiplications $$\mathcal{F}(U_1)\otimes \cdots \otimes \mathcal{F}(U_n)\to \mathcal{F}(V)$$ for any family of pairwise disjoint open subsets of an open set $V$ in $X$. It should satisfy a \lq\lq{}cosheaf-like\rq\rq{} condition, meaning that $\mathcal{F}(V)$ can be computed by \v{C}ech complexes indexed on nice enough covers, called factorizing covers, see~\cite{CG} and Section~\ref{S:Factorization}. The (derived) global sections of a factorization algebra $\mathcal{F}$ is also called the factorization homology of $\mathcal{F}$  and is denoted $HF(\mathcal{F},X)$.

In this context we prove that the higher Hochschild chain functor defines a commutative factorization algebra $\mathcal{CH}_X(A)$, if $X$ admits a good cover whose factorization homology is precisely the derived Hochschild chains $CH_{X}^{\com}(A)$.
\setcounter{section}{4} \setcounter{subsection}{2} \setcounter{theorem}{3}
\begin{theorem}
Let $X$ be a topological space with a factorizing good cover and $A$ be a CDGA.
 Assume further that there is a basis of open sets in $X$ which is also a factorizing good cover. 
Then the assignment $\mathcal{CH}_X: U\mapsto CH_U^\com(A)$ is a factorization algebra on $X$. 
 \end{theorem}
In particular, this applies when $X$ is a manifold. Further, we prove that any factorization algebra for which $\mathcal{F}(U)$ (for contractible $U$) is naturally equivalent to a CDGA $A$ is canonically equivalent to $\mathcal{CH}_X(A)$.
\setcounter{corollary}{9}
\begin{corollary}
Let $X$ be a topological space with a sufficiently nice cover, let $A$ be a CDGA, and let $\mathcal{F}$ be a strongly constant factorization algebra on $X$ of type $A$. Then there is a natural equivalence of factorization algebras $\mathcal{F} \cong \mathcal{CH}_X(A)$.\\
In particular, there is a natural equivalence $HF(\mathcal{F})\cong CH_X^\com(A)$ in $\hkmod$.
\end{corollary}

In Section \ref{S:TCH}, we establish a relationship  between the \emph{topological chiral homology} functor defined by Lurie~\cite{L-TFT,L-VI} and both  the higher Hochschild functor and factorization algebras.  
 To obtain a comparison 
between these functors, it is important to note that they are defined in two different setting
 with a common intersection. Topological chiral homology, denoted  $\int_M A$, is 
defined  for any $E_n$-algebra $A$ (where $E_n$ is an operad equivalent to the little cubes in 
dimension $n$) and an $m$-dimensional manifold $M$, $m\leq n$,  such that $M\times D^{n-m}$ is
 framed (we say $M$ is $n$-framed).  Further $\int_M A$ is an $E_{n-m}$-algebra which is also a
 module over the $E_{n-m+1}$-algebra $\int_{\partial M} A$.  Topological chiral homology can be 
interpreted as an invariant of framed manifolds produced by an extended $(\infty,n)$-Topological 
Field Theory in the sense of~\cite{L-TFT}; the theory in question takes values in an 
$(\infty,n)$-category of $E_n$-algebras whose $n$-morphisms are (homotopy types) of chain complexes.
 Note that topological chiral homology depends on and comes with a choice of a sequence of maps of
 operads,
\begin{equation*}
 \xymatrix{E_1 \ar@{^{(}->}[r]&  E_2\ar@{^{(}->}[r] & \dots \dots \ar@{^{(}->}[r]& E_n\ar@{^{(}->}[r] & \dots \dots \ar@{->>}[r] &  \mathop{Com}  }
 \end{equation*}
which allows one to interpret a CDGA as an $E_n$-algebra for any $n$. When $A$ is a CDGA, things simplify greatly, and we can give a simple description of $\int_M A$ in terms of the higher Hochschild complex of $A$. Using excision for topological chiral homology, see Proposition~\ref{P:TCHpushout}, we prove, in a rather geometric way:
\setcounter{section}{5} \setcounter{subsection}{2} \setcounter{theorem}{4}
\begin{theorem}
Let $M$ be a manifold  endowed with a framing of $M \times D^k$ and $A$ be a differential graded commutative algebra viewed as an $E_{m+k}$-algebra. Then topological chiral homology of $M$ with coefficients in $A$, denoted by $\int_M A$ is equivalent  to $CH^\com_M(A)$ viewed as an $E_k$-algebra.
\end{theorem}
In other words, topological chiral homology  and  higher Hochschild chains coincide on their 
common intersection for an $n$-framed manifold $M$, and a CDGA $A$. As an immediate corollary,  in that case,  $\int_M A$ is independent of the $n$-framing, see \S~\ref{SS:applications}.  

In Section~\ref{S:blob}, we explain briefly the relationship between topological chiral homology and the \emph{blob complex}~\cite{MW}. More precisely, topological chiral homology extends to all manifolds provided that it is applied to an $\mathbb{E}_n^{O(n)}$-algebra, \emph{i.e.} an algebra over the semi-direct product of the $E_n$-operad with the orthogonal group $O(n)$ or said otherwise an $E_n$-algebra homotopically $O(n)$-invariant. The blob complex can also be defined for such algebras and any closed manifold and agrees with topological chiral homology in that case, see Proposition~\ref{P:blob}.

The relation between factorization algebras, derived higher Hochschild chains and topological chiral
 homology for CDGAs can be pushed further. Indeed, the data of an $E_n$-algebra are equivalent to 
those of a locally constant factorization algebra in $\R^n$~\cite{Co,L-VI}, see Proposition~\ref{P:Fac=En}. Further, 
the assumptions of having an $E_n$-algebra and a framed manifold to define topological chiral 
homology  can be replaced by the one of having a suitable (kind of) cosheaf of $E_n$-algebras  on 
an $n$-dimensional manifold $N$. Such a cosheaf is called an $\mathbb{E}_{N}^{\otimes}$-algebra~\cite{L-VI} and is also inspired by the work of Beilinson-Drinfeld~\cite{BD}. The techniques
 developed to compare Hochschild chains with factorization algebras and topological chiral homology
 leads to  
Theorem~\ref{T:HF=TCH} which can be rephrased as 
\setcounter{section}{5} \setcounter{subsection}{4} \setcounter{theorem}{5}
\begin{theorem} Let $M$ be a manifold of dimension $n$.
\begin{enumerate}
\item Topological chiral homology defines a natural $(\infty,1)$-functor $\mathcal{TC}_M$ from the category of $\mathbb{E}_{M\times \R^d}^{\otimes}$-algebras to the category of locally constant factorization algebras on $M$ with value in $E_d$-algebras, such that $\int_M A \cong HF(\mathcal{TC}_M,M)$.
 \item The functor $\mathcal{TC}_M(A)$ is an equivalence.
\end{enumerate}
 \end{theorem}
 Finally in Section~\ref{SS:applications}, we derive some applications of our results to give an interpretation of topological chiral homology in terms of maping spaces and to prove that topological chiral homology satisfies the exponential law, \emph{i.e.},    if $M$ and $N$ are manifolds and $\mathcal{A}$ is an $\mathbb{E}_d[M\times N]$-algebra, then,  there is an equivalence of $E_d$-algebras
$$\int_{M\times N} \mathcal{A} \; \cong \; \int_M\Big(\int_N \mathcal{A}\Big) $$see Corollary~\ref{C:FubiniTCH}. 

\smallskip
 
 Let us outline the philosophy intertwining the different concepts studied here. 
Given an $n$-framed manifold $M$ of dimension $m$ (\emph{i.e.} $M\times \R^{n-m}$ is framed),   
and an $E_n$-algebra $A$, we can form the topological chiral homology $\int_M A$ 
(or equivalently consider factorization algebra homology), which can be thought of as a colimit of tensor products of 
$A$ indexed by balls in the manifold. Now, if we embed $M\times \R^{n-m}$ in $M\times \R^{n-m+1}$ equipped 
with the induced framing, one can form $\int_M B$ for an $E_{n+1}$-algebra. But two different framings of 
$M\times \R^{n-m}$ may become equivalent after the embedding. Since a CDGA $C$ is an $E_k$-algebra
 (as well as an $\mathbb{E}_{M}^{\otimes}$-algebra) for all $k$, $\int_M C$ should not be able to distinguish 
different framings. Since manifolds embed in euclidean spaces, we further see that $\int_M C$ should makes sense for 
\emph{any} manifold.  Note that constant factorization algebra can be pulled back along open immersions and pushed 
forward  any map. This hints that any deformation retract of a manifold  should
 also have a well defined topological chiral homology (with value in a $C$) equivalent to the one of the manifold. 
All of this suggests that, for CDGAs, topological chiral homology may be extended to any CW-complex and is a 
homotopy invariant, which is precisely realized by the derived higher Hochschild functor. Said otherwise, 
higher Hochschild is the \lq\lq{}limit\rq\rq{} for $n$ going to $\infty$ of topological chiral homology defined 
as an invariant of manifolds of dimension $n$.

One of the emerging pattern here is that there is a balance to keep in between the manifolds and the algebraic 
structure needed to produce a (derived) invariant. For instance, in order to consider $E_n$-algebras, 
one need to consider  only  at most $n$-dimensional manifolds (possibly with extra structure such as a framing). 
In particular, working with only associative algebras restricts attention to manifolds of dimension $1$. 
At the opposite side of the spectrum, restricting to CDGAs allows to build and study \emph{explicit} 
examples in a much easier way and to compute them when adding the usual Rational Homotopy techniques 
to the axiomatic properties satisfied by the theory. 

\smallskip

We choose to work with commutative differential graded algebras since we are mainly interested in
 the characteristic zero case. However, it is also possible to work with simplicial commutative algebras,
 and all our results should makes sense in this setting. Simplicial commutative algebras are better behaved 
if one wants to deal with positive characteristic.

\begin{ack*}
We would like to thank the referees, David Ayala, Kevin Costello and Owen Gwilliam for many useful discussions and
 comments. The first author would like to thank the Einstein Chair at CUNY for their invitation, 
and the second and third  would like to thank IH\'ES for inviting them. 
The second and third authors were partially supported by the NSF grant DMS-0757245, and by the Max-Planck 
Institute for Mathematics in Bonn, Germany.
\end{ack*}

\setcounter{section}{1} \setcounter{theorem}{0} \setcounter{corollary}{0}

\section{Preliminary definitions and notation}

In this section we recall some standard definitions and constructions.

\textbf{Conventions:}
\begin{enumerate} 
\item We fix a ground field $k$ of characteristic zero.  The $(\infty,1)$-category of differential graded $k$-modules (\emph{i.e.} complexes) will be denoted $\hkmod$.
\item The (na\"ive) categories of simplicial sets and of commutative differential graded algebras  will be respectively denoted by $sSet$ and $CDGA$. The category of commutative graded algebra will be denoted $CGA$. Unless otherwise stated, all algebras will be assumed to be \emph{unital}.
\item We will simply refer to commutative differential graded algebras as CDGAs.
\item The $(\infty,1)$-categories of simplicial sets and commutative differential graded algebras  will be respectively denoted by $\hsset$ and $\hcdga$. 
\item Let $n\geq 1$ be an integer. By an $E_n$-algebra we mean an algebra over an $E_n$-operad. Unless otherwise stated, we work in the context of operads of differential graded  $k$-modules or $\infty$-operads in $\hkmod$. We will write $E_n\text{-}Alg_\infty$ for the $(\infty,1)$-category of $E_n$-algebras.
\item We work with a cohomological grading (unless otherwise stated)  for all our (co)homology groups and graded spaces, even when we use subscripts to denote the grading. In particular, all differentials are of degree $+1$, of the form $d:A^i\to A^{i+1}$ and the homology groups $H_i(X)$ of a space $X$ are concentrated in non-positive degree.
\item  \label{convention7}
We will denote by $CH_{X_\com}^n(A)$ the \emph{Hochschild chain complex} over $X_\com$ with value in $A$ of \emph{total} degree $n$. This Hochschild chain complex was noted differently in the papers~\cite{G,GTZ}. We choose this notation in order to put emphasis on the \emph{covariance} of the Hochschild chain functor with respect to $X_\com$ and the fact that we are considering cohomological degree.
\end{enumerate}

\subsection{Simplicial sets}
\label{S:Delta}
Denote by $\Delta$ the category whose objects are the ordered sets $[k]=\{0,1,\dots,k\}$, and morphisms $f:[k]\to [l]$ are non-decreasing maps $f(i)\geq f(j)$ for $i>j$. In particular, we have the morphisms $\delta_i:[k-1]\to[k], i=0,\dots, k$, which are injections that miss $i$ and we have surjections $\sigma_j:[k+1]\to [k], i=0,\dots,k$, which send $j$ and $j+1$ to $j$.

A \emph{simplicial set} is by definition a contravariant functor from $\Delta$ to the category of  sets $\Sets$ or written as a formula, $Y_\com:\Delta^{op}\to\Sets$. Denote by $Y_k=Y_\com([k])$, and call its elements simplicies. The image of $\delta_i$ under $Y_\com$ is denoted by $d_i:=Y_\com(\delta_i):Y_{k}\to Y_{k-1}$, for $i=0,\dots,k$, and is called the $i$th face. Similarly, $s_i:=Y_\com(\sigma_i):Y_{k}\to Y_{k+1}$, for $i=0,\dots,k$, is called the $i$th degeneracy. An element in $Y_k$ is called a degenerate simplex, if it is in the image of some $s_i$, otherwise it is called non-degenerate.

A simplicial set is said to be \emph{finite} if $Y_k$ is finite for every object $[k]\in \Delta$.  A \emph{pointed} simplicial set 
is a contravariant functor into the category $\Sets_*$ of pointed finite sets, $Y_\com:\Delta^{op}\to \Sets_*$. In particular, each $Y_k=Y_\com([k])$ has a preferred element called the base point, and all differentials $d_i$ and degeneracies $s_i$ preserve this base point.


A morphism of (finite or not, pointed or not) simplicial sets is a natural transformation of functors $f_\com:X_\com\to Y_\com$. Thus  $f_\com$ is given by a sequence of maps $f_k:X_k\to Y_k$ (preserving the base point in the pointed case), which commute with the faces $f_k d_i = d_i f_{k+1}$, and degeneracies $f_{k+1} s_i=s_i f_{k}$ for all $k\geq 0$ and $i=0,\dots, k$.

\smallskip

One of the most important construction for us is the pushout.
\begin{definition}\label{D:wedge}
Let $X_\com, Y_\com$, and $Z_\com$ be simplicial sets, and let $f_\com:Z_\com\to X_\com$ and $g_\com:Z_\com\to Y_\com$ be maps of simplicial sets. We define the wedge $W_\com=X_\com\cup_{Z_\com} Y_\com$ of $X_\com$ and $Y_\com$ along $Z_\com$ as the simplicial space given by $W_k= (X_k \cup Y_k)/\sim$, where $\sim$ identifies $f_k(z)=g_k(z)$ for all $z\in Z_k$. The face maps are defined as $d^{W_\com}_i(x)=d^{X_\com}_i(x), d^{W_\com}_i(y)=d^{Y_\com}_i(y)$ and the degeneracies are $s^{W_\com}_i(x)=s^{X_\com}_i(x), s^{W_\com}_i(y)=s^{Y_\com}_i(y)$ for any $x\in X_k \hookrightarrow W_k$ and $y\in Y_k \hookrightarrow W_k$. It is clear that $W_\com$ is well-defined and there are simplicial maps $X_\com\stackrel{i_\com}\to W_\com$ and $Y_\com\stackrel{j_\com}\to W_\com$.

If $X_\com$ is a pointed simplicial set, then we can make $W_\com$ into a pointed simplicial set by declaring the base point to be the one induced from the inclusion $X_\com\to W_\com$. (Note that this is in particular the case, when $X_\com, Y_\com, Z_\com, f_\com$ and $g_\com$ are in the pointed setting.)
\end{definition}

\subsection{Commutative Differential Graded Algebras} \label{S:CDGA}
We let $\cdga$ be the category of commutative differential graded algebras (over the characteristic zero field $k$). We do not assume the underlying chain complexes of our algebras to be bounded since in practice, it happens that one has to consider the Hochschild chains of de Rham forms on a space, which is generally $\Z$-graded. We follow the approach of~\cite[Chapter 1.1]{ToVe} and~\cite{Hi} for the model category properties of $\cdga$ and modules over  CDGAs. Recall from~\cite[Section 2.3]{Ho}, that there is a standard cofibrantly generated closed model category structure on the category of unbounded chain complexes for which fibrations are epimorphisms and (weak) equivalences are quasi-isomorphisms. It is further a symmetric monoidal model category with respect to the tensor products of chain complexes.

Since we work in characteristic zero, there is a standard closed model category structure on $\cdga$~\cite[Theorem 4.1.1]{Hi} as well, for which fibrations are epimorphisms and (weak) equivalences are quasi-isomorphisms (of CDGAs). 
The category $\cdga$ also has a monoidal structure given by the tensor product (over the ground field $k$) of differential graded commutative algebras, which makes $\cdga$ a symmetric monoidal model category. Note that since $k$ is assumed to be a field, this monoidal structure is given by an exact bifunctor.

Also note that $\cdga$ is simplicially enriched. Indeed, given $A,B \in \cdga$, we can form $\text{Map}_{\cdga}(A,B)$ the simplicial set of maps $[n]\mapsto \mathop{Hom}_{\cdga}(A, B\otimes \Omega^*(\Delta^n))$ (where $\Omega^*(\Delta^n)$ is the CDGA of forms on the $n$-dimensional standard simplex). 

 For any CDGA $A$, one can consider its category of \emph{differential graded (left) modules}, that we will denote by $A\textit{-}Mod$.  Again it has a natural model category structure with fibrations being epimorphisms and weak equivalences being quasi-isomorphisms. Further all assumptions in~\cite[Chapter 1.1]{ToVe} are satisfied. In particular, the tensor product of $A$-modules makes $A\textit{-}Mod$ a symmetric monoidal model category (in the sense of~\cite{Ho}) such that the functor $M\otimes_A -$ preserves weak equivalences when $M$ is cofibrant. Moreover, for any CDGA $A$, the category $A-\cdga$ of differential graded commutative $A$-algebra, in other words commutative monoid objects in $A\textit{-}Mod$, has a natural structure of  proper model category such that, for any cofibrant $A$-algebra $B$, the base change functor $B\otimes_A -: A\textit{-}Mod \to B\textit{-}Mod$ preserves weak equivalences~\cite[Chapter 1.1]{ToVe}.

\def \Top {Top}
\def \hTop {Top_{\infty}}
\subsection{Dwyer-Kan localization and $(\infty,1)$-categories}\label{S:DKL}
The $(\infty,1)$-categories that we are concerned about in this paper arise from model categories structures via the Dwyer-Kan localization turning them into simplicial categories.  Indeed simplicial categories are model for $(\infty,1)$-categories~\cite{Be1}. We now explain briefly how one gets $(\infty,1)$-categories out of model categories such as those considered in Section~\ref{S:Delta},~\ref{S:CDGA} above.

Following~\cite{Re,L-TFT}, by an $(\infty,1)$-category we mean a \textit{complete Segal space}. Rezk has shown that the category of simplicial spaces has a (simplicial closed) model structure, denoted $\CSS$  such that a complete Segal space is precisely a fibrant object for this model structure~\cite[Theorem 7.2]{Re}.  
Note that there is also a (simplicial closed) model category structure, denoted $\SeSp$, on the category of simplicial spaces such that a fibrant object in the $\SeSp$ structure is precisely a Segal space. We let $\mathbb{R}:\SeSp \to \SeSp$ be a fibrant replacement functor. Rezk~\cite{Re} has defined a completion functor $X_\bullet \to \widehat{X_\bullet}$ which, to a Segal space, associates an equivalent complete Segal space. Thus, the composition $X_\bullet \mapsto \widehat{\mathbb{R}(X_\bullet)}$ gives  a (fibrant replacement in the model category $\CSS$) functor $L_{\CSS}$ from simplicial spaces to complete Segal spaces. 

It remains to explain how to go from a model category to a simplicial space. The standard key idea is to use Dwyer-Kan localization. Let $\mathcal{M}$ be a model category and $\mathcal{W}$ be its subcategory of weak-equivalences. We denote $L^H(\mathcal{M},\mathcal{W})$ its \emph{hammock localization}, see \cite{DK}. One of the main property of $L^H(\mathcal{M},\mathcal{W})$ is that it is a simplicial category and that the (usual) category $\pi_0(L^H(\mathcal{M},\mathcal{W}))$ is the homotopy category of $\mathcal{M}$. Further,   every weak equivalence has a (weak) inverse in $L^H(\mathcal{M},\mathcal{W})$. When $\mathcal{M}$ is  further a simplicial model category, then for every pair $(x,y)$ of objects $\mathop{Hom}_{L^H(\mathcal{M},\mathcal{W})}(x,y)$ is naturally homotopy equivalent to the derived mapping space $ \mathbb{R}Hom (x,y)$. 

It follows that any model category $\mathcal{M}$ gives functorially rise to the simplicial category
$L^H(\mathcal{M},\mathcal{W})$. Taking the nerve $N_\bullet(L^H(\mathcal{M},\mathcal{W}))$ we obtain a simplicial space. Composing with the complete Segal Space replacement functor we get a functor $\mathcal{M}\to L_\infty(\mathcal{M}):= L_{\CSS}(N_\bullet(L^H(\mathcal{M},\mathcal{W})))$ from model categories to $(\infty,1)$-categories (that is complete Segal spaces).

\begin{example} \label{E:hsset}Applying the above procedure to the model category of simplicial sets $\sset$, we obtain the $(\infty,1)$-category $\hsset$. Similarly from the model category $\cdga$ of CDGAs we obtain the $(\infty,1)$-category $\hcdga$.
 Note that a simplicial set is determined by its $(\infty, 0)$ path groupoid and therefore the category of simplicial sets should be thought of as the $(\infty, 1)$ category of all $(\infty, 0)$ groupoids. Further,  the tensor product (over $k$) of algebras is a monoidal functor which gives  $\cdga$ a structure of monoidal model category, see~\cite{Ho}. Thus $\hcdga$ also inherits the structure of a symmetric monoidal $(\infty,1)$-category in the sense of~\cite{Re,L-TFT}. Similarly, the  disjoint union of simplicial sets endows  $\sset$ and $\hsset$ with symmetric monoidal structures. 

The model category of topological spaces yields the $(\infty,1)$-category $\hTop$. Since $\sset$ and $\Top$ are Quillen equivalent~\cite{GoJa,Ho}, the associated $(\infty,1)$-categories are equivalent (as $(\infty,1)$-categories): $\hsset \stackrel{\sim}{\underset{\sim}{\rightleftarrows}} \hTop$, where the left and right equivalences are respectively induced by the singular set and geometric realization functors.

One can also consider the pointed versions ${\hsset}_*$ and ${\hTop}_*$ of the above $(\infty,1)$-categories (using the model categories of these pointed versions~\cite{Ho}).
\end{example}

\begin{example} \label{E:Amod} As recalled in Section~\ref{S:CDGA}, there are model categories categories 
$A\textit{-}Mod$ and $A\text{-}\cdga$ of modules and commutative algebras over a CDGA $A$. Thus the above procedure 
gives us $(\infty,1)$-categories $A\textit{-}Mod_\infty$ and $A\textit{-}\cdga_\infty$ and the base change functor 
lifts to an $(\infty,1)$-functor. Further, if $f: A\to B$ is a weak equivalence, the natural 
functor $f_*:B\textit{-}Mod \to A\textit{-}Mod$ induces an equivalence $B\textit{-}Mod_\infty \stackrel{\sim}\to A\textit{-}Mod_\infty$ 
of $(\infty,1)$-categories since it is a Quillen equivalence. 

Moreover, if $f:A\to B$ is a morphism of CDGAs, it induces a natural functor 
$f^*:A\textit{-}Mod \to B\textit{-}Mod, M\mapsto M\otimes_A B$, which is an equivalence of $(\infty,1)$-categories 
when $f$ is a quasi-isomorphism, and is a (weak) inverse of $f_*$ (see~\cite{ToVe} or~\cite{KM}). 
Here we also (abusively) denote $f^*:A\textit{-}Mod_\infty \to B\textit{-}Mod_\infty$ and $f_*:B\textit{-}Mod_\infty \to A\textit{-}Mod_\infty$ the 
(derived) functors of $(\infty,1)$-categories induced by $f$.
 Since we are working over a field of characteristic zero, the same results applies to monoids in
 $A\textit{-}Mod$ and $B\textit{-}Mod$, that is to the categories $A\text{-}\hcdga$ and $B\text{-}\hcdga$.
Also note that, if  $f:A\to B$, $g:A\to C$ are CDGAs homomorphisms, we can form 
the (homotopy) pushout $D\cong B\otimes^{\mathbb{L}}_{A} C$. 

 Let us denote $p: B\to D$ and $q:C\to D$ 
the natural CDGAs maps. We thus get  two natural base change ($(\infty,1)$-)functors 
$C\textit{-}Mod_\infty \underset{p_*\circ q^*}{\stackrel{f^*\circ g_*}\rightrightarrows} B\textit{-}Mod_\infty$. Given any $M\in C\textit{-}Mod$,
the natural map $ f^*\circ g_*(M) \to p_*\circ q^*(M)$ is an 
equivalence~\cite[Proposition 1.1.0.8]{ToVe}.  

 The $(\infty,1)$-category $\hcdga$  is \emph{tensored} over $\hsset$ (and thus $\hTop$ as well) as follows from~\cite[\S 4.4.4]{Lu11}. We refer to~\cite{Lu11, L-VI} for tensored $(\infty,1)$-categories (which is the obvious analogue of the classical notion of tensored categories over $\Top$ or $\sset$); we simply recall that an $(\infty,1)$-category $\mathcal{C}$ is tensored over $\hsset$ if there exists an $(\infty,1)$-functor $\mathcal{C} \times \hsset \to \mathcal{C}$, denoted $(C, X_\bullet)\mapsto C\boxtimes X_\bullet$, together with natural equivalences
 $$Map_{\mathcal{C}}\big(C\boxtimes X_\bullet, D\big) \; \cong \; Map_{\hsset}\big(X_\com, Map_{\mathcal{C}}\big(C, D\big)\big). $$
 In fact, the tensor $A\boxtimes X_\bullet$ is precisely realized by the Hochschild derived functor, see Theorem~\ref{T:derivedfunctor}.

\end{example}

\begin{example}\label{E:EnOperad}
 We denote $E_n\textit{-}Alg_\infty$ the $(\infty,1)$-category of $E_n$-algebras which is given by algebras over any  $E_n$-($\infty$-)operads as introduced in~\cite[Section 5.1]{L-VI} in the symmetric monoidal ($(\infty,1)$-)category $(\hkmod,\otimes)$. It is equivalent to the $(\infty,1)$-category associated to model categories (deduced for instance from~\cite[Theorem 4.1.1]{Hi}) of algebras over the usual operad of singular chains on the little $n$-dimensional disk operad or as algebras over the Barratt-Eccles operad (which is an Hopf operad)~\cite{BF}.  
\end{example}

\begin{remark} \label{R:non-unique} There are other functors that yields a complete Segal space out 
of a model category. For instance, one can use the classification diagram of Rezk~\cite{Re}.
 Let again $\mathcal{M}$ be a model category and, for any integer $n$, let $\mathcal{M}^{[n]}$ be 
the (model) category of $n$-composables morphisms, that is the category fo functors from the poset 
$[n]$ to $\mathcal{M}$. The \emph{classification diagram} of $\mathcal{M}$ is the simplicial space
 $n\mapsto N_\bullet(\mathcal{W}e(\mathcal{M}^{[n]}))$ where $\mathcal{W}e(\mathcal{M}^{[n]})$ is 
the subcategory of weak equivalences of $\mathcal{M}^{[n]}$. Then taking a \emph{Reedy} fibrant 
replacement yields another complete Segal space
 $N_\bullet(\mathcal{W}e(\mathcal{M}^{[n]}))^{f}$~\cite[Theorem 6.2]{Be2}, \cite[Theorem 8.3]{Re}.
 It is known that the Segal space $N_\bullet(\mathcal{W}e(\mathcal{M}^{[n]}))^{f}$ is equivalent 
to $L_{\CSS}(N_\bullet(L^H(\mathcal{M},\mathcal{W})))$~\cite{Be2}.

\smallskip

More generally, there are several model for $(\infty,1)$-categories and several equivalent ways to obtain an $(\infty,1)$-category out of a ``homotopy theory''. We believe the results of this paper can easily be applied to the favorite model of the reader.
 
\end{remark}

\section{Derived higher Hochschild functor}

\subsection{Naive axiomatic approach to higher Hochschild  homology}
\label{S:NaiveHH}
We first recall the standard construction of  chain complexes computing higher Hochschild homology (also called Hochschild homology over spaces) following~\cite{P,GTZ}. The higher Hochschild complex is a functor $CH: sSet \times CDGA \to CDGA$. This functor is defined as follows:
the tensor products $A\otimes B$ of two CDGAs has an natural structure of \cdga (in other words, $\cdga$ has a symmetric monoidal structure canonically induced by the underlying tensor product of chain complexes). Furthermore,  the multiplication $A\otimes A\to A$ is an algebra homomorphism since $A$ is commutative. It follows that  a $\cdga$ can be thought of a strict symmetric monoidal functor from the category of finite sets with disjoint union to the category of chain complexes (whose value on a finite set $J$ is given by $A^{\otimes J}$), which can be extended to the category of all sets by taking colimits. Given a simplicial set $X_\com$, thought of as a functor $\Delta^{op} \to \sset$, compose these two functors to obtain a simplicial complex $X_\com \mapsto A^{\otimes X_\com}$. The total complex (that is the geometric realization $ \big|A^{\otimes X_\com}\big| $) of this simplicial complex is, by definition, $CH^{\com}_{X_\com}(A,A)$. In more details, we get the following explicit definitions.
\begin{definition}\label{D:Hoch}
First let $Y_\com:\Delta^{op}\to \Sets_*$ be a finite pointed simplicial set, and for $k\geq 0$, we set $y_k:=Y_k -\{*\}$ to be the complement of the base point in $Y_k$. Furthermore, let $(A=\bigoplus_{i\in \Z} A^i, d, \com)$ be a differential graded, associative, commutative algebra, and $(M=\bigoplus_{i\in \Z}M^i,d_M)$ a differential graded module over $A$ (viewed as a symmetric bimodule). Then, the 
{\bf Hochschild chain complex of $A$ with values in $M$ over $Y_\com$} is defined as\footnote{We recall that we are using a cohomological type grading for our differential graded modules, see Convention (7) on page \pageref{convention7}, hence the upper index $n$. Note that, nowhere in this paper will we consider Hochschild cochains.} $CH^{\com}_{Y_\com}(A,M):=\bigoplus_{n\in \Z} CH^n_{Y_\com}(A,M)$, where $$ CH^n_{Y_\com}(A,M):=\bigoplus_{k\geq 0} (M\otimes A^{\otimes y_k})^{n+k} $$ is given by a sum of elements of total degree $n+k$. In order to define a differential $D$ on $CH_{Y_\com}^{\com}(A,M)$, we define morphisms $d_i:Y_k\to Y_{k-1}$, for  $i=0,\dots,k$ as follows. First note that for any map $f:Y_k\to Y_l$ of pointed sets, and for $m\otimes a_1\otimes \dots\otimes a_{y_k}\in M\otimes A^{\otimes y_k}$, we denote by $f_*:M\otimes A^{\otimes y_k}\to M\otimes A^{\otimes y_l}$, 
\begin{equation}\label{f_*}
 f_*(m\otimes a_1\otimes \dots\otimes a_{y_k})=(-1)^{\epsilon} n\otimes b_1\otimes \dots\otimes b_{y_l},
 \end{equation}
  where $b_{j}=\prod_{i\in f^{-1}(j)} a_i$ (or $b_j=1$ if $f^{-1}(j)=\emptyset$) for $j=0,\dots,y_{l}$, and $n=m\com \prod_{i\in f^{-1}(\text{basepoint}), i\neq \text{basepoint}}a_i$. The sign $\epsilon$ in equation \eqref{f_*} is determined by the usual Koszul sign rule of $(-1)^{|x|\com |y|}$ whenever $x$ moves across $y$. In particular, there are induced boundaries $(d_i)_*:CH^k_{Y_\com}(A,M)\to CH^{k-1}_{Y_\com}(A,M)$ and degeneracies $(s_j)_*:CH^k_{Y_\com}(A,M)\to CH^{k+1}_{Y_\com}(A,M)$, which we denote by abuse of notation again by $d_i$ and $s_j$. Using these, the differential $D:CH^{\com}_{Y_\com}(A,A)\to CH^{\com}_{Y_\com}(A,A)$ is defined by letting $D(a_0\otimes a_1\otimes \dots\otimes a_{y_k})$ be equal to 
\begin{equation*}
\sum_{i=0}^{y_k} (-1)^{k+\epsilon_i} a_0\otimes \dots\otimes d(a_i)\otimes \dots\otimes a_{y_k}+\sum_{i=0}^k (-1)^i d_i (a_0\otimes \dots\otimes a_{y_k}),
\end{equation*}
where $\epsilon_i$ is again given by the Koszul sign rule, \emph{i.e.}, $(-1)^{\epsilon_i}=(-1)^{|a_0|+\cdots+|a_{i-1}|}$. The simplicial conditions on $d_i$ imply that $D^2=0$.

\smallskip

If $Y_\com:\Delta^{op}\to \Sets$ is a finite (not necessarily pointed) simplicial set, we may still define $CH_{Y_\com}^{\com}(A):=\bigoplus_{n\in \Z} CH^n_{Y_\com}(A,A)$ via the same formula as above, $CH^n_{Y_\com}(A,A):=\bigoplus_{k\geq 0} (A\otimes A^{\otimes y_k})_{n+k}$. Formula \eqref{f_*} again induces boundaries $d_i$ and degeneracies $s_i$, which produce a differential $D$ of square zero on $CH^{\com}_{Y_\com}(A,A)$ as above.

\smallskip

If $Y_\com$ is any simplicial set we define
$$CH_{Y_\com}^{\com}(A,M):=\colim_{\small \begin{array}{l}K_\com\to Y_\com, \\ K_\com \mbox{ finite} \end{array}} CH_{K_\com}^{\com}(A,M) $$
as the colimit over all finite simplicial sets. If $Y_\com$ is finite, then this definition agrees with the previous ones thanks to the Yoneda lemma.
\end{definition}
\begin{remark}
Note that due to our grading convention, if $A$ is non graded, or concentrated in degree 0, then $HH_\com^{Y_{\com}}(A,A)$ is concentrated in non positive degrees. In particular, our grading is  opposite of the one in~\cite{L}.
\end{remark}

Note that the equation~\eqref{f_*} also makes sense for any map of simplicial pointed sets $f:X_k\to Y_k$. 
Since $A$ is graded commutative and $M$ symmetric, $(f\circ g)_* =f_*\circ g_*$, hence $Y_\com \mapsto CH^\com_{Y_\com}(A,M)$ is a functor from the category of finite pointed simplicial sets to the category of simplicial $k$-vector spaces, see~\cite{P}. If $M=A$, $CH^\com_{Y_\com}(A)$ is a functor from the category of finite simplicial sets to the category of simplicial $k$-algebras. \smallskip

Now note that any map  $g:A\to B$ of CDGAs and any maps of modules $\rho:M\to N$  over $g:A\to B$, \emph{i.e.} $\rho(am)=g(a)\rho(m)$, induces a map $CH_{Y_\com}^{\com}(g,\rho): CH_{Y_\com}^{\com}(A,M)\to CH_{Y_\com}^{\com}(B,N)$ of simplicial vector spaces.


The chain complex $\left(CH_{Y_\com}^{\com}(A), D\right)$ inherits a structure of (differential graded) algebra. This is a formal consequence of the fact that $CH_{Y_\com}^{\com}(A)$ is a simplicial commutative algebra.
Indeed, given two simplicial vector spaces $V_\com$ and $W_\com$, one defines a simplicial structure on the simplicial space $(V\times W)_k:=V_k\otimes W_k$ using the boundaries $d^V_i\otimes d^W_i$ and degeneracies $s^V_i\otimes s^W_i$. The \emph{shuffle product} is (the collection of) maps  $sh:V_p\otimes W_q\to (V\times W)_{p+q}$ defined by $$ sh(v\otimes w)=\sum_{(\mu,\nu)} sgn(\mu,\nu) (s_{\nu_q}\dots s_{\nu_1}(v)\otimes s_{\mu_p}\dots s_{\mu_1}(w)), $$
where $(\mu,\nu)$ denotes a $(p,q)$-shuffle, \emph{i.e.} a permutation of $\{0,\dots,p+q-1\}$ mapping $0\leq j\leq p-1$ to $\mu_{j+1}$ and $p\leq j\leq p+q-1$ to $\nu_{j-p+1}$, such that $\mu_1<\dots<\mu_p$ and $\nu_1<\dots<\nu_q$.

Since $CH^{\com}_{Y_\com}(A,M)$ is a simplicial vector space, we obtain an induced shuffle map $sh:CH_{Y_p}^{\com}(A,M)\otimes CH^{\com}_{Y_q}(B,N) \to CH^{\com}_{Y_{p+q}}(A\otimes B, M\otimes N)$ for any CDGAs $A,B$ and modules $M,N$. Now, since $A$ is a CDGA, the multiplication $\mu:A\otimes A\to A$ is an algebra map, and the map $\nu:M\otimes A\to M$ a map of $A$-modules. Composing these maps with the shuffle products  we obtain the 
multiplication 
$$ sh_{Y_\com}:CH^{\com}_{Y_\com}(A,M)\otimes CH^{\com}_{Y_\com}(A)\stackrel {sh} \to CH^{\com}_{Y_\com}(A\otimes A,M\otimes A)\stackrel {CH_{Y_\com}^{\com}(\mu,\nu)} \longrightarrow  CH^{\com}_{Y_\com}(A,M).  $$

\begin{proposition}\label{P:shuffleinvariance}  The multiplication  $sh_{Y_\com}$ makes $CH^{\com}_{Y_\com}(A)$ a differential graded commutative algebra and $CH_{Y_\com}^{\com}(A,M)$ a DG-module over $CH^{\com}_{Y_\com}(A)$, which are natural in $A$ and $M$. 
\end{proposition}
\begin{proof} The proof of the algebra structure is given in~\cite[Proposition 2.4.2]{GTZ} and the proof of the module structure is the same.
\end{proof}

Note that $CH_{\pt}^{\com}(A)$ is the (chain complex associated to the) constant simplicial CDGA $A$. In particular there is a canonical quasi-isomorphism $\eta: A=CH_{\mathop{pt}_0}^{\com}(A)\to CH_{\pt}^{\com}(A)$ splitting the augmentation map $CH_{\pt}^{\com}(A)\to CH_{\mathop{pt}_0}^{\com}(A)$.    It follows from Proposition~\ref{P:shuffleinvariance} above that if $X_\com$ is a pointed simplicial set, the canonical map $\pt \to X_\com $
induces a natural $A$-module structure on $CH_{X_\com}^{\com}(A,M)$ (and an $A$-algebra structure on $CH_{X_\com}^{\com}(A)$). In other words, $CH_{X_\com}^\com (A,M)$ is naturally an $A$-module.

Summing up the previous discussion and proposition we obtain:
\begin{corollary}\label{C:Hfunctor} 
The rule $(Y_\com, A)\mapsto (CH^{\com}_{Y_\com}(A),D,sh_{Y_\com})$ is a functor $CH:\sset \times \cdga \to \cdga$. Similarly, the rule $(Y_\com, M)\mapsto (CH^{\com}_{Y_\com}(A,M),D,sh_{Y_\com})$ is a functor $CH:\sset_* \times A\textit{-}Mod \to A\textit{-}Mod$.
\end{corollary}

\begin{definition}\label{D:HH}
The Hochschild homology of a CDGA $A$  over a simplicial set $X_\com$ is the cohomology\footnote{Recall that we are using a cohomological type grading for our differential graded modules, see Convention (7) on page \pageref{convention7}.} $HH_{X_\com}^{\com}(A)=H^\bullet(CH_{X_\com}^{\com}(A),D)$ of the CDGA $(CH_{X_\com}^{\com}(A),D, sh)$ as a commutative graded algebra.

Further if $X_\bullet$ is pointed and $M$ is an $A$-module, the Hochschild homology of $A$ with value in $M$ over $X_\com$ is the homology $HH_{X_\com}^{\com}(A,M)=H^*(CH_{X_\com}^{\com}(A,M),D)$ as a graded module over  $HH_{X_\com}^{\com}(A)$.
\end{definition}

Now let $X$ be a \emph{topological space},  we define the Hochschild homology of a $CDGA$ $A$ over $X$ to be $HH^\com_{S_\com(X)}(A)$ where $$S_\com(X)=\mathop{Map}(\Delta^\bullet, X)$$ is the singular simplicial set of $X$. If $X$ is pointed, then $S_\com(X)$ is a pointed simplicial set and we define the Hochschild homology of a $CDGA$ $A$ with value in an $A$-module $M$ over $X$ to be $HH^\com_{S_\com(X)}(A,M)$.


The Hochschild chain functor satisfies the following properties which allows to build \emph{explicitly and easily} these chain complexes out of other simplicial sets and do computations (for instance, see~\cite{P,G,GTZ}). 


\begin{proposition}[Tensor Products of CDGAs and disjoint union of simplicial sets] \label{P:tensor}Let $A,B$ be two CDGAs. For any $X_\com\in \sset$, there is a canonical isomorphism
$$CH^{\com}_{X_\com}(A\otimes B) \cong CH^{\com}_{X_\com}(A)\otimes CH^{\com}_{X_\com}(B) $$ of CDGAs. Further for any simplicial set $Y_\com$, one has a natural isomorphism
$$ CH^{\com}_{X_\com \coprod Y_\com}(A) \cong CH^{\com}_{X_\com}(A)\otimes CH^{\com}_{Y_\com}(A)$$ of CDGAs and a natural isomorphism of modules $$ CH^{\com}_{X_\com \coprod Y_\com}(A,M) \cong CH^{\com}_{X_\com}(A,M)\otimes CH^{\com}_{Y_\com}(A)$$ if $X_\com$ is a pointed simplicial set.
\end{proposition}
\begin{proof}It follows from the canonical isomorphisms $(A\otimes B)^{\otimes n}\cong A^{\otimes n} \otimes B^{\otimes n}$ and $A^{\otimes n+m}\cong A^{\otimes n}\otimes A^{\otimes m}$.
\end{proof}

Recall that, by functoriality, if $f:Y_\com\to X_\com$ is a map of simplicial sets, then for any CDGA $A$, we have a map of algebra $f_*:CH_{Y_\com}^{\com}(A)\to CH_{X_\com}^{\com}(A)$ which exhibits the Hochschild complex  of $A$ over $X_\com$ as a module over the Hochschild complex  of $A$ over $Y_\com$.  
Let $Z\com \to X_\com$, $Z_\com \to Y_\com$ be two maps of simplicial sets and let $W_\com$ be a pushout $W_\com \cong X_\com \coprod_{Z_\com} Y_\com$.
\begin{proposition}\label{P:pushout} 
There is a natural map of simplicial modules \footnote{The tensor product in Proposition~\ref{P:pushout} is the tensor product of (simplicial) modules over the simplicial differential graded commutative algebra $CH_{Z_\com}^{\com}(A,A)$. Passing to the Hochschild chain complexes,  it induces a natural map of CDGAs and modules and yield a quasi-isomorphism if $Z_\com$ injects into either $ X_\com$ or $ Y_\com$, see~\cite[Corollary 2.4.3]{GTZ}.}
$$  CH_{X_\com}^{\com}(A,M)\otimes_{CH_{Z_\com}^{\com}(A,A)} CH_{Y_\com}^{\com}(A,A) \to CH_{W_\com}^{\com}(A,M) $$ which is a map of algebras if $M=A$ (with its natural module structure).
If $Z_\com$ injects into either $Z_\com\stackrel {f_\com}\to X_\com$ or $Z_\com\stackrel{g_\com}\to Y_\com$, then this map is in fact an isomorphism of $CH_{W_\com}^{\com}(A)$-modules.
\end{proposition}
\begin{proof} The proof~\cite[Lemma 2.1.6]{GTZ} given in the case $M=A$ applies to any module $M$.
\end{proof}



\begin{corollary}\label{C:Homologyfunctor}
The rule $(X_\com,A)\mapsto HH_{X_\com}^{\com}(A)$ is a functor 
$HH:\sset \times \cdga\to \cga$ which satisfies the following axioms
\begin{enumerate}
\item {\bf bimonoidal:} Hochschild homology is monoidal with respect to the monoidal structures given by the disjoint union of simplicial sets and tensor products of algebras. In other words, there are natural isomorphisms:
 $$HH_{X_\com \times Y_\com}^{\com}(A)\cong HH_{X_\com}^{\com}\!(A) \otimes HH_{Y_\com}^{\com}(A), \; HH_{X_\com}^{\com}(A\otimes B) \cong HH_{X_\com}^{\com}(A) \otimes HH_{X_\com}^{\com}(B). $$
\item {\bf homotopy invariance :} if $f:X_\com\to Y_\com$ and $g:A\to B$ are (weak) homotopy equivalences, then $HH(f,g):HH_{X_\com}^{\com}(A)\to HH_{Y_\com}^{\com}(B)$ is an isomorphism.
\item {\bf point} There is a natural isomorphism $HH_{pt}^{\com}(A)\cong A$   
\end{enumerate}
A similar statement holds with the category of topological spaces instead of simplicial sets, and with the pointed analogs of these categories (as in Corollary~\ref{C:Hfunctor}).
\end{corollary}
\begin{proof}
This follows from Proposition~\ref{P:pushout},  Proposition~\ref{P:tensor}, Corollary~\ref{C:Hfunctor} and Proposition~\ref{P:homologyinvariance} below.
\end{proof}
The axioms listed in the above proposition are \emph{not} enough to uniquely determine Hochschild homology as a functor. Indeed, we are missing an analog of the Excision/Mayer-Vietoris axioms in the classical Eilenberg-Steenrod axioms for singular homology. The analog of this axiom is similar to the \emph{locality axiom} of a Topological Field Theory.  In view of Proposition~\ref{P:pushout}, we wish to compute the Hochschild homology over  an union of two open sets as the tensor product of the Hochschild homology of each open tensored over the Hochschild homology over their intersection. This forces us to take \emph{derived} tensor products. Thus a better framework for an axiomatic description of Hochschild chains is given by considering derived categories or $(\infty,1)$-categories. We deal with this \emph{locality} axiom in Section~\ref{S:inftyfunctor} below. This axiom translates into an Eilenberg-Moore spectral sequence for Hochschild homology, see Corollary~\ref{C:HlocalitySpecSeq}.

\smallskip

A crucial property of Hochschild chains which allows to pass to homotopy categories, is the fact, proved by Pirashvili~\cite{P}, that the higher Hochschild chain complex is  invariant along \emph{quasi-isomorphisms} in both arguments.
\begin{proposition}\label{P:homologyinvariance}[Homotopy and homology invariance]   If $f: X_\com \to Y_\com$ is a map of simplicial sets inducing an isomorphism in homology $H_\com (X)\stackrel{\simeq}\to H_\com (Y)$ , then the map $CH_{X_\com}^\com (A, M) \to CH_{Y_\com}^\com	(A, M)$ is a quasi-isomorphism. 

\smallskip

Further if $h:A\to B$ is a quasi-isomorphism of CDGAs, then the induced map $h_\com : CH_{X_\com}^{\com}(A)\to CH_{X_\com}^{\com}(B)$ is a quasi-isomorphism of CDGAs.

\smallskip If $Z_\com$ is a pointed simplicial set and $M$ is a $B$-module, the induced map $h_\com : CH_{Z_\com}^{\com}(A,M)\to CH_{Z_\com}^{\com}(B,M)$ is a quasi-isomorphism of $ CH_{Z_\com}^{\com}(A)$-modules and if $\alpha: M\to N$ is a map of $B$-modules, the induced map $\alpha_\com: CH_{Z_\com}^{\com}(B,M)\to CH_{Z_\com}^{\com}(B,N)$ is a quasi-isomorphism of $CH_{Z_\com}^{\com}(B)$-modules.
\end{proposition}
\begin{proof} This is essentially due to Pirashvili~\cite{P}. Indeed, let $\Gamma$ be the category of finite sets, then the Hochschild chain complex $CH_{X_\com}^\com (A)$ is isomorphic to the tensor product $$k_{X_\com} \otimes_{\Gamma} \mathcal{L}(A)$$ of the left $\Gamma$-module $\mathcal{L}(A)$ and the simplicial right  $\Gamma$-module $k_{X_\com}$. Here the left $\Gamma$-module $\mathcal{L}(A)$ is defined  by $n\mapsto A^{\otimes n+1}$ and formula~\eqref{f_*}. The right $\Gamma$-module $k_{X_\com}$ is  defined by $n\mapsto \colim_{K_\com \mbox{ finite}} k\big[Hom_{\Gamma}\big([n], K_\com\big)]$ where $[n]$ is the finite set $\{0,\dots,n\}$, see~\cite{P}. A quasi-isomorphism of CDGAs induces a quasi-isomorphism of left $\Gamma$-module and similarly for a quasi-isomorphism of simplicial sets. Since each  right $\Gamma$-module $k_{X_m}$  ($m\in \N$) is a projective right $\Gamma$-module (see~\cite{P}), the tensor product $k_{X_\com} \otimes_{\Gamma} \mathcal{L}(A)$ is quasi-isomorphic to the derived tensor product $k_{X_\com} \otimes^{\mathbb{L}}_{\Gamma} \mathcal{L}(A)$. It follows that this complex is invariant along quasi-isomorphisms in both arguments ($X_\com$ and $A$).  
The proof in the case of pointed simplicial sets and modules is the same with the category $\Gamma$ replaced by the category of pointed finite sets.
\end{proof}

\subsection{Higher Hochschild as an $(\infty, 1)$-functor}\label{S:inftyfunctor}

In this section we deal with axioms for the theory of higher Hochschild \emph{chains} instead of mere homology. That is, we upgrade the previous section, in particular Corollary~\ref{C:Homologyfunctor},  to the setting of derived categories, or  more precisely $(\infty,1)$-categories.   Said otherwise,  we replace the category of simplicial sets by its homotopical analogue: the   $(\infty, 1)$-category of simplicial sets denoted $\hsset$ and we replace the category of commutative differential algebras by $\hcdga$, the $(\infty,1)$-category associated to the derived category of commutative differential graded algebras (obtained by inverting quasi-isomorphisms of CDGAs). 

In this settting, we will prove that the axioms  determine \emph{uniquely} the Hochschild chain as an $(\infty,1)$-functor (lifting Hochschild homology).  These axioms are \emph{not} specific to CDGA  but rather come from the fact that any  presentable $(\infty,1)$-category is (homotopically) canonically tensored over simplicial sets according to~\cite[Corollary 4.4.4.9]{Lu11}.

\begin{theorem}\label{T:derivedfunctor}
There is a canonical equivalence $CH_{X_\com}(A)\cong X_\com \boxtimes A$ between the Hochschild chains and the tensor of $A$ and $X_\com$, \emph{i.e.} there are natural equivalences (in $\hsset$)  
\begin{equation}
 \text{Map}_{\hcdga}\big(CH_{X_\com}(A),B\big) \; \cong \; 
 \text{Map}_{\hsset}\big(X_\com, \text{Map}_{\hcdga}(A,B) \big).
\end{equation}
In particular, 
the Hochschild chains lift as a functor of $(\infty,1)$-categories $CH: \hsset \times \hcdga \to \hcdga$ which satisfies the following axioms
\begin{enumerate}
\item {\bf value on a point:} there is a natural equivalence $CH^\com_{pt}(A)\cong A$ of CDGAs.
\item {\bf monoidal:} $CH$ is monoidal with respect to both variables. Precisely, there are natural equivalences
$$CH_{X_\com\coprod Y_\com}^{\com}(A)\cong CH_{X_\com}^{\com}(A) \otimes CH_{Y_\com}^{\com}(A), $$ $$
CH_{X_\com}^\com (A\otimes B)\cong CH_{X_\com}^{\com}(A) \otimes  CH_{X_\com}^{\com}(B).$$  
\item {\bf homotopy gluing/pushout:} $CH$ sends homotopy pushout in $\hsset$ to homotopy pushout in $\hcdga$. More precisely, given maps $Z_\com\stackrel{f}\to X_\com$ and $Z_\com\stackrel{g} \to Y_\com$  in $\hsset$, and $W_\com \cong X_\com \bigcup^{h}_{Z_\com} Y_\com$ a homotopy pushout, there is a natural equivalence
$$CH^{\com}_{W_\com}(A)\cong CH^{\com}_{X_\com}(A)\otimes_{CH^{\com}_{Z_\com}(A)}^{\mathbb{L}} CH^{\com}_{Y_\com}(A).$$   
\end{enumerate}
\end{theorem}
In particular Theorem~\ref{T:derivedfunctor} implies  similar statements involving the homotopy categories of simplicial sets and CDGAs instead of their $(\infty,1)$-refinement.
Axiom~{\bf (3)} is the \emph{locality} axiom (as for Topological Field Theories) playing the role of the excision/Mayer-Vietoris axiom for classical homology. 
\begin{proof} The second equivalence in the monoidal axiom follows from Proposition~\ref{P:tensor}. The rest  is an immediate consequence of~\cite[Corollary 4.4.4.9]{Lu11}, \S~\ref{S:NaiveHH} and the fact that the coproduct in $\hcdga$ is given by the tensor product of CDGAs. Note that the axioms can also be proved easily and directly using the results of Section~\ref{S:NaiveHH}, as we now demonstrate for the interested reader's convenience (this incidentally yields immediately another proof of the identification of derived Hochschild chains with a tensor). 
{We already know that the Hochschild chain complex defines a bifunctor $(X_\com,A)\mapsto CH_{X_\com}^{\com}(A)$ from the  category $\sset \times \cdga$ to $\cdga$, see Proposition~\ref{P:shuffleinvariance}. The $(\infty,1)$-categories $\hsset$ and $\hcdga$ both arise from suitables localizations of the weak equivalences (see Section~\ref{S:DKL}). That is $\hsset=L_{\CSS}(N_\bullet(L^H(\sset,\mathcal{W}_{\sset})))$ and $\hcdga=L_{\CSS}(N_\bullet(L^H(\cdga,\mathcal{W}_{\cdga})))$ where $W_{\sset}$ and $W_{\cdga}$ are the respective subcategories of weak equivalences in $\sset$ and $\cdga$ given in Sections~\ref{S:Delta} and~\ref{S:CDGA}. Thus, in order to show that the functor $(X_\com,A)\mapsto CH_{X_\com}^{\com}(A)$ lifts to the $(\infty,1)$-associated categories, it suffices to show that this bifunctor passes to the (hammock)  localization of these model categories along their subcategory of weak equivalences  $W_{\sset}$ and $W_{\cdga}$. By construction of the hammock localization functor $L^H$, it suffices to prove that the usual Hochschild chain complex bifunctor maps weak equivalences (in $\sset\times \cdga$) to weak equivalences (in  $\cdga$) which follows from Proposition~\ref{P:homologyinvariance}. More precisely, the bifunctor $CH: \hsset\times \hcdga \to \hcdga$ is the bifunctor induced by the bifunctor \begin{multline*} 
  L_{\CSS}\big(N_\bullet\big((CH_{-}^{\com}(-))\big)\big):L_{\CSS}(N_\bullet(L^H(\sset \times \cdga, W_{\sset}\times W_{\cdga})))\\ \longrightarrow \quad L_{\CSS}(N_\bullet(L^H(\cdga, W_{\cdga})))                                                                                                                   \end{multline*}
 obtained by localizing the Hochschild chains.

\smallskip

When $X_\com$ is contractible, Proposition~\ref{P:homologyinvariance} implies that  there  is an natural equivalence  $CH_{X_\com}^{\com}(A)\cong CH_{pt_\com}^{\com}(A)\cong A$ of CDGAs where $pt_\com$ is the standard simplicial model of a point, see~\cite[Example 2.3.4]{GTZ}. This proves the value on a point axiom.

\smallskip

The compatibility with the monoidal structures follows from Proposition~\ref{P:tensor} since both monoidal functors $\coprod: \sset \times \sset \to \sset$ and $\otimes_k: \cdga\times \cdga \to \cdga$ are Quillen exact, thus passes to the $(\infty,1)$-category induced by the model categories structures.

\smallskip

It remains to prove the homotopy gluing axiom. We let $f_\com: Z_\com\to X_\com$ and $g_\com: Z_\com\to Y_\com$ be simplicial sets maps. Since we already know that the Hochschild chain complex preserves weak equivalences, we can replace $W_\com$ by  the pushout $X_\com \bigcup_{Z_\com} \widehat{Y}_\com$ where $\hat{g}_\com:  Z_\com\to \widehat{Y}_\com$ is a cofibration and $p:\widehat{Y}\stackrel{\simeq}\to Y$ is a fibrant replacement of $Y$. 

\smallskip

By Proposition~\ref{P:pushout}, there is a natural isomorphism 
\begin{equation}\label{eq:homotopypushout}CH_{X_\com \bigcup_{Z_\com} \widehat{Y}_\com}^\com (A)\cong  CH^{\com}_{X_\com}(A)\otimes_{CH^{\com}_{Z_\com}(A)} CH^{\com}_{\widehat{Y}_\com}(A).\end{equation}

Since $Z_\com \stackrel{\hat{g}_\com}\hookrightarrow \widehat{Y}_\com$ is injective, the induced map $\hat{g}_*: CH^{\com}_{Z_\com}(A) \to CH^{\com}_{\widehat{Y}_\com}(A)$ of CDGAs exhibits $CH^{\com}_{\widehat{Y}_\com}(A)$ as a semi-free $CH^{\com}_{Z_\com}(A)$ algebra, thus the tensor product in the quasi-isomorphism~\eqref{eq:homotopypushout} can be replaced by a derived tensor product:
  \begin{equation}\label{eq:homotopypushout2}CH_{X_\com \bigcup_{Z_\com} \widehat{Y}_\com}^\com (A)\cong  CH^{\com}_{X_\com}(A)\otimes_{CH^{\com}_{Z_\com}(A)}^{\mathbb{L}} CH^{\com}_{\widehat{Y}_\com}(A).\end{equation}
By homotopy invariance, the map $p_*:CH^{\com}_{\widehat{Y}_\com}(A)\cong CH^{\com}_{Y_\com}(A)$ is a quasi-isomorphism and thus the homotopy pushout axiom follows from the quasi-isomorphism~\eqref{eq:homotopypushout2}.}
\end{proof}

\begin{remark}\label{R:derivedfunctorTop}
As in Corollary~\ref{C:Homologyfunctor},   the (model) categories of simplicial sets and  (compactly generated) topological spaces being Quillen equivalent,  one can replace   $\hsset$ by its topological counterpart $\hTop$ in Theorem~\ref{T:derivedfunctor}. This yield the Hochschild chain functor over spaces $CH: \hTop \times \hcdga \to \hcdga$, see \S~\ref{S:factor-alg}. By definition Hochschild chain are functorial in spaces; for instance any continuous map $f: X\to Y$ between topological spaces induces a chain map $f_\com:CH_X^\com(A) \to CH_Y^\com(A)$ which is an equivalence if $f$ is a weak-equivalence. We will get back to and focus on this Hochschild functor over topological spaces in Section~\ref{S:factor-alg}
\end{remark}

The locality axiom~{\bf (3)} in Theorem~\ref{T:derivedfunctor} yields an Eilenberg-Moore type spectral sequence for computing higher Hochschild homology.
\begin{corollary} \label{C:HlocalitySpecSeq} Given an homotopy pushout $W_\com \cong X_\com \cup_{Z_\com}^h Y_\com$, 
 there is a natural strongly convergent spectral sequence of cohomological type of the form
 $$E_2^{p,q}:= {Tor}_{p,q}^{HH_{Z_\com}^{\com}(A)}\left(HH_{X_\com}^{\com}(A),HH_{Y_\com}^{\com}(A)\right) \Longrightarrow HH_{W_\com}^{p+q}(A) $$
 where $q$ is the \emph{internal} grading.  The spectral sequence is furthermore a spectral sequence of differential $HH_{Z_\com}^{\com}(A)$-algebras.
\end{corollary}
Recall that we are considering a cohomological grading; in particular the spectral sequence is concentrated in the left half-plane with respect to this grading (and $p$ is negative). 
\begin{proof} 
By Theorem~\ref{T:derivedfunctor}, we have $CH^{\com}_{W_\com}(A)\cong CH^{\com}_{X_\com}(A)\mathop{\otimes}_{CH^{\com}_{Z_\com}(A)}^{\mathbb{L}} CH^{\com}_{Y_\com}(A)$ as CDGAs. Now the spectral sequence follows from  standard results on derived tensor products (of $CH^{\com}_{Z_\com}(A)$-modules)~\cite[Theorem 4.7]{KM}. That the spectral sequence is one of algebras comes from the fact, that we can choose a semi-free  resolution of $CH^{\com}_{X_\com}$ by a $CH^{\com}_{Z_\com}$-algebra.  
\end{proof}

\begin{example}
\begin{enumerate}
\item
It is a well-known fact, that the usual Hochschild complex of an associative algebra $A$ $CH_{\com}(A)=CH^{\com}_{S^1_\bullet} (A)$ may be written as a $Tor$ over the bimodule $A^e=A\otimes A^{op}$, see \emph{e.g.} \cite[Proposition 1.1.13]{L}. More explicitly,
$$ HH_\bullet(A)=Tor^{A^e}_{\com}(A,A). $$
Identifying $HH^{pt_\com}_\bullet(A)=A$, and $HH_{\{pt_\com,pt_\com^-\}}^\bullet(A)=A\otimes A^{op}=A^e$, where $pt_\com^-$ denotes the point with opposite orientation, we see that, in this case, the spectral sequence of Corollary \ref{C:HlocalitySpecSeq} collapses at the $E_2$ level,
 $$ {Tor}^{HH^{\{pt_\com,pt_\com^-\}_\com}_{\com}(A)}\left(HH_{pt_\com}^{\com}(A),HH_{pt_\com}^{\com}(A)\right) = HH_{S^1_\com}^{\com}(A) $$
 where we used that $S^1_\com \cong pt_\com \cup_{\{pt,pt^-\}_\com}^h pt_\com$ and further that when $A$ is commutative, $A^{op}=A$.
\item
Let $f_\bullet:Z_\bullet\to X_\bullet$ be a map of simplicial spaces. Then, the mapping cone  $(C_f)_\bullet$ is given as the homotopy pushout $(C_f)_\bullet\cong X_\bullet \cup_{Z_\bullet}^h pt_\bullet$. We obtain a spectral sequence computing $HH^\com_{(C_f)_\com}(A)$,
$$ {Tor}_{p,q}^{HH_{Z_\com}^{\com}(A)}\left(HH_{X_\com}^{\com}(A),A \right) \Longrightarrow HH_{(C_f)_\com}^{\com}(A) $$
\item
In a more straightforward way, we may use the gluing property to give explicit models for spaces with cubic subdivisions. Our starting point is given by models for interval $I_\com$ and the square $I^2_\com$ via,
$$ CH^\com_{I_\com}(A)=\sum_k A^{\otimes (k+2)}\quad\text{ and }\quad CH^\com_{I^2_\com}(A)=\sum_k A^{\otimes (k+2)^2}. $$
The $i^{\text{th}}$ differential is given in the case of $I_\com$ by multiplying the $i^{\text{th}}$ and $(i+1)^{\text{th}}$ tensor factors, and in the case of $I^2_\com$ by multiplying the $i^{\text{th}}$ and $(i+1)^{\text{th}}$ column of tensor factors and the $i^{\text{th}}$ and $(i+1)^{\text{th}}$ rows of tensor factors simultaneously. For more detail, see \cite[Example 2.3.4]{GTZ}. Then we obtain the Hochschild complex for the cylinder $C_\com=I^2_\com\cup_{(I_\com\cup I_\com)}  I_\com$ by gluing $I^2_\com$ and $I_\com$ along $I_\com\cup I_\com$ on opposite edges of $I^2_\com$, and with this the torus $T_\com=C_\com\cup_{(I_\com\cup I_\com)} I_\com$ by gluing the remaining sides together. A more elaborate version of this is given in \cite[Example 2.3.2]{GTZ}. By using similar, but more elaborate considerations, one can in fact obtain Hochschild models over any surface, see \cite[Section 3.1]{GTZ}.
\end{enumerate}
\end{example}

Higher Hochschild chain complexes behaves much like cochains of mapping spaces (see~\cite[sections 2.2,  2.4]{GTZ}). Indeed they satisfy a kind of exponential law:  

\begin{proposition}[Finite Products of simplicial sets] \label{P:product} Let $X_\com$, $Y_\com$ be simplicial sets. Then there is a natural equivalence
$$CH_{X_\com\times Y_\com}^{\com}(A) \stackrel{\sim}\to CH_{X_\com}^\com\left( CH_{Y_\com}^{\com}(A)\right)$$ in $\hcdga$.
\end{proposition}
\begin{proof} It follows from Corollary 2.4.4 of \cite{GTZ}.
\end{proof}

The Hochschild chain functor $CH:\hsset\times \hcdga \to \hcdga$ is essentialy (up to equivalences) determined by the homotopy pushout axiom, coproduct and its value on a point. In other words, it is the \emph{unique} $(\infty,1)$-functor (up to natural equivalences) satisfying the three axioms (value on a  point, coproduct and locality) listed below in Theorem~\ref{T:deriveduniqueness}. This is once again a consequence of the fact that higher Hochschild is a tensor. The precise uniqueness statement is :

\begin{theorem}[Derived Uniqueness]\label{T:deriveduniqueness}
Let $(X_\com, A)\mapsto F_{X_\com}(A)$ be a bifunctor $\hsset\times \hcdga \to \hcdga$ which satisfies the following three axioms.

\begin{enumerate}
\item {\bf value on a point:} \label{A:point} There is a natural equivalence of CDGAs $F_{pt_\com}(A)\cong A$.
\item {\bf coproduct:} \label{A:coproduct}There are natural equivalences $$F_{\coprod\limits_{I} (X_i)_\com}(A)\cong \colim_{\small \begin{array}{l}K\subset I \\ K \mbox{ finite}\end{array}} \bigotimes_{k\in K} F_{(X_k)_\com}(A) $$
\item {\bf homotopy gluing/pushout:} \label{A:pushout} $F$ sends homotopy pushout in $\hsset$ to homotopy pushout in $\hcdga$. More precisely, given two maps $Z_\com\stackrel{f}\to X_\com$ and $Z_\com\stackrel{g}\to Y_\com$ in $\hsset$, and $W_\com \cong X_\com \bigcup^{h}_{Z_\com} Y_\com$ a homotopy pushout,  one has a natural equivalence
$$F_{W_\com}(A)\cong F_{X_\com}(A)\otimes_{F_{Z_\com}(A)}^{\mathbb{L}} F_{Y_\com}(A).$$  
\end{enumerate}
Then $F$ is naturally equivalent to the higher Hochschild chains bifunctor $CH$ as a bifunctor \emph{i.e.} as an object in $Hom_{(\infty,1)-cat}(\hsset\times \hcdga, \hcdga)$.
\end{theorem}

\begin{proof} This is a consequence of the fact that $\hcdga$ (and actually any presentable $(\infty,1)$-category) is uniquely tensored over $\hsset$, see~\cite[Corollay 4.4.4.9]{Lu11}. 

Alternatively it can be prove   by noticing that any simplicial set $X_\com$ is a (homotopy) colimit of its skeletal filtration
 $sk_n X_\com $, which in turns is obtained by taking a homotopy pushout of $sk_{n-1} X_\com$ with coproducts of standard model $\Delta^n_\com$ of the simplices which are contractible. Then the axioms imply an equivalence $F_{sk_n X_\com}(A) \stackrel{\simeq}\to CH^{\com}_{sk_n X_\com}(A)$  which commutes with the inclusions $sk_{n-1}X_\com \hookrightarrow sk_n X_\com$ so that it induces an natural equivalence $F_{X_\com}(A) \cong CH_{X_\com}(A)$ since the latter satisfies teh axioms by Theorem~\ref{T:derivedfunctor}. 

For the interested reader, we now make this sketch more precise.
{Let $(X_\com,A)\mapsto F_{X_\com}(A)$ be a functor satisfying the assumption of Theorem~\ref{T:deriveduniqueness}. 
If $X_\com$ is a finite discrete simplicial set,  there are natural equivalences in $\hcdga$ $$F_{X_\com}(A)\cong F_{X_0}(A)\cong F_{\coprod_{x\in X_0} \{x\}}(A)\cong A^{\otimes X_0}\cong CH^\com_{X_\com}(A)$$ since both bifunctors $F$ and $CH$ are  homotopy invariant,  commute with finite coproducts and satisfies Axiom~\eqref{A:point}. If $X_\com$ is discrete, non-necessarily finite, then Axiom~\eqref{A:coproduct} implies that $F_{X_\com}(A)\cong \colim_{K \mbox{ finite}} F_{K}(A)\cong \colim_{K \mbox{ finite}} A^{\otimes K} \cong CH_{X_\com}^{\com}(A)$ by definition of Hochschild chain of simplicial sets (Definition~\ref{D:Hoch}).

Let  $sk_n X_\com$ be the $n^{\mbox{th}}$-skeleton of a simplicial set $X_\com$ (see~\cite{GoJa} for instance), \emph{i.e.}, the sub-simpicial set generated by all non-degenerate simplices of dimension less than or equal to $n$.  Note that $X_\com$ is the filtered colimit $X_\com = \bigcup_{n\geq 0} sk_n X_\com$. 
Since $sk_0 X_\com$ is a discrete simplicial set,  we have an natural (in $A$ and $X_\com$) equivalence of algebras $F_{sk_0 X_\com}(A) \stackrel{\simeq}\to CH^{\com}_{sk_0 X_\com}(A)$.

 Now assume $n\geq 1$ and that we have a natural equivalence (induced by a natural zigzag of quasi-isomorphisms) of CDGAs $F_{sk_{n-1} X_\com}(A) \stackrel{\simeq}\to CH^{\com}_{sk_{n-1} X_\com}(A)$ (for any simplicial set $X_\bullet$). 
Let $NX_n$ be the subset of non-degenerate simplices in $X_n$,  $\Delta^n_\com$  be the standard simplicial model $\Delta^n_{k}=Hom_{\Delta}([k], [n])$ of the \emph{topological} $n$-simplex and $ \partial \Delta^n_\com$  its boundary (that is the sub-complex obtained by dropping the only non-degenerate simplex of dimension $n$). Note that $\Delta^n_\com$ is contractible and $\partial \Delta^n_\com$ is a simplicial model for the sphere $S^{n-1}$.
For  $n\geq 1$, one has a pushout diagrams 
\begin{equation} \label{eq:pushoutskn} \xymatrix{ \coprod_{x\in NX_n} \partial \Delta^n_\com \ar[r]^{j_n} \ar@{^{(}->}[d]_{i_n} & sk_{n-1}X_\com\ar@{^{(}->}[d] \\ \coprod_{x\in NX_n}  \Delta^n_\com \ar[r] & sk_n X_\com}\end{equation}
which are homotopy pushouts since the vertical arrows are cofibrations of simplicial sets. Thus Axiom~\eqref{A:pushout} and Theorem~\ref{T:derivedfunctor} yield natural equivalences
$$F_{sk_n X_\com}(A)\cong F_{sk_{n-1}X_\com}(A) \otimes_{F_{\coprod_{x\in NX_n} \partial \Delta^n_\com}(A)}^{\mathbb{L}}  F_{\coprod_{x\in NX_n}  \Delta^n_\com}(A)$$ 
of $F_{\coprod_{x\in NX_n}  \Delta^n_\com}(A)$-algebras as well as a natural equivalence $$CH_{sk_n X_\com}^{\com}(A) \cong CH^{\com}_{sk_{n-1}X_\com}(A) \otimes_{CH^{\com}_{\coprod_{x\in NX_n} \partial \Delta^n_\com}(A)}^{\mathbb{L}}  CH^{\com}_{\coprod_{x\in NX_n}  \Delta^n_\com}(A) $$
of $CH^{\com}_{\coprod_{x\in NX_n} \partial \Delta^n_\com}(A)$-algebras. 
Here the modules structures    are the natural ones induced by the maps in the pushout diagram~\eqref{eq:pushoutskn}. 

We now deduce that $F_{sk_n X_\com}(A)$ is equivalent to $CH_{sk_nX_\com}^{\com}(A)$. 
Recall from Example~\ref{E:Amod}, that an equivalence $A\stackrel{\sim}\to B$ of CDGAs induces an equivalence of their  $(\infty,1)$-categories of modules $B\textit{-}Mod_\infty \stackrel{\sim}\to A\textit{-}Mod_{\infty}$. Note that $\partial \Delta^n_\com$ has no non-degenerate simplices in dimension $n$ and higher. Thus $\partial \Delta^n_\com=sk_{n-1} \partial \Delta^n_\com$ and, by our induction assumption, we have a natural equivalence of CDGAs between $F_{\coprod_{x\in NX_n} \partial \Delta^n_\com}(A)$ and $CH^{\com}_{\coprod_{x\in NX_n} \partial \Delta^n_\com}(A)$  as well as between $F_{sk_{n-1}X_\com}(A)$ and $CH^\com_{sk_{n-1}X_\com}(A)$. Further, as the modules structures are induced by the simplicial set map $\coprod_{x\in NX_n} \partial \Delta^n_\com \to sk_{n-1}X_\com$, it follows that the diagram 
 $$\xymatrix{ F_{\coprod_{x\in NX_n} \partial \Delta^n_\com}(A)\otimes F_{sk_{n-1}X_\com}(A) \ar[r]^{\simeq} \ar[d]_{\mu\circ(j_n\otimes 1)} & CH_{\coprod_{x\in NX_n} \partial \Delta^n_\com}^{\com}(A)\otimes CH_{sk_{n-1}X_\com}^{\com}(A)\ar[d]^{\mu\circ(j_n\otimes 1)} \\ F_{sk_{n-1}X_\com}(A) \ar[r]^{\simeq} & CH_{sk_{n-1} X_\com}^{\com}(A)}$$ is commutative (here $\mu$ is the multiplication and $j_n$ the top map in diagram~\eqref{eq:pushoutskn}). Hence $CH_{sk_{n-1} X_\com}^{\com}(A)$ is equivalent to $F_{sk_{n-1}X_\com}(A)$ as an $F_{\coprod_{x\in NX_n} \partial \Delta^n_\com}(A)$-CDGA. 
 
 We are left to prove that $F_{\coprod_{x\in NX_n}  \Delta^n_\com}(A)$ and $CH^\com_{\coprod_{x\in NX_n}  \Delta^n_\com}(A)$ are equivalent $F_{\coprod_{x\in NX_n} \partial \Delta^n_\com}(A)$-CDGAs. By the homotopy invariance and value on a point axiom, we have a natural equivalence $F_{\coprod_{x\in NX_n}  \Delta^n_\com}(A)\stackrel{\simeq}\to A^{\otimes \# NX_n}$ (induced by the unique map $\Delta^n_\com\to pt_\com$ for each non-degenerate simplex in $NX_n$) and further the $F_{\coprod_{x\in NX_n} \partial \Delta^n_\com}(A)$-algebra structure on $A^{\otimes \# NX_n}$ is induced by the canonical map $\coprod_{x\in NX_n} \partial \Delta^n_\com  \longrightarrow \coprod_{x\in NX_n} pt_\com$. The same argument holds for  Hochschild chains $CH$ instead of $F$ so that  we have a commutative diagram 
 $$\xymatrix{  F_{\coprod_{x\in NX_n} \partial \Delta^n_\com}(A) \ar[rr]^{\simeq} \ar[d]_{i_n}&  & CH_{\coprod_{x\in NX_n} \partial \Delta^n_\com}^{\com}(A) \ar[d]^{i_n}\\ 
 F_{\coprod_{x\in NX_n}  \Delta^n_\com}(A) \ar[r]^{\simeq}  & A^{\otimes \# NX_n} & \ar[l]_{\simeq} CH^\com_{\coprod_{x\in NX_n}  \Delta^n_\com}(A)
 }$$ from which it follows that both $CH_{\coprod_{x\in NX_n}  \Delta^n_\com}^{\com}(A)$ and $F_{\coprod_{x\in NX_n}  \Delta^n_\com}(A)$ are equivalent to $ A^{\otimes \# NX_n}$ as $F_{\coprod_{x\in NX_n} \partial \Delta^n_\com}(A)$-CDGAs. 
 Tensoring the previous equivalence 
 with $ A^{\otimes \# NX_n}$ (over ${F_{\coprod_{x\in NX_n} \partial \Delta^n_\com}(A)}$) we obtain a natural equivalence 
$F_{sk_{n}X_\com}(A) \cong  CH^{\com}_{sk_{n}X_\com}(A)$ in $\hcdga$.

By induction,  we get (for all $n$'s)  natural equivalences $F_{sk_n X_\com}(A) \stackrel{\simeq}\to CH^{\com}_{sk_n X_\com}(A)$ which commutes with the inclusions $sk_{n-1}X_\com \hookrightarrow sk_n X_\com$. Since $X_\com=\bigcup_{n\geq 0} sk_n X_\com$, there is a coequalizer in the $(\infty,1)$-category $\hsset$:
 $$\xymatrix{\coprod_{n\in \N} sk_n X_\com \ar[d]^{\delta} \ar[r]^{\rm id} & \coprod_{n\in \N} sk_n X_\com \ar[d] \\ \coprod_{n\in \N} sk_{n} X_\com\ar[r] & X_\com }$$ where $\delta:\coprod_{n\in \N} sk_n X_\com\to \coprod_{n\in \N} sk_{n} X_\com $ is the map induced by the canonical inclusions $sk_n X_\com \to sk_{n+1} X_\com$. Thus by axiom~{\bf (3)}, there is a natural equivalence \begin{equation} \label{eq:uniquenesspushout} F_{X_\com}(A) \cong F_{\coprod_{n\in \N} sk_n X_\com}(A) \mathop{\otimes}\limits_{F_{\coprod_{n\in \N} sk_{n} X_\com}(A)}^{\mathbb{L}} F_{\coprod_{n\in \N} sk_n X_\com}(A) \end{equation} and similarly for Hochschild chains by Theorem~\ref{T:derivedfunctor}. By axiom~{\bf  (2)} applied to the bifunctors $F$ and $CH$, the natural equivalences $F_{sk_n X_\com}(A) \stackrel{\simeq}\to CH^{\com}_{sk_n X_\com}(A)$ for all $n$ thus yield a natural equivalence $F_{\coprod_{n\in \N} sk_n X_\com}(A) \cong CH^\com_{\coprod_{n\in \N} sk_n X_\com}(A) $.  The natural equivalence of CDGAs  $F_{X_\com}(A) \cong CH^{\com}_{X_\com}(A)$ now follows from equivalence~\eqref{eq:uniquenesspushout} and the analogous equivalence for Hochschild chains.} 
\end{proof}

\begin{remark}\label{R:cdga-}
If $A$ is concentrated in non-positive degrees, then $CH_{X_\com}^{\com}(A)$ is also concentrated in non-positive degrees. This happens for instance if $A$ is the CDGA associated to a simplicial (non-graded) commutative algebra. In that case, it is possible to replace $\cdga$ and $\hcdga$ in Theorem~\ref{T:derivedfunctor} and Theorem~\ref{T:deriveduniqueness} by $\cdga^{\leq 0}$ the category of CDGAs concentrated in non-positive degrees and $\hcdga^{ \leq 0}$ its associated $(\infty,1)$-categories (the proofs being unchanged).  
\end{remark}

\begin{remark}Again, one can replace $\hsset$ by its topological counterpart $\hTop$ in Theorem~\ref{T:deriveduniqueness}, see Proposition~\ref{P:Topderived}.
\end{remark}

\begin{remark} \label{R:limskn} Note that one can deduce from the coproduct axiom~{\bf (2)} in Theorem~\ref{T:deriveduniqueness} and the natural  equivalence~\eqref{eq:uniquenesspushout} that the natural map $\colim F_{sk_n X_\com}(A) \stackrel{\simeq}\to F_{X_\com}(A)$ is an equivalence. This is in particular true for Hochschild chains:
  \begin{equation}\label{eq:limskn} \colim_{n\geq 0} CH_{sk_n X_\com}(A) \stackrel{\simeq}\to CH_{X_\com}(A).\end{equation}
\end{remark}

\begin{remark} \label{R:deriveduniquenessfunctor} If $G:\hcdga \to \hcdga$ is a functor,  one can replace the value on a point axiom by the existence of a natural quasi-isomorphism $F_{pt}(A)\cong G(A)$. The proof of the  Theorem~\ref{T:deriveduniqueness} shows the following
\begin{corollary}\label{C:deriveduniquenessfunctor}
 Let $G:\hcdga \to \hcdga$ be a functor and $(X_\com, A)\mapsto F_{X_\com}(A)$ be a bifunctor $\hsset\times \hcdga \to \hcdga$ which satisfies the axioms~{\bf (2)} and~{\bf (3)} in Theorem~\ref{T:deriveduniqueness} and with axiom~{\bf (1)} replaced by $F_{pt_\com}(A)\cong G(A)$. Then $F_{X_\com}(A)$ is naturally equivalent $CH_{X_\com}^{\com}(G(A))$.
\end{corollary}
For instance, consider  the bifunctor given by $(X_\com, A)\mapsto CH_{X_\com}^{\com}(A)\otimes CH_{X_\com}^{\com}(B)$ whose value on a point is the functor $A\mapsto A\otimes B$.  By Corollary~\ref{C:deriveduniquenessfunctor}, this functor is  isomorphic to $(X,A)\mapsto CH_{X_\com}^\com (A\otimes B)$ which gives another proof of the fact that the Hochschild chains preserve finite coproduct of CDGAs.  The same argument shows that $CH$ also commutes with finite homotopy pushouts of CDGAs, see Corollary~\ref{C:hocolim} below.
\end{remark}

\begin{corollary}\label{C:hocolim}
The Hochschild chain bifunctor $CH:\hsset\times \hcdga \to \hcdga$ commutes with finite colimits in $\hsset$ and all colimits in $\hcdga$, that is  there are natural equivalences
\begin{eqnarray*}
 CH_{\colim_{\mathcal{F}} {X_i}_\com}^{\com}(A) & \cong & \colim_{\mathcal{F}} CH_{{X_i}_\com}^{\com}(A) \quad (\text{for a finite category } \mathcal{F}),\\
 CH^\com_{X_\com}(\colim {A_i}) & \cong & \colim CH^\com_{X_\com}(A_i). 
\end{eqnarray*}
\end{corollary}
Here the colimits are colimits in $\hsset$ or $\hcdga$.
\begin{proof}
Any finite colimits can be obtained by a composition of coproducts and pushouts (or coequalizers). Thus the result for colimits in $\hsset$ follows from Theorem~\ref{T:derivedfunctor}, Axioms~{\bf (2)} and~{\bf (3)}. 
Let $\colim_{i\in \mathcal{I}} {A_i}$ be a non-empty colimit of CDGAs and let $i_0$ be an object in the indexing category $\mathcal{I}$. By functoriality we can define a functor $G_{\mathcal{I}}:\hcdga \to \hcdga$ by the formula $A\mapsto G_{\mathcal{I}}(B) :=\colim \tilde{B}_i$ where $\tilde{B}_i\cong A_i$ if $i\neq i_0$ and $\tilde{B}_{i_0}\cong B$. In other words we fix all the variables but the one indexed by $i_0$.
Now applying Corollary~\ref{C:deriveduniquenessfunctor} to the bifunctor  $F: \hsset \times \hcdga \to \hcdga$  defined by 
$F_{X_\com}(B)\cong \colim CH_{X_\com}^{\com}(\tilde{B}_i)$ we get an natural equivalence $$CH_{X_\com}^{\com}(\colim_{i\in \mathcal{I}} A_i) \cong \colim_{i\in \mathcal{I}} CH^\com_{X_\com}(A_i).$$
Since the simplicial module $n\mapsto CH_{X_n}^{\com}(k)$ is isomorphic to the constant simplicial $k$-algebra $n\mapsto k$, the result also follows for empty colimits. 
\end{proof}

\begin{example}\label{EX:A-pushout}
By Corollary~\ref{C:hocolim},
 given two maps $f:R\to A$ and $g:R\to B$ of CDGAs,  there is a natural equivalence (in $\hcdga$) $$CH_{X_\com}^\com\Big(A\mathop{\otimes}^{\mathbb{L}}_{R} B\Big)\; \cong \; CH_{X_\com}^{\com}(A) \mathop{\otimes}^{\mathbb{L}}_{CH_{X_\com}^{\com}(R)}CH_{X_\com}^{\com}(B).$$
\end{example}

\subsection{Pointed simplicial sets and modules}·\label{SS:modules}
In this section we quickly explain how to add an $A$-module $M$ to the story developed in Section~\ref{S:inftyfunctor}. 

Let $A$ be a CDGA and recall from Example~\ref{E:Amod} the $(\infty,1)$-category $A\textit{-}Mod_\infty$ induced by the (model category) $A\textit{-}Mod$ of $A$-modules.  Similarly, the model category of pointed simplicial sets yields the $(\infty,1)$-category ${\hsset}_*$ of pointed simplicial sets (Example~\ref{E:hsset}).

Since the inclusion $pt_\com \to X_\com$ is always a cofibration, the canonical equivalence $CH_{X_\com}^{\com}(A,M) \cong M\otimes_{A} CH_{X_\com}^{\com}(A)$ given by Proposition~\ref{P:pushout} implies that 
\begin{equation}
\label{eq:Mderivedtensor} CH_{X_\com}^{\com}(A,M) \;\cong \;M\mathop{\otimes}^{\mathbb{L}}_{A} CH_{X_\com}^{\com}(A) \;\cong \;  M\!\!\mathop{\otimes}^{\mathbb{L}}_{CH_{pt_\com}^{\com}(A)} CH_{X_\com}^{\com}(A).
\end{equation}
naturally in $A\textit{-}Mod_\infty$ and $CH_{X_\com}^{\com}(A)\textit{-}Mod_\infty$.

\smallskip

Proposition~\ref{P:tensor}, Proposition~\ref{P:pushout}, Proposition~\ref{P:product}, Theorem~\ref{T:derivedfunctor} and its proof imply 
\begin{theorem}\label{T:Mderivedfunctor}
The Hochschild chain lifts as a bifunctor of $(\infty,1)$-categories $CH_{(-)}(A,-): {\hsset}_* \times A\textit{-}Mod_\infty \to A\textit{-}Mod_\infty$ which satisfies the following axioms
\begin{enumerate}
\item {\bf value on a point:} {there is a natural equivalence} $CH_{pt_\com}^\bullet(A,M)\cong M$ in $A\textit{-}Mod_\infty$.
\item {\bf action of $CH$:} $CH^\com_{X_\com}(A,M)$ is naturally a $CH^\com_{X_\com}(A)$-module, \emph{i.e.},  the Hochschild chain lifts as an $(\infty,1)$-functor $CH_{X_\com}^{\com}(A,-):A\textit{-}Mod_\infty \to CH_{X_\com}^{\com}(A)\textit{-}Mod_\infty$. 
\item {\bf bimonoidal:} there is an natural equivalence
$$CH_{X_\com\coprod Y_\com}^{\com}(A,M)\cong CH_{X_\com}^{\com}(A,M) \otimes CH_{Y_\com}^{\com}(A)$$ in $A\textit{-}Mod_\infty$ as well as in $CH_{X_\com\coprod Y_\com}^{\com}(A)\textit{-}Mod_\infty$,  for any pointed simplicial set $X_\com$ and simplicial set $Y_\com$. For $M\in A\textit{-}Mod_\infty$ and $N\in B\textit{-}Mod_\infty$, there is an natural equivalence 
$$CH_{X_\com}^\com (A\otimes B, M\otimes N)\cong CH_{X_\com}^{\com}(A,M) \otimes  CH_{X_\com}^{\com}(B,N)$$ in $A\otimes B\textit{-}Mod_\infty$ and $CH_{X_\com}^{\com}(A\otimes B)\textit{-}Mod_\infty$.
\item {\bf locality:} let  $f:Z_\com\to X_\com$ and $g:Z_\com\to Y_\com$ be maps in $\hsset*$. There is an natural equivalence  
$$CH^{\com}_{X_\com \bigcup^{h}_{Z_\com} Y_\com }(A,M\otimes^{\mathbb{L}}_{A} N)\cong CH^{\com}_{X_\com}(A,M)\otimes_{CH^{\com}_{Z_\com}(A)}^{\mathbb{L}} CH^{\com}_{Y_\com}(A,N)$$ in $A\textit{-}Mod_\infty$ and $CH_{X_\com \bigcup^{h}_{Z_\com} Y_\com }^{\com}(A)\textit{-}Mod_\infty$. If $N=A$, then only $X_\com$ needs to be pointed and the maps may be in $\hsset$.
\item {\bf product:} {Let $X_\com$, $Y_\com$ be pointed simplicial sets.} There is an natural equivalence
$$CH_{X_\com\times Y_\com}^{\com}(A,M) \stackrel{\sim}\to CH_{X_\com}^\com\left( CH_{Y_\com}^{\com}(A), CH_{Y_\com}^{\com}(A,M))\right)$$ in $A\textit{-}Mod_\infty$ and $CH_{X_\com\times Y_\com}^{\com}(A)\textit{-}Mod_\infty$.
\end{enumerate}
\end{theorem}
{Note that from the point Axiom~{\bf (1)} and the locality Axiom~{\bf (4)} applied to the canonical maps $pt_\com \to pt_\com$ and $pt_\com \to X_\com$ (given by the base point of $X_\com$) follows immediately the equivalence~\eqref{eq:Mderivedtensor} above. 
This property actually implies that Axioms~{\bf (1)}, {\bf (2)} and {\bf (4)}  characterize $CH_{(-)}(A,-)$:}
\begin{proposition}\label{P:Mderiveduniqueness}
 Let $G: A\textit{-}Mod_\infty \to A\textit{-}Mod_\infty$ be an $(\infty,1)$-functor and let $\mathcal{M}:{\hsset}_* \times A\textit{-}Mod_\infty \to A\textit{-}Mod_\infty$ be any  $(\infty,1)$-bifunctor  which satisfies the following axioms
\begin{description}
\item[i)] {\bf value on a point:} {there is a natural equivalence} $\mathcal{M}(pt_\com,M) \cong G(M)$ naturally in $A\textit{-}Mod_\infty$.
\item[ii)] {\bf action of $CH$:} $\mathcal{M}(X_\com,M)$ is naturally a $CH_{X_\com}(A)$-module. 
\item[iii)]  {\bf locality:}  {Given maps $f:Z_\com\to X_\com$ and $g:Z_\com\to Y_\com$ with $X_\com$ a pointed simplicial set,} There is a natural equivalence  ( in $A\textit{-}Mod_\infty$)  
$$\mathcal{M}(X_\com \cup_{Z_\com}^h Y_\com,M)\cong \mathcal{M}(X_\com,M)\otimes_{CH^{\com}_{Z_\com}(A)}^{\mathbb{L}} CH^{\com}_{Y_\com}(A)$$
\end{description}
Then $\mathcal{M}$ is naturally equivalent to $CH_{(-)}(A,G(-))$ as an $(\infty,1)$-bifunctor.
\end{proposition}
Note that Axiom~{\bf ii)} is needed to make sense of Axiom~{\bf iii)}.
\begin{proof}
 Let $Y_\com$ be in ${\sset}_*$ and $g:pt_\com\to Y_\com$ be the structure map. The locality axiom for the pushout $pt_\com \leftarrow pt_\com \stackrel{g}\to Y_\com $ (where we take $X_\com=pt_\com$) gives natural equivalences
 $$\mathcal{M}(X_\com,M) \;\cong \;\mathcal{M}(pt_\com,M) \! \mathop{\otimes}^{\mathbb{L}}_{CH_{pt_\com}^{\com}(A)}  \! CH_{X_\com}^{\com}(A)\;\cong \;G(M) \! \mathop{\otimes}^{\mathbb{L}}_{CH_{pt_\com}^{\com}(A)} \! CH_{X_\com}^{\com}(A)  $$ where the last equivalence follows from Axiom~{\bf i)}. {The result now follows from the natural equivalences~\eqref{eq:Mderivedtensor}.}
\end{proof}

One can specialize the locality Axiom~{\bf (4)} in Theorem~\ref{T:Mderivedfunctor} a bit more by considering a pointed simplicial set $Z_\com$ and pointed maps   $f:Z_\com\to X_\com$ and $g:Z_\com\to Y_\com$. In that case,  $CH_{X_\com}^{\com}(A,M)$ and $CH_{Y_\com}(A,N)$ inherit natural $CH_{Z_\com}^{\com}(A)$-modules structures (induced by $f$ and $g$).  Then
Theorem~\ref{T:Mderivedfunctor}  yields a relative version of the Eilenberg-Moore spectral sequence.
\begin{corollary} \label{C:MlocalitySpecSeq}Let $W_\com \cong X_\com \cup_{Z_\com}^h Y_\com$ be a homotopy pushout, 
there is an natural equivalence $$CH_{X_\com \cup_{Z_\com}^h Y_\com}^\com\Big(A, M\mathop{\otimes}^{\mathbb{L}}_A N\Big) \;\cong \; CH^{\com}_{X_\com}(A,M)\mathop{\otimes}_{CH^{\com}_{Z_\com}(A)}^{\mathbb{L}} CH^{\com}_{Y_\com}(A,N)$$ in $A\textit{-}Mod_\infty$ and $CH_{X_\com \cup_{Z_\com}^h Y_\com}^{\com}(A)\textit{-}Mod_\infty$. 

In particular there is a natural strongly convergent spectral sequence of cohomological type of the form
 $$E_2^{p,q}:= {Tor}_{p,q}^{HH_{Z_\com}^{\com}(A)}\left(HH_{X_\com}^{\com}(A,M),HH_{Y_\com}^{\com}(A,N)\right) \Longrightarrow HH_{W_\com}^{p+q}\Big(A, M\mathop{\otimes}_A^{\mathbb{L}} N\Big) $$ where $q$ is the \emph{internal} grading.  The spectral sequence is furthermore a spectral sequence of differential $HH_{Z_\com}^{\com}(A)$-modules.
\end{corollary}
Recall that we are considering a cohomological grading; thus the spectral sequence lies in the left half-plane with respect to this grading (and $p$ is negative). 
\begin{proof}
Let $\widehat{M}\to M$ be a cofibrant replacement of $M$ in $A\textit{-}Mod$,  $\widehat{X}_\com\to X_\com$ a fibrant replacement of $X_\com$ and $\widehat{f}:Z_\com\to \widehat{X}_\com$ be  a cofibration lifting $f$. By Theorem~\ref{T:Mderivedfunctor},  we have  $CH_{W_\com}\Big(A,M\mathop{\otimes}^{\mathbb{L}}_A N\Big)\cong CH_{W_\com}\Big(A,\widehat{M}\otimes_A N\Big)$  (in $CH_{W_\com}(A)\text{-}Mod$)  and $$CH_{X_\com}^{\com}(A,M)\cong CH_{\widehat{X}_\com}^{\com}(A,\widehat{M})\cong \widehat{M}\otimes_A CH_{\widehat{X}_\com}^{\com}(A).$$  Since $\widehat{f}:Z_\com\to \widehat{X}_\com$ is a degree wise injection,   it suffices to prove that $$CH_{\widehat{X}_\com}^{\com}(A)\mathop{\otimes}_{CH_{Z_\com}^{\com}(A)} CH_{Y_\com}(A,N) \cong CH_{\widehat{X}_\com\cup_{Z_\com} Y_\com}^{\com}(A,N)$$ as a $CH_{\widehat{X}_\com\cup_{Z_\com} Y_\com}^{\com}(A)\cong CH_{\widehat{X}_\com}^{\com}(A)\mathop{\otimes}_{CH_{Z_\com}^{\com}(A)} CH_{Y_\com}(A)$-module which is Proposition~\ref{P:pushout}. 
Now the spectral sequence is obtained as in the proof of Corollary~\ref{C:HlocalitySpecSeq} (using Theorem~\ref{T:Mderivedfunctor} instead of Theorem~\ref{T:derivedfunctor}). 
\end{proof}
\begin{example}
If $X_\com$, $Y_\com$ and $Z_\com$ are contractible, the spectral sequence in Corollary~\ref{C:MlocalitySpecSeq} boils down to the usual Eilenberg-Moore spectral sequence  $${Tor}_{p,q}^{H^\com(A)}\Big(H^\com(M), H^{\com}(N)\Big)\Longrightarrow H^{p+q}\Big(M\otimes_A^{\mathbb{L}} N\Big)$$ of differential $H^\com(A)$-modules  (see~\cite{KM}). 
\end{example}

\begin{remark}\label{R:2othersSpecSeq}
 Besides the Eilenberg Moore spectral sequence~\ref{C:MlocalitySpecSeq}, there is also an Atiyah-Hirzebruch kind of spectral sequence for higher Hochschild chains:  the skeletal filtration of a simplicial set $X_\com$ induces a decreasing filtration $ \cdots \supset F^{p} \cdots \supset F^{-1} \supset F^{0} \supset \{0\} $ of $CH_{X_\com}^{\com}(A,M)$, where  $F^{p}:=\bigoplus_{n\leq -p} CH_{X_n}^{\com}(A,M)$. 
This filtration yields a left half-plane spectral sequence of cohomological type with exiting differential and further, the cohomology of the associated graded $\bigoplus_{p} F^p/F^{p+1}$ is the Hochschild chain complex over $X_\com$ of the CGA $H^\com(A)$ with value in $H^\com(M)$. Hence we get from~\cite{Bo}:
\begin{proposition}\label{P:AHSpecSeq}
There is a strongly convergent spectral sequence of cohomological type
$$ E^2_{p,q}:= HH_{X_\com}^{p+q}(H^\com(A),H^\com(M))^q \Longrightarrow HH_{X_\com}^{p+q}(A,M)$$ where $q$ is the internal degree. If $M=A$, this is  a spectral sequence of CDGAs. 
\end{proposition}
For the sake of completeness, we also mention that there is another spectral sequence to compute higher Hochschild due to Pirashvili which is the Grothendieck  spectral sequence associated to the composition of functors $X_\com\mapsto k_{X_\com}\mapsto k_{X_\com}\otimes_{\Gamma}^{\mathbb{L}}\mathcal{L}(A)$ that was defined in Proposition~\ref{P:homologyinvariance}. See~\cite[Theorem 2.4]{P} for details.
\end{remark}

\section{Factorization algebras and derived Hochschild functor over spaces}\label{S:factor-alg}
The main goal of this section  is to prove that Hochschild chains are a special kind of \emph{factorization algebras} in the sense of~\cite{CG} (allowing to compute it  using covers or CW-decomposition).

\subsection{The Hochschild  $(\infty,1)$-functor in $\Top$}
 The Quillen equivalence between the  model categories of simplicial sets and  topological spaces induces an equivalence $\hTop \stackrel{S_\infty}\longrightarrow \hsset$ of $(\infty,1)$-categories (Example~\ref{E:hsset}).  Here $S_\infty$ is the $(\infty,1)$-functor lifting the singular set functor $X\mapsto S_\com(X)=\mathop{Map}(\Delta^\bullet, X)$.   Recall that to any space $X$  we naturally associate the CDGA $CH^\com_X(A) = CH^\com_{S_\bullet(X)}(A)$, the Hochschild chains of $A$ over $X$. The canonical adjunction map $X_\com\to S_\com(|X_\com|)$ yields a natural quasi-isomorphism $CH_{X_\com}^{\com}(A)\to CH_{|X_\com|}^{\com}(A)$ of CDGAs by Proposition~\ref{P:homologyinvariance}.

From the above equivalence $\hTop \stackrel{\sim}\longrightarrow \hsset$ (or alternatively by changing $\hsset$ to $\hTop$ in all the proofs in Section~\ref{S:inftyfunctor} and~\ref{SS:modules}), we deduce the following   topological counterpart to the results of Section~\ref{S:inftyfunctor} and Section~\ref{SS:modules}. 
\begin{proposition}\label{P:Topderived}\begin{itemize}
\item[i)] The Hochschild chain over spaces functor $(X,A)\mapsto CH_X^\com(A)$ lifts as an $(\infty,1)$-bifunctor $CH: \hTop\times \hcdga\to \hcdga$ fitting into the commutative diagram
$${\small \xymatrix{ \hsset \times \hcdga \ar[rr]^{CH} & & \hcdga \\
\hTop\times \hcdga \ar[u]^{S_\infty}_{\simeq} \ar[rru]_{CH} && }} $$ that satisfies all  the axioms of Theorem~\ref{T:derivedfunctor} (with $\hTop$ instead of $\hsset$).
\item[ii)]
Further, up to natural equivalences of $(\infty,1)$-bifunctors, it is the only bifunctor $\hTop\times \hcdga\to \hcdga$ satisfying the axioms of Theorem~\ref{T:deriveduniqueness} (with $\hTop$ instead of $\hsset$).
\item[iii)] Replacing $\hsset$ by $\hTop$ and ${\hsset}_*$ by ${\hTop}_*$, the analogs of Corollary~\ref{C:HlocalitySpecSeq}, Proposition~\ref{P:product}, Corollary~\ref{C:hocolim},  Theorem~\ref{T:Mderivedfunctor} and  Corollary~\ref{C:MlocalitySpecSeq} hold.
\end{itemize}
\end{proposition}

The fact that Hochschild chains are computed by taking colimits over finite simplicial sets has the following translation for topological spaces.
\begin{proposition}[Compact support] Let $X$ be (weakly homotopic to) a $CW$-complex, $A$ a CDGA and $M$ an $A$-module. There are natural equivalences 
 $$\colim_{ \scriptsize \begin{array}{l} K\to X\\ K\mbox{ compact}\end{array}}\!\! \!\!\!\!\! \Big(CH_{K}^{\com}(A) \Big)\stackrel{\simeq} \longrightarrow CH_{X}^{\com}(A) \;\; \text{ and } \;\colim_{ \scriptsize \begin{array}{l} K\to X\\ K\mbox{ compact}\end{array}}\!\!\!\!\!\!\! \Big( CH_{K}^{\com}(A,M) \Big)\stackrel{\simeq} \longrightarrow CH_{X}^{\com}(A,M)$$ in $\hcdga$ and $CH_{X}^{\com}(A)\textit{-}Mod_\infty$ respectively.
\end{proposition}
\begin{proof} Since $CH_{X}^{\com}(A,M)\cong M\mathop{\otimes}^{\mathop{L}}_{A} CH_{X}^{\com}(A)$ in $CH_{X}^{\com}(A)\textit{-}Mod_\infty$, we only need to prove that the first map is an equivalence.

 We first assume $X$ to be a $CW$-complex of finite dimension $n$. Let $X_\bullet$ be a simplicial set model of $X$ with no non-degenerate simplices in dimension $m>n$. Then, given a finite simplicial set $K_\com$,  any map $f: K_\bullet \to X_\bullet$ factors through a finite simplicial set $\tilde{K}_\bullet$ with no  non-degenerate simplices in dimension $m>n$. Thus,  the realization $|\tilde{K}_\bullet|$ is compact.
Conversely, if $\tilde{K}$ is a compact subset of the $CW$-complex $X$,  it has a simplicial model $K_\bullet$ with finitely many non-degenerate simplices. Further, any map $K\to X$ has a compact image, since $X$ is Hausdorff, and thus factors through a compact subset of $X$. We get a zigzag  
\begin{multline*}
 \colim_{ \scriptsize \begin{array}{l} K\to X \\ K\mbox{ compact}\end{array}} CH^\com_{K}(A) \stackrel{\simeq}\longleftarrow \colim_{ \scriptsize \begin{array}{l} \tilde{K}\subset X  \\ \tilde{K}\mbox{ compact}\end{array}} CH^\com_{\tilde{K}} (A) \stackrel{\simeq}\longrightarrow \colim_{ \scriptsize  \tilde{K}_\bullet \in \mbox{ FNDS}(X_\bullet) } CH^\com_{\tilde{K}_\bullet}(A) \\
 \stackrel{\simeq}\longrightarrow \colim_{ \scriptsize \begin{array}{l} K\bullet \to X_\bullet \\ {K}_\bullet \mbox{ finite} \end{array}  } CH^\com_{{K}_\bullet}(A)
 \stackrel{\simeq}\longrightarrow CH^\com_{X_\bullet}(A)   \stackrel{\simeq}\longrightarrow  CH^\com_{X}(A)
\end{multline*}
where the first and third arrows are equivalences since they are induced by cofinal functors. 
  Here $FNDS(X_\bullet)$ is the set of simplicial subsets of $X_\bullet$ with finitely many non-degenerate simplices. That the other arrows in the zigzag are equivalences follows from Proposition~\ref{P:Topderived} and thus the result is proved  for finite dimensional $CW$-complexes.
  
 We now reduce the general case to the finite dimensional one. Let $X_\com$ be a simplicial set model of $X$. The geometric realization $|sk_n X_\com|$ of  $sk_nX_\com$ is a finite dimensional $CW$-complex, and, if $K$ is compact, any map $f: K\to X$ factors as the composition $K\to |sk_n X_\com| \hookrightarrow X$ for some $n$. Hence the natural map
 $$\colim_{n} \Big(\colim_{ \scriptsize \begin{array}{l} \tilde{K}\to |sk_n X_\com| \\ \tilde{K}\mbox{ compact}\end{array}} CH^\com_{K}(A)\Big)  \quad\longrightarrow \quad  \colim_{ \scriptsize \begin{array}{l} K\to X \\ K\mbox{ compact}\end{array}} CH^\com_{K}(A) $$ is a natural equivalence in $\hcdga$. The result now follows from the finite dimensional case, the natural equivalence  $\colim_{n} CH_{sk_n X_\com}^{\com}(A) \; \stackrel{\sim} \to CH_{X_\com}^{\com}(A)$  (see Remark~\ref{R:limskn}), and Proposition~\ref{P:Topderived}.i). 
\end{proof}
{Specifying Corollary~\ref{C:hocolim} to the case of topological spaces that are obtained by attaching cells, we get the following lemma.}
\begin{lemma}\label{L:CHhandling}
 Let $X_0$ be (weakly homotopic to) a $CW$-complex and $X$ be (weakly homotopic to) a $CW$-complex obtained from $X_0$ by attaching a countable family $ (C_n)_{n\in \N}$ of cells. We let  $X_n$ be the result of attaching the first $n$ cells. For any CDGA $A$, one has a natural equivalence
$$\colim_{n\in \N} CH_{X_n}^{\com}(A) \stackrel{\simeq}\longrightarrow CH_{X}^{\com}(A) $$ in $\hcdga$, as well as $\colim_{n\in \N} CH_{X_n}^{\com}(A,M) \stackrel{\simeq}\longrightarrow CH_{X}^{\com}(A) $ in $CH_{X}^{\com}(A)\textit{-}Mod_\infty$.
\end{lemma}
\begin{proof} Since $CH_{X}^{\com}(A,M)\cong M\mathop{\otimes}^{\mathop{L}}_{A} CH_{X}^{\com}(A)$ as $CH_{X}^{\com}(A)$-modules, we only need to consider the case of $CH_{X}^{\com}(A)$.
Further, the cells $C_i$ are homeomorphic to euclidean balls and the attaching maps have domain given by their boundaries. Thus we may assume that each $X_n$ is obtained from a simplicial set model $(X_0)_\bullet$ of $X_0$ by adding finitely many non-degenerates simplices. Thus  we get a sequence of  cofibrations of simplicial sets (\emph{i.e.} degree wise injective maps) 
$(X_0)_\bullet \hookrightarrow (X_1)_\bullet \cdots \hookrightarrow (X_n)_\bullet \cdots \hookrightarrow X_\bullet=\colim_{n\in \N} (X_n)_\bullet$ which are (homotopy) models for the sequence of maps $X_0\to X_1\to \cdots \to X$. By definition of Hochschild chains, there is a canonical equivalence
$ \colim_{\scriptsize \begin{array}{l} K_\bullet\to X_\bullet\\
K_\bullet\mbox{ finite}\end{array}} \!\!CH_{K_\bullet}^{\com}(A) \cong CH_{X_\bullet}^{\com}(A)$ of CDGAs. The maps $(X_n)_\bullet \to X_\bullet$ assemble to give a map of colimits
\begin{equation}\label{eq:CHhandling}
 \colim_{n\in \N} \Big( \colim_{\scriptsize \begin{array}{l} K_\bullet\to (X_n)_\bullet\\
K_\bullet\mbox{ finite}\end{array}} \!\!CH_{K_\bullet}^{\com}(A)\Big) \longrightarrow \colim_{\scriptsize \begin{array}{l} K_\bullet\to X_\bullet\\
K_\bullet\mbox{ finite}\end{array}} \!\!CH_{K_\bullet}^{\com}(A) .
\end{equation}
Given a finite set $K_i$ and a map $f_i: K_i\to X_i$, the image $f_i(K_i)$ lies in some $(X_n)_i$ since $X_\bullet=\colim_{n\in \N} (X_n)_\bullet$ and $f_i(K_i)$ is finite. This proves that the family $K_\bullet \to (X_n)_\bullet$ of maps from a pointed set into some $(X_n)_\bullet$ is cofinal and thus the map~\eqref{eq:CHhandling} is an equivalence in $\hcdga$.
\end{proof}
\begin{remark} It is possible to enhance the result of Lemma~\ref{L:CHhandling} in the following way.  One can take any space $X_\emptyset$ and a space $X$ obtained by attaching a family $(C_i)_{i\in I}$ of other spaces to it. Then, essentially the same argument as the one in  Lemma~\ref{L:CHhandling} shows that $CH^\com_{X}(A)$ is the colimit $\colim CH^\com_{X_F}(A)$ over all possible subspaces $X_F\subset X$ obtained by attaching finitely many $C_i$'s. 
\end{remark}

\smallskip

Let us conclude this section by giving an analog of Leray acyclic cover theorem/Mayer Vietoris principle for Hochschild chains. 

Let $X$ be a topological space and $\mathcal{U} =(U_i)_{i\in I}$ be a \emph{good cover}  for $X$, \emph{i.e.} a cover such that the $U_i$ and all of their nonempty finite intersections are contractible. We denote $N_\com(I)$ the nerve of the cover, that is $N_0(I)=I$, $N_1(I)$ is the set of pairs of indices $i_0,i_1$ such that $U_{i_0}\cap U_{i_1} \neq \emptyset$ and so on. 
\begin{corollary} Let $X$ be a topological space and $\mathcal{U} =(U_i)_{i\in I}$ be a \emph{good cover} for $X$ such that the inclusions $U_i\cap U_j \to U_i$ are cofibrations. Then there is a natural equivalence
 $$ CH^\com_{X}(A)\stackrel{\simeq}\longrightarrow  A^{\otimes I} \mathop{\otimes}\limits_{A^{\otimes N_1(I)}}^{\mathbb{L}} A^{\otimes I}$$ in $\hcdga$. Here the left and right module structure are induced by the two canonical projections $N_1(I)\to N_0(I)=I$ given by $ (i,j)\mapsto i$, $(i,j)\mapsto j$.
\end{corollary}
\begin{proof}
 Since each $U_i$ is contractible, the natural map $CH_{U_i}^{\com}(A) \to CH_{pt}^{\com}(A)\cong A$ is an equivalence by Theorem~\ref{T:derivedfunctor} and similarly (when $U_i\cap U_j$ is not empty) for the natural map $CH_{U_i\cap U_j}^{\com}(A)\to A$ because $\mathcal{U}$ is a good cover. Since $X$ is the coequalizer $\coprod_{N_1(I)} U_i\cap U_j \rightrightarrows \coprod_{I} U_i \rightarrow X$,  the result follows from the coproduct axiom (2) and the gluing axiom (3) in Theorem~\ref{T:derivedfunctor} (or Proposition~\ref{P:Topderived}).
\end{proof}

\subsection{(Pre)Factorization algebras}\label{S:Factorization}

We now explain a relationship between factorization algebras (as defined by Costello and Gwilliam~\cite{CG,Co}) and Higher Hochschild chains. 

Let $A$ be a CDGA and $X$ be a topological space. We denote $Op(X)$ the set of open subsets of $X$. For every open subset $V$ of $X$ and a family of disjoint open subsets $U_1,\dots, U_n \subset V$, there is a canonical morphism of CDGAs 
 $$\mu_{U_1,\dots, U_n,V}: CH_{U_1}^{\com}(A) \otimes \cdots \otimes CH_{U_n}^{\com}(A) \to \big(CH_{V}^{\com}(A)\big)^{\otimes n} \to CH_{V}^{\com}(A) $$ induced by functoriality by the inclusions  $U_i \hookrightarrow V$ and the multiplication in $CH_{V}^{\com}(A)$. 
 
 \smallskip
 
 These maps are the structure maps of a prefactorization algebra on $X$ in the sense of~\cite{CG,Co}. Note that for a (possibly $(\infty,1)$-) symmetric monoidal category $(\mathcal{C},\otimes)$, we denote by $\mathop{PFac}_{X}(\mathcal{C})$ the ($(\infty,1)$-) category of \textbf{prefactorization algebras on $X$ taking values in $\mathcal{C}$} (see~\cite{CG}).  In particular, $\mathop{PFac}_{X}(\cdga)$ is the category of \emph{commutative} prefactorization algebras on $X$ as defined in~\cite{CG}. We will actually be only interested in the case where $\mathcal{C}$ is an ($(\infty,1)$-)category of algebras over an ($\infty$)-Hopf operad. We have:
 \begin{lemma}\label{P:prefactorization} The rule $U\mapsto CH_U^\com(A)$ together with the maps $\mu_{U_1,\dots, U_n,V}$ define a natural structure of a \emph{prefactorization algebra} on $X$. Further
 \begin{enumerate}
  \item The above rule $A\mapsto \Big( \big(CH_{(U)}^{\com}(A)\big)_{ U\in Op(X)}; \, \big(\mu_{U_1,\dots,U_n,V}\big)\Big)$ defines a functor $\mathcal{CH}_X: \cdga \to \mathop{PFac}_{X}(\cdga)$.
  \item $\mathcal{CH}_X$ lifts as an $(\infty,1)$-functor $\mathcal{CH}_X:\hcdga \to \mathop{PFac}_{X}(\hcdga)$.
  \end{enumerate}
 \end{lemma}  
\begin{proof}
 First we note that $CH^\com_\emptyset (A)\cong k$ and that the maps $\mu_{U_1,\dots, U_n,V}$ are CDGAs morphisms. Further, if $V_1,\dots, V_l$ is a collection of pairwise disjoint open subsets of $V\in Op(X)$ and $U_1,\dots, U_n$ is another family of pairwise disjoint open subsets of $V$ such that  each $U_i$ is contained in some $V_j$, we can form the diagram
 \begin{equation*}
 \xymatrix{ \bigotimes_{j=1}^l \left(\bigotimes_{U_i \subset V_j} CH_{U_i}^{\com}(A)\right) \ar[rr]^{\quad \qquad \mu_{U_1,\dots, U_n, V}} \ar[d]_{ \bigotimes_{j=1}^l\Big(\mu_{(U_i\subset V_j), V_j}\Big)} && CH_{V}^{\com}(A)\\
 \bigotimes_{j=1}^l \big( CH_{V_j}^{\com}(A)\big) \ar[urr]_{\quad \mu_{V_1,\dots, V_j,V}} && } 
 \end{equation*}
which is commutative by functoriality of Hochschild chains. This proves that the rule $U\mapsto CH_U^\com(A)$ is a prefactorization algebra with value in the category $\cdga$. The naturality follows from the naturality of Hochschild chains in the algebra variable and the lift to the $(\infty,1)$-framework follows from (the proof of) Proposition~\ref{P:Topderived} and Theorem~\ref{T:derivedfunctor}.
\end{proof}

In particular, the prefactorization algebra $U\mapsto CH_{U}^{\com}(A)$ is a \emph{commutative} prefactorization algebra.

\smallskip

Following the terminology of~\cite{CG,Co}, we said that an open cover $\mathcal{U}$ of $U\in Op(X)$ is \textbf{factorizing} if,  for all finite collections $x_1,\dots, x_n$ of distincts points in $U$, there are \emph{pairwise disjoint} open subsets $U_1,\dots, U_k$ in $\mathcal{U}$ such that $\{x_1,\dots, x_n\} \subset \bigcup_{i=1}^k U_i$. To define a factorization algebra, we need to introduce the \v{C}ech complex of a prefactorization algebra $\mathcal{F}$. 
Let $\mathcal{U}$ be a cover and denote $P\mathcal{U}$ the set of finite pairwise disjoint open subsets $\{U_1,\dots,U_n \, ,\, U_i\in \mathcal{U}\}$. Now the \textbf{\v{C}ech complex $\check{C}(\mathcal{U},\mathcal{F})$} is the chain (bi-)complex
$$\check{C}(\mathcal{U},\mathcal{F})= \bigoplus_{P\mathcal{U}} \mathcal{F}(U_1) \otimes \cdots \otimes \mathcal{F}(U_n)  \leftarrow \bigoplus_{P\mathcal{U} \times P\mathcal{U}} \mathcal{F}({U_1}\cap {V_1}) \otimes \cdots \otimes \mathcal{F}({U_n}\cap {V_m})  \leftarrow \cdots$$
where the horizontal arrows are induced by the alternate sum of the natural inclusions as for the usual \v{C}ech complex of a cosheaf (see~\cite{CG}). Let us introduce a convenient notation for the \v{C}ech complex: given $\alpha_1,\dots,\alpha_k \in P\mathcal{U}$, we denote $$\mathcal{F}(\alpha_1,\dots,\alpha_k) = \bigotimes_{U_{i_1}\in \alpha_1,\dots, U_{i_k}\in \alpha_k} \mathcal{F}(U_{i_1}\cap \cdots \cap U_{i_k}),.$$
The prefactorization algebra structure yields,  for all $j=1,\dots, k$,  natural maps $\mathcal{F}(\alpha_1,\dots,\alpha_k)\to \mathcal{F}(\alpha_1,\dots,\widehat{\alpha_j},\cdots,\alpha_k)$. 
The \v{C}ech complex of $\mathcal{F}$ can be simply written as \begin{equation}\label{eq:Cechcomplex}\check{C}(\mathcal{U},\mathcal{F})=\bigoplus_{k>0} \bigoplus_{\alpha_1,\dots,\alpha_k \in P\mathcal{U}} \mathcal{F}(\alpha_1,\dots,\alpha_k) [k-1].\end{equation}
The prefactorization algebra structure also induce a canonical map $\check{C}(\mathcal{U},\mathcal{F})\to \mathcal{F}(U)$.

A prefactorization algebra $\mathcal{F}$ on $X$ (with value in $\cdga$ or $k\textit{-}Mod$) is said to be a \textbf{factorization algebra} if, for all open subset $U\in Op(X)$ and every factorizing cover $\mathcal{U}$ of $U$, the canonical map
$$\check{C}(\mathcal{U},\mathcal{F})\to \mathcal{F}(U)$$ is a quasi-isomorphism (see~\cite{Co,CG}). 
 When $\mathcal{F}$ is a commutative factorization algebra, the sequence $\bigoplus_{\alpha_1,\dots,\alpha_k \in P\mathcal{U}} \mathcal{F}(\alpha_1,\dots,\alpha_k) [k-1]$ is naturally a simplicial CDGA and thus the \v{C}ech complex $\check{C}(\mathcal{U},\mathcal{F})$ has a natural structure of CDGA.

\smallskip

Note that $X$ itself, is always a factorizing cover. A  Hausdorff space usually admits many different factorizing covers. This is in particular true for manifolds. Indeed, choosing a Riemannian metric on a manifold $X$ yields a nice factorizing cover  given by the set of geodesically convex neighborhoods of every point in $X$. 

It is shown in~\cite{CG} that, if $\mathcal{U}$ is a \emph{basis}  for the topology of a space $X$ which is also a \emph{factorizing cover}, and $\mathcal{F}$ is a \emph{$\mathcal{U}$-factorization algebra}, then one obtains a factorization algebra $i_*^{\mathcal{U}}(\mathcal{F})$ on $X$ defined by \begin{equation}\label{eq:DFacAlginducedbyaBasis} i_*^{\mathcal{U}}(\mathcal{F})(V):=\check{C}(\mathcal{U}_V,\mathcal{F})\end{equation} where $\mathcal{U}_V$ is the cover of $V$ consisting of open subsets in $\mathcal{U}$ which are also subsets of $V$.
We recall that a \textbf{$\mathcal{U}$-factorization algebra} is like a factorization algebra, except that $\mathcal{F}(U)$ is only defined for $U\in \mathcal{U}$ and further that we only require a quasi-isomorphism $\check{C}(\mathcal{V},\mathcal{F}) \stackrel{\sim}\to \mathcal{F}(U) $ for factorizing covers $\mathcal{V}$ of $U$ consisting of open sets in $\mathcal{U}$.

A \textbf{factorizing good cover}  is a good cover which is also a factorizing cover. For instance, any CW-complex has a factorizing good cover. Admitting a basis of factorizing good cover is a sufficient condition to prove that the Hochschild prefactorization algebra $\mathcal{CH}_X$ is a factorization algebra:

\begin{theorem}\label{T:CH=FH} Let $X$ be a topological space with a factorizing good cover and $A$ be a CDGA.
 Assume further that there is a basis of open sets in $X$ which is also a factorizing good cover.
The prefactorization algebra  $\mathcal{CH}_X: U\mapsto CH_U^\com(A)$ given by Lemma~\ref{P:prefactorization} is a \emph{factorization algebra} on $X$. 

In particular, for any factorizing cover $\mathcal{U}$ of $X$, there is a canonical equivalence of CDGAs $$CH_{X}^{\com}(A) \cong  \check{C}(\mathcal{U},\mathcal{CH}_X).$$
 \end{theorem}
For instance the theorem applies to all manifolds (that we always assume to be paracompact) and more generally to  CW-complexes.
\begin{proof}
 Let $\mathcal{U}$  be  a factorizing  good cover.
We first prove that the rule  $U\mapsto CH_{U}^{\com}(A)$ is a \emph{$\mathcal{U}$-factorization algebra}. 
Since, we already know that $U\mapsto CH_{U}^{\com}(A)$ is a prefactorization algebra (Lemma~\ref{P:prefactorization}), we only need to prove that, for any $U\in \mathcal{U}$ and any factorizing cover $\mathcal{V}$ of $U$ consisting of open sets in $\mathcal{U}$, the canonical map  
$\check{C}(\mathcal{V},\mathcal{CH}_X) \stackrel{\sim}\to CH_{U}^{\com}(A)$ 
is a quasi-isomorphism. {Since $U$ is contractible, $CH^U_\com(A)\cong A$ by Theorem~\ref{T:derivedfunctor}.} Let us denote $P\mathcal{V}$ the set of finite pairwise disjoint open subsets $\{(U_1,\dots,U_n \, ,\, U_i\in \mathcal{V}\}$. Now the \v{C}ech complex $\check{C}(\mathcal{V},\mathcal{CH}_X)$ is the chain (bi-)complex
$$ \bigoplus_{P\mathcal{V}} CH_{U_1}^{\com}(A) \otimes \cdots \otimes CH_{U_n}^{\com}(A)  \leftarrow \bigoplus_{P\mathcal{V} \times P\mathcal{V}} CH_{{U_1}\cap {V_1}}^{\com}(A) \otimes \cdots \otimes CH_{{U_n}\cap {V_m}}^{\com}(A)  \leftarrow \cdots$$
where the horizontal arrows are induced by the alternate sum of the natural inclusions (see~\cite{CG}). Since $\mathcal{U}$ is a good cover, Theorem~\ref{T:derivedfunctor} and the prefactorization algebra structure of $\mathcal{CH}_X$ gives a natural equivalence of chain complexes
\begin{equation}\label{eq:diagFH=HH} {\small
\xymatrix@R=2pc{ \mathop{\bigoplus}\limits_{P\mathcal{V}} CH_{U_1}^{\com}(A) \otimes \cdots \otimes CH_{U_n}^{\com}(A)  \ar[d]_{\simeq} & \bigoplus\limits_{P\mathcal{V} \times P\mathcal{V}} \Big(\bigotimes CH_{{U_i}\cap {V_j}}^{\com}(A) \Big) \ar[l] \ar[d]_{\simeq} & \ar[l]\cdots \\
\bigoplus\limits_{P\mathcal{V}} A\otimes \cdots \otimes A & \bigoplus\limits_{P\mathcal{V} \times P\mathcal{V}} A\otimes \cdots \otimes A \ar[l]& \cdots \ar[l] }}
\end{equation}
We can form a simplicial set  $N_\com(\mathcal{V})$ given by the nerve of the cover $\mathcal{V}$. Since $\mathcal{V}$ is factorizing, the canonical map \begin{equation} \label{eq:diagFH=HH2}\colim_{\scriptsize \begin{array}{l} K_\bullet\stackrel{disj}\hookrightarrow N_\com(\mathcal{V}) \\
K_\bullet\mbox{ finite}\end{array}} \!\!\!\!CH_{K_\bullet}^{\com}(A) \;\; \; \longrightarrow \colim_{\scriptsize \begin{array}{l} K_\bullet\to N_\com(\mathcal{V})\\
K_\bullet\mbox{ finite}\end{array}} \!\!CH_{K_\bullet}^{\com}(A)
\cong CH_{N_\com(\mathcal{V})}(A)\end{equation} (where the the left colimit is over maps whose images are required to be disjoint open subsets) is an equivalence.  
The bottom line of diagram~\eqref{eq:diagFH=HH} now identifies with the left colimit of the map~\eqref{eq:diagFH=HH2}, hence with Hochschild chain complex of nerve $N_\com(\mathcal{V})$. 
 Since $\mathcal{U}$ is a good cover, the intersections $(U_1 \coprod \cdots \coprod U_n) \cap (V_1\coprod \cdots \coprod V_m)$ are contractible. Thus by the Nerve Theorem (or Leray acyclic cover), the geometric realization of $N_\com(\mathcal{V})$ is quasi-isomorphic  to the reunion $\bigcup_{\mathcal{V}} U_i = U$. Since $U$ is assumed to be contractible, we get from above  and Proposition~\ref{P:homologyinvariance}
  a natural equivalence (of prefactorization algebras) $ \check{C}(\mathcal{V},\mathcal{CH}_X)(U) \stackrel{\simeq}\to  CH_{U}^{\com}(A) \,(\cong A)$. Thus  the $\mathcal{U}$-prefactorization algebra $U\mapsto CH_{U}^{\com}(A)$ is a $\mathcal{U}$-factorization algebra. We denote $CH_{\mathcal{U}}$ this $\mathcal{U}$-factorization algebra.

\smallskip

To conclude, we are left to prove that the induced factorization algebra $i_*^{\mathcal{U}}(CH_{\mathcal{U}})$ on $X$ is equivalent to $\mathcal{CH}_X$ as a prefactorization algebra.  Let $V$ be an open subset of $X$.  By~\cite{CG}, there is a natural equivalence $i_*^{\mathcal{U}}(CH_{\mathcal{U}})(V)\cong \check{C}(\mathcal{U}_V,CH_{\mathcal{U}})$ (where $\mathcal{U}_V$ is the cover of $V$ consisting of open subsets in $\mathcal{U}$ which are also subsets of $V$). Since the cover $\mathcal{U}_V$ is a good cover, as in the case where $V$ was in $\mathcal{U}$ above, there is a natural equivalence
$$ 
{\small
\xymatrix@R=2pc{ \mathop{\bigoplus}\limits_{P\mathcal{U}_V} CH_{U_1}^{\com}(A) \otimes \cdots \otimes CH_{U_n}^{\com}(A)  \ar[d]_{\simeq} & \bigoplus\limits_{P\mathcal{U}_V \times P\mathcal{U}_V} \Big(\bigotimes CH_{{U_i}\cap {V_j}}^{\com}(A)\Big)   \ar[l] \ar[d]_{\simeq} & \ar[l]\cdots \\
\bigoplus\limits_{P\mathcal{U}_V} A\otimes \cdots \otimes A & \bigoplus\limits_{P\mathcal{U}_V \times P\mathcal{U}_V} A\otimes \cdots \otimes A \ar[l]& \cdots \ar[l] }}
$$
where, as for diagram~\eqref{eq:diagFH=HH} above, the bottom line is  equivalent to the Hochschild chain complex $CH_{N_\com(\mathcal{U}_V)}^\com(A)$ of the simplicial set $N_\com(\mathcal{U}_V)$ given by the nerve of $\mathcal{U}_V$.   Since $\mathcal{U}$ is a good cover, we can again use the Nerve Theorem to see that 
$CH_{N_\com(\mathcal{U}_V)}^\com(A)\cong CH_{\bigcup_{\mathcal{U}_V}U_i}^\com(A)\cong CH_{V}^\com(A)$ and thus to
get  the  natural equivalence (of prefactorization algebras) $ \check{C}(\mathcal{V},i_*^{\mathcal{U}}(CH_{\mathcal{U}}))(U) \stackrel{\simeq}\to CH_{V}^{\com}(A)\cong \mathcal{CH}_X(V)$. 
\end{proof}

If $\mathcal{F}$ is a factorization algebra on $X$ (with value in $k\textit{-}Mod$ or $\hkmod$), and $f:X\to Y$ is a continuous map, one can define the \textbf{pushforward $f_*(\mathcal{F})$} by the formula $f_*(\mathcal{F})(V) =\mathcal{F}(f^{-1}(V))$ which actually is a factorization algebra on $Y$,  see~\cite{CG}. Costello and Gwilliam~\cite[Section 3.a]{CG} have defined  the \textbf{factorization homology} $HF(\mathcal{F})$ of $\mathcal{F}$ as the pushforward $p_*(\mathcal{F})$ where $p:X\to pt$ is the unique map. In other words, we have natural equivalences \begin{equation}HF(\mathcal{F})\;\cong\; p_*(\mathcal{F})\;\cong\; \mathcal{F}(X)\;\cong\; \check{C}(\mathcal{U},\mathcal{F})\end{equation} in $\hkmod$ (for any factorizing cover $\mathcal{U}$ of $X$). Note that, despite its name, $HF(\mathcal{F})$ is a cochain complex (up to equivalences) and, in particular  a \emph{derived} object (which may be thought as the \lq\lq{}derived global sections\rq\rq{} of $\mathcal{F}$). If $\mathcal{F}$ has value in $\cdga$, then $HF(\mathcal{F})$ is an object in $\hcdga$ too. 
Theorem~\ref{T:CH=FH} and Lemma~\ref{P:prefactorization} yields
\begin{corollary} \label{C:HFoCH=CH}
Let $X$ be a CW-complex. 
The Hochschild prefactorization functor $\mathcal{CH}_X$ is actually a functor 
$\mathcal{CH}_X: \hcdga \to \mathop{Fac}_X(\hcdga)$.
Further, the factorization homology of $\mathcal{CH}_X(A)$ is equivalent to  $CH^X_\com(A)$ (as an object of $\hcdga$); in other words the following diagram commutes:
$$\xymatrix{\hcdga \ar[rr]^{CH_{X}^{\com}(-)} \ar[d]_{\mathcal{CH}_X} & &\hcdga  \\ 
\mathop{Fac}_X(\hcdga) \ar[rru]_{HF(-)} && } $$
\end{corollary}
For manifolds, we will give below another geometric interpretation of the functor $\mathcal{CH}_X$
in terms of embeddings of manifolds in euclidean spaces (see Example~\ref{E:cstFac}, Corollary~\ref{C:HFact=CH} and Remark~\ref{R:cstFac}).

\smallskip

A factorization algebra $\mathcal{F}$ on a manifold is said to be \textbf{locally constant}, if, the  natural map $ \mathcal{F}(U) \to \mathcal{F}(V)$ is a quasi-isomorphism when $U\subset V$ are homeomorphic to a ball (see~\cite{CG,Co}).
Furthermore,  we call $\mathcal{F}$ a \textbf{commutative constant factorization algebra} (on $X$), if there is a CDGA $A$, and natural  quasi-isomorphisms $ \mathcal{F}(U) \to A$ for any open $U\subset X$ homeomorphic to a ball.  
Here, natural means, that for any pairwise disjoint open subsets homeomorphic to a ball $U_1,\dots, U_n \in V$ of a contractible open subset $V\in X$ also homeomorphic to a ball, the following diagram is commutative in $\hkmod$
\begin{equation}\label{eq:cstFacAlg}\xymatrix{ \mathcal{F}({U_1})  \otimes \cdots \otimes \mathcal{F}({U_n}) \ar[d] \ar[rr]^{\qquad \quad \mu_{U_1,\dots, U_n,V}}  & & \mathcal{F}({V})  \ar[d]\\ A^{^{\otimes n}} \ar[rr]_{m^{(n)}} & & A}\end{equation}
where $m^{(n)}$ is the $(n-1)$-times iterated multiplication of $A$.
\begin{example}\label{E:cstFac}
 A class of examples of locally constant factorization algebras occurs as follows. By results of Lurie~\cite{L-VI} (also see Proposition~\ref{P:Fac=En} below), the data of a locally constant factorization algebra on $\R^n$ is the same as the data of an $E_n$-algebra. Thus any CDGA yields a locally constant factorization algebra (denoted $\mathcal{A}$ for the moment) on $\R^n$ (for any $n\geq 1$) and also, by restriction,  on any open subset of $\R^n$. 

Now, let $X$ be any manifold, and let $i:X\hookrightarrow \R^n$ be an embedding of $X$ in $\R^n$. Let $NX$ be an open tubular neighborhood of $X$ in $\R^n$. We write $p:NX\to X$ for the bundle map. Any factorization algebra $\mathcal{F}$ on $\R^n$ restricts to a factorization algebra $\mathcal{F}_{|NX}$ on $NX$ and  the pushforward $p_*(\mathcal{F}_{|NX})$ is a factorization algebra on $X$. Thus, a CDGA $A$ yields a   locally constant factorization algebra $p_*(\mathcal{A}_{|NX})$ on $X$ for any manifold $X$. Since $p_*(\mathcal{F}_{|NX})(U)\cong \mathcal{F}(p^{-1}(U))$, it is easy to check that   $p_*(\mathcal{A}_{|NX})$ is indeed locally constant. Since by construction there is a natural quasi-isomorphism $\mathcal{A}(B) \stackrel{\sim}\to A=\mathcal{A}(\R^n)$ for any open ball $B\subset \R^n$, the induced locally constant factorization algebra $p_*(\mathcal{A}_{|NX})$ satisfies that there exists a natural equivalence $\mathcal{A}_{|NX}(B)\stackrel{\sim} \to A$ for any open set $B\in Op(X)$ homeomorphic to a ball.
In other words, we can think to the factorization algebra $p_*(\mathcal{A}_{|NX})$ as being \emph{constant} (and commutative). 

\smallskip

Note that the previous analysis can be extended to any space $X$ which embeds as the base of locally trivial fibration $U\to X$ where $U$ is an open in some $\R^n$ and the fibers are homeomorphic to a ball. 
\end{example}

\begin{example}
 Let $G$ be  a discrete group acting properly discontinuously on a manifold $X$. According to~\cite{CG}, any $G$-equivariant factorization algebra $\mathcal{F}$ on $X$ yields a factorization algebra $\mathcal{F}^G$ on $X/G$. Further, it is easy to check that, if $\mathcal{F}$ is locally constant then so is $\mathcal{F}^G$ too. In particular any CDGA $A$ yields a locally constant factorization algebra on $\R^n/G$ 
for any discrete group acting  properly discontinuously on $\R^n$.
\end{example}

The next corollary describes what \emph{constant} commutative factorization algebra are; namely they all are equivalent to  derived Hochschild chains for some CDGA.

\begin{corollary}\label{C:HFact=CH} Let $X$ be a manifold and $\mathcal{F}$ a commutative factorization algebra  such that there exists a CDGA $A$ and a natural equivalence $\mathcal{F}(B)\stackrel{\sim} \to A$ for any open set $B\in Op(X)$ homeomorphic to a ball. 
 Then $\mathcal{F}$ is equivalent to the Hochschild chain factorization algebra $\mathcal{CH}_X(A)$ (given by Lemma~\ref{P:prefactorization}).
\end{corollary}
In particular, there is a natural equivalence $HF(\mathcal{F}) \cong CH_{X}^{\com}(A)$.
\begin{proof}
 Choosing a Riemannian metric on $X$, the set $\Ball^g(X)$ of geodesic open balls is \emph{factorizing}. Further, $\mathcal{F}$ being constant, we have, for any $\alpha=\{B_1,\dots, B_k\}\in \Ball^g(X)$,  natural equivalences
 \begin{eqnarray*}
 \mathcal{F}(\alpha)=\mathcal{F}(B_1)\otimes \cdots \otimes \mathcal{F}(B_k) &\cong & \bigotimes_{i=1}^k A 
\cong  CH_{B_1}^\com(A) \otimes \cdots \otimes CH_{B_k}^\com(A) 
\cong \mathcal{CH}_X(\alpha)\end{eqnarray*}
Similarly, there are natural equivalences $$\mathcal{F}(\alpha_1,\dots,\alpha_j)\; \cong \; \mathcal{CH}_X(\alpha_1,\dots,\alpha_j)$$ for any  $\alpha_1,\dots,\alpha_j\in P\Ball^g(X)$.
Thus, by Theorem~\ref{T:CH=FH}, we have
\begin{eqnarray*}
 HF(\mathcal{F})&\cong & \check{C}(\Ball^g(X),\mathcal{F}) \\ 
&\cong & \bigoplus_{k>0} \bigoplus_{\alpha_1,\dots,\alpha_k \in P\Ball^g(X)} \mathcal{F}(\alpha_1,\dots,\alpha_k) [k-1]\\
 &\cong &\bigoplus_{k>0} \bigoplus_{\alpha_1,\dots,\alpha_k \in P\Ball^g(X)} \mathcal{CH}_X(\alpha_1,\dots,\alpha_k) [k-1] \; \cong \; CH_{X}^{\com}(A).\end{eqnarray*}  It follows that we have an equivalence $\check{C}(\Ball^g(X),\mathcal{F}) \cong\check{C}(\Ball^g(X),\mathcal{CH}_X)$. The same analysis can be made for any open  subset $U\in Op(X)$ instead of $X$ and the naturality of Hochschild chains ensures that the equivalence $\mathcal{F}(U)\cong CH_U^\com(A)$ is natural in $U$.
\end{proof}
\begin{remark}
Note that the above Corollary~\ref{C:HFact=CH} can be extended  to manifolds with corners, where, by a locally constant factorization algebra on a manifold  with corners, we mean a factorization algebra which is locally constant if, whenever restricted to the strata (which are manifolds), it is locally constant see~\cite{CG,AFT}. One can extend the definition of constant factorization algebra in the same way. 

\smallskip

Let us also sketch how Corollary~\ref{C:HFact=CH} can be used in the general case of \emph{locally constant commutative factorization algebras}.
Let $\mathcal{A}$ be a locally constant factorization algebra on a manifold $M$ and assume that  there is a codimension $1$ submanifold (possibly with corners) $N$ of $M$  with a trivialization $N\times I$ of its neighborhood  such that $M$ is decomposable as $M=X\cup_{N\times I}Y$ where $X,Y$ are submanifolds (with corners) of $M$ glued along  $N \times I$. The inclusion $i:N\times I\to X$ induces a map of factorization algebras $i_\ast(\mathcal{A}_{|N\times I}) \to \mathcal{A}_{|X}$, which gives a structure of $\mathcal{A}_{|N\times I}$-module to $\mathcal{A}_{|X}$ since $\mathcal{A}$ is commutative.
\begin{lemma}\label{L:HFact=CHloc} If $\mathcal{A}_{|Y}$ is constant (say $\mathcal{A}_{|Y}(B)\cong A$ for a \cdga{ } $A$ and any ball $B\in Op(Y)$),  $\mathcal{A}$ is equivalent to $\mathcal{A}_{|X} \mathop{\otimes}\limits^{\mathbb{L}}_{\mathcal{CH}_{N\times I}(A)} \mathcal{CH}_{Y}(A)$
in $\mathop{Fac}_M(\cdga_\infty)$ (where the factorization algebras are pushforward to $M$ along the natural inclusions). In particular,
$$HF(\mathcal{A})\; \cong\; HF(\mathcal{A}_{|X})\mathop{\otimes}^{\mathbb{L}}_{CH_{N\times I}^\com(A)} CH_{Y}^\com(A). $$
\end{lemma}
Using a handle decomposition of $M$, one can use the lemma above to compute the homology of a locally constant commutative factorization algebra $\mathcal{A}$ in terms of (iterated) derived tensor products of derived Hochschild functors (see Section~\ref{S:locality} for a related construction).
\begin{proof} It essentially follows from Corollary~\ref{C:HFact=CH} and Lemma~\ref{L:Enmodule}.
{Indeed, since $\mathcal{A}_{|Y}$ is constant of type $A$, it is isomorphic to $\mathcal{CH}_{Y}(A)$ and similarly, $\mathcal{A}_{|N\times I}\cong \mathcal{CH}_{N\times I}(A)$ (by Corollary~\ref{C:HFact=CH}). Further, there is a quasi-isomorphism of prefactorization algebras:
\begin{equation} \label{eq:HFact=CHloc}\mathcal{A}_{|X} \mathop{\otimes}\limits^{\mathbb{L}}_{\mathcal{CH}_{N\times I}(A)} \mathcal{CH}_{Y}(A) \; \cong \; \mathcal{A}_{|X} \otimes \bigoplus_{\ell\geq 0} \mathcal{A}_{|N\times I}[\ell] \otimes \mathcal{A}_{|Y}  \end{equation}
where the middle term is the bar construction on $\mathcal{A}_{|N\times I}$. Applying the same construction to the \v{C}ech complex of a factorizing cover yields that $\mathcal{A}_{|X} \mathop{\otimes}^{\mathbb{L}}_{\mathcal{CH}_{N\times I}(A)} \mathcal{CH}_{Y}(A)$ is indeed a factorization algebra.

By the identity~eq\ref{eq:HFact=CHloc} above, for any open $U$ in $X$ or open $V$ in $Y$, we have,  $$\mathcal{A}_{|X} \mathop{\otimes}^{\mathbb{L}}_{\mathcal{CH}_{N\times I}(A)} \mathcal{CH}_{Y}(A)\,\Big(U\Big) \;\cong \;\mathcal{A}_{|X}\big(U\big), \quad \mathcal{A}_{|Y} \mathop{\otimes}^{\mathbb{L}}_{\mathcal{CH}_{N\times I}(A)} \mathcal{CH}_{Y}(A)\,\Big(V\Big) \;\cong \;\mathcal{A}_{|Y}\big(V\big)$$
Similarly, using that the two-sided bar construction of $CH_{W}^\com(A)$ is naturally isomorphic to $CH_{W}^\com(A)$ for any open $W$ in $N\times I$, we get that the restriction (to $N\times I$) $\big(\mathcal{A}_{|X} \mathop{\otimes}^{\mathbb{L}}_{\mathcal{CH}_{N\times I}(A)} \mathcal{CH}_{Y}(A)\big)_{|N\times I}$  is isomorphic to $\mathcal{CH}_{N\times I}(A)$. Thus  the factorization algebra $\mathcal{A}_{|X} \mathop{\otimes}^{\mathbb{L}}_{\mathcal{CH}_{N\times I}(A)} \mathcal{CH}_{Y}(A)$ is equivalent to the one obtained by gluing together $\mathcal{A}_{X}$ and $\mathcal{A}_{Y}$ along their intersection $\mathcal{A}_{N\times I}$, that is $\mathcal{A}$.}
\end{proof} 
\end{remark}

\begin{remark}\label{R:cstFac}
  Corollary~\ref{C:HFact=CH} implies that the factorization algebras $p_*(\mathcal{A}_{|NX})$ are independent (up to equivalences in $\mathop{Fac}_X(\hkmod)$) of the choices of the embedding and of the tubular neighborhood  made in Example~\ref{E:cstFac}.  Indeed, if $i_1:X\hookrightarrow \R^{n_1}$ and $i_2:X\hookrightarrow \R^{n_2}$ are two embeddings of a manifold $X$ in an euclidean space, and given two choices $\R^{n_1}\supset N_1 X \stackrel{p_1}\to X$,$\R^{n_2}\supset N_2 X \stackrel{p_2}\to X$ of tubular neighborhoods, Corollary~\ref{C:HFact=CH} implies that there are  natural equivalences   $${p_1}_*(\mathcal{A}_{|N_1 X})(U)\stackrel{\simeq} \longrightarrow CH^U_\com(A)  \stackrel{\simeq}\longleftarrow {p_2}_*(\mathcal{A}_{|N_2 X})(U)$$ for open subsets $U\subset X$.
\end{remark}

\begin{example}\label{E:Surface}
 Let $M$ be a manifold and $A=C^{\infty}(M)$ its algebra of functions. Also let $\Sigma^g$ be a compact Riemann surface of genus $g$ and $\mathcal{F}$  the factorization algebra (see Theorem~\ref{T:CH=FH}) on $\Sigma^g$ defined by the rule $U\mapsto C_{U}^{\com}(C^{\infty}(M))$ (here, in the definition of Hochschild chains, the tensor product over $k$ is the \emph{completed} tensor product so that $C^\infty(M)\otimes C^\infty(M)\cong C^\infty(M\times M)$). Let $\Omega^n(M)$ denote the de Rham $n$-forms on $M$, viewed as complex concentrated in degree $0$.
 
 
An analogue of Hochschild-Kostant-Rosenberg theorem for Hochschild chains over surfaces~\cite[Theorem 4.3.3]{GTZ} implies that the factorization homology of $\mathcal{F}$ on  $\Sigma^g$ is given by
$$HF(\mathcal{F}) \cong S_{C^\infty(M)}\Big(\Omega^1(M)[2]\Big)\mathop{\otimes}_{C^\infty(M)} S_{C^\infty(M)}\Big(\Omega^1(M)\Big) \mathop{\otimes}_{C^\infty(M)} \!\bigotimes_{C^\infty(M)}^{i=1\dots 2g}\! \Omega^{\com}(M)[\bullet]$$ where $V[n]$ is the graded space $(V[n])^i=V^{i+n}$ (\emph{i.e.} with cohomological degree shifted down by $n$) and $S_{C^\infty(M)}(W)$ is the symmetric graded algebra of a graded $C^\infty(M)$-module $W$. In terms of graded-geometry the above isomorphism is equivalent to saying that $HF(\mathcal{F})$ is (equivalent to) the algebra of functions of the graded manifold $$HF(\mathcal{F}) \; \cong \;T[2](M)\oplus  \bigoplus_{i=1}^{2g} T[1]M. $$
\end{example}

   We now study a homotopical strengthening of the locally constant condition. We said that a factorization algebra on $X$ is \textbf{strongly locally constant} if the natural map $  \mathcal{F}(U) \to \mathcal{F}(V)$ is a  quasi-isomorphism when $U\subset V$ are contractibles. Let $A$ be a CDGA; we say that a factorization algebra $\mathcal{F}$ is \textbf{strongly constant of type $A$}, if there exists a natural quasi-isomorphism $\mathcal{F}(U) \to A$ for any contractible $U$. Here, natural means, that for any pairwise disjoint contractible open subsets $U_1,\dots, U_n \in V$ of a contractible open subset $V\in X$, the following diagram (similar to diagram~\ref{eq:cstFacAlg}) is commutative in $\hkmod$
$$\xymatrix{ \mathcal{F}({U_1})  \otimes \cdots \otimes \mathcal{F}({U_n}) \ar[d] \ar[rr]^{\qquad \quad \mu_{U_1,\dots, U_n,V}}  & & \mathcal{F}({V})  \ar[d]\\ A^{^{\otimes n}} \ar[rr]_{m^{(n)}} & & A}$$

\begin{example}
 The factorization algebras given by Theorem~\ref{T:CH=FH} are strongly constant of type $A$. 
\end{example}

\begin{example}\label{E:stcstFac}
Let $X$ be a manifold, $A$ a CDGA and $p_*(\mathcal{A}_{|NX})$ be the factorization algebra on $X$ constructed in Example~\ref{E:cstFac}.  Corollary~\ref{C:HFact=CH} {and the homotopy invariance of Hochschild chains} implies that $p_*(\mathcal{A}_{|NX})$ is a \emph{strongly constant} factorization algebra.

\smallskip

Now let $X$ be a topological space that embeds  as a  retract $i:X \stackrel{\hookrightarrow}{\leftarrow} UX:r$, where $UX$ is an open subset of some $\R^n$ (where $n$ can be infinite). Then, similarly to the manifold case, the CDGA $A$ yields a factorization algebra $\mathcal{A}$ on $\R^n$ and a factorization algebra $r_*(\mathcal{A}_{|UX})$ on $X$. The above paragraph implies that $\mathcal{A}$ is strongly constant. Further if the fibers of $r$ are contractible, then  $r_*(\mathcal{A}_{|UX})$ is \emph{strongly constant of type $A$}.
\end{example}

\smallskip

We have the following analogue of Corollary~\ref{C:HFact=CH} for strongly constant factorization algebras on   a topological space that admits a factorizing good cover (for instance those given by Example~\ref{E:stcstFac} when $X$ is a CW-complex) 
\begin{corollary}\label{C:CH=FH} Let $X$ be a topological space with a factorizing good cover and $A$ be a CDGA.
 Let $\mathcal{F}$ be a  factorization algebra on $X$. 
\begin{enumerate}
\item Assume $\mathcal{F}$ is strongly constant of type $A$,  then one has a natural equivalence $HF(\mathcal{F})\cong CH_X^\com(A)$  in $\hkmod$.
\item Assume  that there is a basis $\mathcal{B}$ of open sets which is a factorizing good cover and that $\mathcal{F}$ is a factorization algebra which satisfies the strongly constant condition (of type $A$)  with respect to opens in $\mathcal{B}$, \emph{i.e.}, there exists a natural quasi-isomorphism $\mathcal{F}(U) \to A$ for  $U\in \mathcal{B}$ (which is automatically contractible).  Then, there is a natural equivalence of 
 factorization algebras $\mathcal{F} \cong \mathcal{CH}_X(A)$ between $\mathcal{F}$ and the Hochschild prefactorization algebra given by Lemma~\ref{P:prefactorization} (in particular, $\mathcal{F}$ is strongly constant of type $A$).
\end{enumerate}
\end{corollary}
\begin{proof} The first assertion is proved as in Corollary~\ref{C:HFact=CH} using
a factorizing good cover  $\mathcal{U}$ for $X$ instead of $\Ball^g(X)$ (and using a proof similar to  the proof of Theorem~\ref{T:CH=FH} to get that $\bigoplus_{k>0} \bigoplus_{\alpha_1,\dots,\alpha_k \in P\mathcal{U}} \mathcal{CH}_X(\alpha_1,\dots,\alpha_k) [k-1] \; \cong \; CH_{X}^{\com}(A)$).

Since any factorization algebra is uniquely defined by its restriction to a factorizing basis,  the second assertion follows easily from the first one applied to all $V\in \mathcal{B}$.
\end{proof}

As a further corollary to Theorem \ref{T:CH=FH}, we can extend the Hochschild construction as a pullback and pushforward of factorization algebras in a particular setting. The pushforward construction of factorization algebras was described above Corollary \ref{C:HFoCH=CH}. Following \cite{CG}, there is also a \textbf{pullback} construction for factorization algebras given for an \emph{open immersion} $f:N\to M$. For a factorization algebra $\mathcal F$ on $M$, let $f^*\mathcal F$ be the factorization algebra on $N$ given by $f^*\mathcal F(U)=\mathcal F(f(U))$ for all open subsets $U\subset N$ such that $f|_U:U\to f(U)$ is a homeomorphism, extended to a full factorization algebra.

Now, assume that $X, Y, Z$ are topological spaces, and that there is an open immersion of $X\times Y\hookrightarrow Z$ of $X\times Y$ into $Z$. For a factorization algebra $\mathcal F$ on $Z$, we define the Hochschild factorization algebra with respect to $X$ to be the factorization algebra ${\bf CH}_X(\mathcal F)$ of $\mathcal F$ on $Y$ given by
\begin{equation*}
 {\bf CH}_X(\mathcal F) := (pr_Y)_*\circ f^*(\mathcal F), \text{ where } \quad
 Y \stackrel {pr_Y} \longleftarrow X\times Y \stackrel f \longrightarrow  Z.
\end{equation*}
Here $pr_Y:X\times Y\to Y$ denotes the projection to $Y$. This construction induces a functor, called ${\bf CH}_X: \mathop{Fac}_Z(\hcdga) \to \mathop{Fac}_Y(\hcdga) $, $\mathcal F\mapsto (pr_Y)_*\circ f^*(\mathcal F)$ which satisfies the following naturality condition.
\begin{corollary}\label{L:CHYoCHX=CHXoCHZ}
Assume that $X$, $Y$ and $Z$ all admit a basis of open sets which is a factorizing good cover.
Under the functors $\mathcal{CH}_Y: \hcdga \to \mathop{Fac}_Y(\hcdga)$ and $\mathcal{CH}_Z: \hcdga \to \mathop{Fac}_Z(\hcdga)$ from Corollary \ref{C:HFoCH=CH}, the functor ${\bf CH}_X$ represents the functor $CH_X^\bullet$ on $CDGA_\infty$, \emph{i.e.}, the following diagram commutes:
$$\xymatrix{\hcdga \ar[rr]^{CH_{X}^{\com}} \ar[d]_{\mathcal{CH}_Z} & &\hcdga \ar[d]^{\mathcal{CH}_Y} \\ 
\mathop{Fac}_Z(\hcdga) \ar[rr]^{{\bf CH}_X} && \mathop{Fac}_Y(\hcdga) } $$
\end{corollary}
\begin{proof} Let $A\in \hcdga$ and  $\mathcal{W}$ be a basis and factorizing good cover by open subsets $U\times V\subset X\times Y$ such that $U,V$ are contractible and $f|_{U\times V}:U\times V\to f(U\times V)$ is a homeomorphism. In particular, for $U\times V\in \mathcal{W}$, we have natural equivalences
$$f^*( \mathcal{CH}_Z (A)) (U\times V) \cong  \mathcal{CH}_Z (A)(f(U\times V)) \cong CH^{\com}_{f(U\times V)}(A)\cong A$$
since $f(U\times V)$ is contractible. Hence Corollary~\ref{C:CH=FH} implies that 
$f^*( \mathcal{CH}_Z (A))$ is strongly constant of type $A$ and further that, for any
contractible open $U\subset Y$, 
\begin{equation*}
{\bf CH}_X ( \mathcal{CH}_Z (A)) (U) 
\;\cong \; f^*( \mathcal{CH}_Z (A)) (X\times U)\; \cong \;CH_{X\times U}^\com(A)\;\cong \; CH_X^\com(A).
\end{equation*}
Thus $\mathcal{CH}_X ( \mathcal{CH}_Z (A))$ is a strongly constant factorization algebra on $Y$ of type $CH^{\com}_{X}(A)$, hence,  by Corollary~\ref{C:CH=FH}.(2),  is naturally equivalent to $\mathcal {CH}_Y(CH^{\com}_X(A))$. 
\end{proof}

As an application of the above Corollary \ref{L:CHYoCHX=CHXoCHZ}, we look at the particular case where $X\times D^{n-k}\hookrightarrow \R^n$ is an open immersion. In this case, ${\bf CH}_X$ maps factorization algebras in $\R^n$ to factorization algebras in $D^{n-k}$, and thus via the inclusion $D^{n-k}\hookrightarrow \R^{n-k}$ also to factorization algebras in $\R^{n-k}$,
$$ {\bf CH}_X: \mathop{Fac}\, _{\R^{n}}(\hcdga)\to \mathop{Fac}\, _{\R^{n-k}}(\hcdga). $$

\subsection{Locally constant factorization algebras and $E_n$-algebras}

If $\mathcal{A}$ is a locally constant factorization algebra on $M\times D^{n-m}$,  pushforward along the natural projection ${\pi_1}: M\times D^{m-n}\to M$ induces a  factorization algebra ${\pi_1}_*(\mathcal{A})$ on $M$, which is locally constant. Given a monoidal $(\infty,1)$-category $\mathcal{C}$ and a manifold $X$, we denote by \textbf{$Fac_X^{lc}(\mathcal{C})$ the (sub)category of locally constant factorization algebra} on $X$ taking value in $\mathcal{C}$.

We recall the following Proposition which is essentially due to Lurie~\cite{L-VI} and Costello~\cite{Co}.
\begin{proposition} \label{P:Fac=En}
 The pushforward along $p: \R^n\to pt$  induces an equivalence of $(\infty,1)$-categories $$p_*:{\mathop{Fac}}^{lc}_{\R^n}(\hkmod) \stackrel{\simeq}\longrightarrow E_n\textit{-}Alg_\infty.$$
\end{proposition}
\begin{proof} We sketch the proof. Details will appear elsewhere.
 Restricting to open sets homeomorphic to an euclidean disk, we obtain a tautological functor
 $\mathop{for}: {\mathop{Fac}}^{lc}_{\R^n}(\hkmod)\to N(\text{Disk}^{lc}(\R^n))\textit{-}Alg$ where $N(\text{Disk}^{lc}(\R^n))\textit{-}Alg$ stands for the full subcategory spanned by  the locally constant $N(\mathop{Disk}(\R^n))$-algebras in the sense of~\cite[Definition 5.2.4.7]{L-VI}. By ~\cite[Theorem 5.2.4.9 and Example 5.2.4.3]{L-VI}, the latter category is equivalent 
 to $E_n\textit{-}Alg_\infty$ and, under this equivalence, $\mathop{for}(\mathcal{F})$ identifies with the global section $\mathcal{F}(\R^n)\cong p_*(\mathcal{F})$ for any locally constant factorization algebra $\mathcal{F}$.
 
 We define an inverse to $p_*$ as follows.  Let $\mathcal{CVX}$ be the set of open convex subsets of $\R^n$, which is a factorizing basis.  To an $E_n$-algebra $E$, one associates a $\mathcal{CVX}$-factorization algebra $\mathcal{E}$ defined by $C\mapsto \mathcal{E}(C):= E$.  As shown in~\cite{CG}, the $\mathcal{CVX}$-factorization algebra $\mathcal{E}$ has an \emph{unique} extension to a factorization algebra on $\R^n$ iff it satisfies the \v{C}ech condition
$\check{C}(\mathcal{U},\mathcal{E})\to \mathcal{E}(U)$ for any factorizing subcover $\mathcal{U}\subset \mathcal{CVX}$ of a convex open $U$.
This follows again from \cite[Section 5]{L-VI}. Indeed, by Theorem~5.3.4.10 in~\cite{L-VI}, the $E_n$-algebra $E$ gives rise to a factorizable cosheaf $\underline{E}$ on the Ran space $\text{Ran}(\R^n)$. It is easy to check that 
the factorizing cover $\mathcal{U}$ gives rise to a cover of $\text{Ran}(U)$ precisely given by $P\mathcal{U}$.  Since every open set in $\mathcal{U}$ is convex, and $\int_C E \cong E$ for any convex open subset $C\subset U$,  by~\cite[Theorem 5.3.4.14]{L-VI}, we get a canonical equivalence $\check{C}(P\mathcal{U},\underline{E}) \stackrel{\simeq}\longrightarrow \check{C}(\mathcal{U},\mathcal{E})$ which makes the diagram
  $$\xymatrix{  \underline{E}(\text{Ran}(U)) \ar[rd]_{\simeq}&  \check{C}(P\mathcal{U},\underline{E}) \ar[l]_{\simeq} \ar[r]^{\simeq} & \check{C}(\mathcal{U},\mathcal{E}) \ar[ld]\\ & E & }$$
commutative (the top left equivalence follows from the fact that $\underline{E}$ is a cosheaf on $\text{Ran}(U)$). Thus, $\check{C}(\mathcal{U},\mathcal{E})\to \mathcal{E}(U)$ is an equivalence and  $\mathcal{E}$ a $\mathcal{CVX}$-factorization algebra. We denote $q(E):=\mathcal{E}$ the induced factorization algebra over $\R^n$. 

We need to check that $q(E)$ is locally constant. It is sufficient to prove that for any open $D$ homeomorphic to an euclidean disk, the map $\mathcal{E}(D) \to \mathcal{E}(\R^n)\cong E$ is an equivalence. This follows from the Kister-Mazur Theorem~\cite[Theorem 5.2.1.5]{L-VI} which yields an isotopy between the inclusion $D\hookrightarrow \R^n$ and an homeomorphism of $\R^n$.

By construction, $p_*\circ q(E) = p_*(\mathcal{E}) \cong E$. Conversely,  $\mathcal{F}$ being locally constant,  for any convex open set $C$, we have a canonical equivalence $\mathcal{F}(C) \cong \mathcal{F}(\R^n)$. It follows that the $\mathcal{CVX}$-factorization algebra defined by $p_*(\mathcal{F})$ is canonically equivalent to the restriction of $\mathcal{F}$ to convex open sets. By uniqueness of the extension of factorization algebra defined on a factorization basis, $q\circ p_* \cong \text{id}$.
\end{proof}

\begin{lemma}\label{L:EnFact} Let $M$ be a manifold and $\pi_1: M\times \R^d \to M$ the canonical projection. 
The pushforward by $\pi_1$ induces an equivalence of $(\infty,1)$-categories $${\pi_1}_*: {\mathop{Fac}}^{lc}_{M\times \R^d}(\hkmod) \stackrel{\simeq}\longrightarrow {\mathop{Fac}}^{lc}_M({E_d}\textit{-}Alg_\infty)$$
\end{lemma}
In particular, if $\mathcal{F}\in \mathop{Fac}_{M\times \R^d}^{lc}(\hkmod)$, then its factorization homology $$HF(\mathcal{F},M\times \R^d)=p_*(\mathcal{F})(pt)\cong p_* \circ {\pi_1}_*(\mathcal{F})(pt)\cong HF({\pi_1}_*(\mathcal{F}),M)$$ is an $E_d$-algebra (here $p:X\to pt$ is the unique map).
\begin{proof}
Let $\pi_2: M\times \R^d \to \R^d$ be the second canonical projection.
For any open set $U$ in $M$, the restriction $\mathcal{A}_{|U\times \R^d}$ is a (locally constant) factorization algebra and ${\pi_2}_*(\mathcal{A}_{|U\times \R^d})$ is a  factorization algebra on $\R^d$. Note that ${\pi_2}_*(\mathcal{A}_{|U\times \R^d})$ is locally constant. Indeed, let $B$ be  any (open, homeomorphic to a) ball $B\subset \R^n$ and denote $\Ball^g(U)$ the set of geodesic convex open sets in $U$ (for a choice of a metric on $U$). Then the family $\Ball^g(U)\times B$ is a factorizing cover of $U\times B$. Since $\mathcal{A}$ is locally constant, for any inclusion $B\hookrightarrow \tilde{B}$ of open sets (homeomorphic to) balls and geodesic convex open set $O$ in $U$, the structure map $\mathcal{A}(O\times B) \to \mathcal{A}(O\times \widetilde{B})$ is a quasi-isomorphism. Using that  $\mathcal{A}$ is a factorization algebra we get 
\begin{eqnarray*} {\pi_2}_*(\mathcal{A}_{|U\times \R^d})(B)&\cong& \mathcal{A}(U\times B) \\&\cong & \check{C}(\Ball^g(U)\times B,\mathcal{A})\\
&\cong &\bigoplus_{k>0}\bigoplus_{\alpha_1\dots \alpha_k \in P\Ball^g(U)}  \bigotimes_{i_1\in \alpha_1,\dots i_k\in \alpha_k}\mathcal{A}((U_{i_1}\cap \cdots \cap U_{i_k})\times B)[k-1]\\
&\cong &\bigoplus_{k>0}\bigoplus_{\alpha_1\dots \alpha_k \in P\Ball^g(U)}  \bigotimes_{i_1\in \alpha_1,\dots i_k\in \alpha_k}\mathcal{A}((U_{i_1}\cap \cdots \cap U_{i_k})\times \tilde{B})[k-1]\\
&\cong & \check{C}(\Ball^g(U)\times \widetilde{B},\mathcal{A}) \; \cong\; {\pi_2}_*(\mathcal{A}_{|U\times \R^d})(\widetilde{B}),
\end{eqnarray*}
which proves that ${\pi_2}_*(\mathcal{A}_{|U\times \R^d})$ is locally constant (since  the above composition is  the structure map ${\pi_2}_*(\mathcal{A}_{|U\times \R^d})(B)\to {\pi_2}_*(\mathcal{A}_{|U\times \R^d})(\widetilde{B})$).
Hence, by Proposition~\ref{P:Fac=En}, ${\pi_2}_*(\mathcal{A}_{|U\times \R^d})$ is equivalent to an $E_d$-algebra $A_U $ and there are canonical equivalences  $$ A_U \; \cong \; {\pi_2}_*(\mathcal{A}_{|U\times \R^d})(\R^d)\; \cong \; \mathcal{A}(U\times \R^d).$$
Since ${\pi_1}_*(\mathcal{A})(U)= \mathcal{A}(U\times \R^d)\cong A_U$, it follows that for each open subset $U\in Op(M)$, ${\pi_1}_*(\mathcal{A})(U)$ is an $E_d$-algebra.
Since $\mathcal{A}$ is a (pre)factorization algebra on $M\times \R^d$, the structure maps \begin{multline*}\mu_{U_1,\dots, U_k, V}:{\pi_1}_*(\mathcal{A})(U_1) \otimes \cdots \otimes {\pi_1}_*(\mathcal{A})(U_k) \cong \mathcal{A}(U_1\times \R^d)\otimes \cdots \otimes \mathcal{A}(U_k\times \R^d) \\
\longrightarrow \mathcal{A}(V\times \R^d)\cong {\pi_1}_*(\mathcal{A})(V)\end{multline*} are maps of $E_d$-algebras.
Further, since $\mathcal{A}$ is a factorization algebra and locally constant, it follows that ${\pi_1}_*(\mathcal{A})(U)$ belongs to $\mathop{Fac}^{lc}_M({E_d}\textit{-}Alg_\infty)$.

Now we build an inverse of ${\pi_1}_*$. Consider  $\mathcal{B}$ in $\mathop{Fac}^{lc}_M({E_d}\textit{-}Alg_\infty)$. For any $U\in Op(M)$, $\mathcal{B}(U)$ is an $E_d$-algebra (compatible with the prefactorization algebra structure map) and thus defines canonically a locally constant factorization algebra on $\R^d$: $Op(\R^d) \ni V \mapsto \mathcal{B}(U)(V)$. A basis of neighborhood
of $M\times \R^d$ is given by the products $U\times V \in Op(M)\times Op(\R^d)$. In order to extend $\mathcal{B}$ to a factorization algebra on $M\times \R^d$, it is enough (by~\cite{CG}) to prove that the rule $(U\times V)\mapsto \mathcal{B}(U)(V)$ defines an $Op(M)\times Op(\R^d)$-factorization algebra where $Op(M)\times Op(\R^d)$ is the cover of $M\times \R^d$ obtained by taking the products of open sets. 
The latter follows from the fact that the $E_d$-algebra structure is natural with respects to the inclusions $\mu_{U_1\dots, U_n, \tilde{U}}$ of pairwise disjoint open subsets of $\tilde{U}\in Op(M)$. 
Hence the rule $(U\times V)\mapsto \mathcal{B}(U)(V)$ extends to give a factorization algebra $E(\mathcal{B})$ on $M\times \R^d$. It is immediate by construction that $E(\mathcal{B})$ is locally constant and functorial in $\mathcal{B}$.

It remains to prove that $E:\mathop{Fac}^{lc}_M({E_d}\textit{-}Alg_\infty) \to \mathop{Fac}^{lc}_{M\times \R^d}(\hkmod)$ is a natural inverse  to  ${\pi_1}_*$. Recall from above that the $E_d$-algebra structure on ${\pi_1}_*(\mathcal{A})(U)$ is   the one of the  $E_d$-algebra  $\mathcal{A}(U\times \R^d)$, which corresponds to the factorization algebra $V\mapsto {\pi_2}_*(\mathcal{A}_{|U\times \R^d}(V)) \cong \mathcal{A}(U\times V)$. It follows that there is a natural equivalence  $E\big({\pi_1}_*(\mathcal{A})\big) \big(U\times V\big)\cong \mathcal{A}(U\times V)$ in $\hkmod$. Further there are natural equivalences (in $ {E_d}\textit{-}Alg_\infty$) ${\pi_1}_*(E(\mathcal{B}))(U)\cong E(\mathcal{B})(U\times \R^d)\cong \mathcal{B}(U)$. The lemma now follows.
\end{proof}

Let $\mathcal{A}$ be a locally constant factorization algebra on a manifold $M$ and assume that  there is a codimension $1$ submanifold (possibly with corners) $N$ of $M$  with a trivialization $N\times D^1$ of its neighborhood  such that $M$ is decomposable as $M=X\cup_{N\times I}Y$ where $X,Y$ are submanifolds (with corners) of $M$ glued along  $N \times D^1$.  According to Lemma~\ref{L:EnFact} above $HF(\mathcal{A}_{|N\times D^1})$ is an $E_1$-algebra.

\begin{lemma}[Excision for locally constant factorization algebras]\label{L:Enmodule} $HF(\mathcal{A}_{|X})$ and $HF(\mathcal{A}_{|Y})$ are right and left $HF(\mathcal{A}_{|N\times D^1})$-modules and further,
 $$HF(\mathcal{A}) \; \cong \; HF(\mathcal{A}_{|X}) \mathop{\otimes}_{HF(\mathcal{A}_{|N\times D^1})}^{\mathbb{L}} HF(\mathcal{A}_{|Y}).$$
\end{lemma}
\begin{proof}
Since $\mathcal{A}$ is locally constant, the canonical map $\mathcal{A}\Big(X\setminus \big(N\times [t, 1)\big)\Big) \to \mathcal{A}(X) $ is an equivalence for all $t\in D^1$. This follows, since for any open set of the form $U\times (a,b) \subset N\times D^1$, where $U$ is homeomorphic to a ball, there is a natural equivalence $\mathcal{A}(U\times (a,b)) \cong \mathcal{A}(U\times (a',b'))$ for any $a'<a<b<b'$ and that the open sets of the form $U\times (a,b)$ form a factorizing cover of $N\times D^1$.  Similarly, we have natural equivalences of $E_1$-algebras $\mathcal{A}(N\times (a,b))\stackrel{\simeq}\longrightarrow N\times D^1$ (as in the proof of Lemma~\ref{L:EnFact}).
 
Let $U$ be an open set in $X$  and $V_1,\dots, V_k$ be (not necessarily disjoint) open subsets in $N$. Then for any sequence of pairwise disjoint open intervals $I_0,I_1,\dots, I_k$ in $D^1$ (where we assume $I_0=(-1,t_0)$ for some $t_0\in D^1$), the open $V_i \times I_i$ ($i=1\dots k$) are pairwise disjoint and disjoint from $X-\setminus \big(N\times [t_0, 1)$. To shorten notation, we denote $X_{t_0}:=X\setminus \big(N\times [t_0, 1)\big)$  The structure maps of a prefactorization algebras yield a map
\begin{eqnarray*} 
\xymatrix{\mathcal{A}(X) \otimes \mathcal{A}(N\times D^1)^{\otimes n} \ar[rrd] \ar[rr]^{\hspace{-15pt}\cong} && \mathcal{A}(X_{t_0}) \otimes \mathcal{A}(N\times I_1) \otimes \cdots \otimes \mathcal{A}(N\times I_k) \ar[d]^{\mu_{X_{t_0}, N\times I_1\dots N\times I_k, X}}\\ && \mathcal{A}(X)}
\end{eqnarray*}
This map is natural with respect to the prefactorization algebra structure of $\mathcal{A}$ and $\mathcal{A}_{N\times D^1}$, hence induces a structure of right $HF(\mathcal{A}_{N\times D^1})\cong \mathcal{A} ({N\times D^1})$-module on $HF(\mathcal{A}_{|X})\cong \mathcal{A}(X)$. A similar argument applies to $HF(\mathcal{A}_{|Y})$.

The open sets $X_{t_0}$, $Y_{s}:=Y\setminus \big(N\times (-1,s]\big)$ and $N\times (a,b) $ (where $t_0, s, a<b \in D^1$) forms a factorizing  cover $\mathcal{N}$ of $M$ and we also denote $\widetilde{\mathcal{N}}$ the induced factorizing cover of $N\times D^1$. A finite sequence of pairwise disjoint open sets in $\mathcal{N}$ is just a sequence $X_{to},  N\times (t_1,t_2),\dots, N\times (t_m, t_{m+1}), Y_{t_m}$ for $-1<t_0<\cdots<t_{m+2}<1$. Note that the complexes $\mathcal{A}\Big(N\times (t_i, t_{i+1})\Big)$ are canonically equivalent to ${\pi_2}_\ast(\mathcal{A})\Big((t_i,t_{i+1})\Big)$, where $\pi_2: N\times D^1 \to D^1$ is the projection on the second factor. Since $\mathcal{A}(X_{t_0})\cong \mathcal{A}(X)$, $\mathcal{A}(Y_{t_{m+2}})\cong \mathcal{A}(Y)$, we have
$$HF(\mathcal{A}) \;\cong \; \check{C}(\mathcal{N},\mathcal{A}) \; \cong \; \mathcal{A}(X) \otimes \check{C}(\widetilde{\mathcal{N}},\mathcal{A}_{|N\times D^1})\otimes \mathcal{A}(Y).$$
Note that the canonical map $p:M\to pt$ factors as $M\stackrel{q}\to [-1,1] \to pt$ where $q$ is the map identifying $X\setminus (N\times D^1)$ with $\{-1\}$, $Y\setminus (N\times D^1)$ with $\{1\}$ and projecting $N\times D^1$ onto $D^1=(-1,1)$ by the second projection. Then the factorization homology $HF(\mathcal{A})\cong p_\ast(\mathcal{A}) \cong p_\ast\big(q_\ast(\mathcal{A}) \big)$  identifies with the factorization homology of the locally constant factorization algebra $q_\ast(\mathcal{A})$ on the closed interval $[-1,1]$ and further $\check{C}\big(\mathcal{N},\mathcal{A}\big)\cong \check{C}\big(\mathcal{I},q_\ast(\mathcal{A})\big)$ where $\mathcal{I}$ is the (factorizing) cover of $[-1,1]$ given by the intervals. Thus we are left to the case of a (locally constant) factorization algebra on $[-1,1]$ which assign the $E_1$-algebra $\mathcal{A}(N\times D^1)$ to any open interval $(a,b)$, and the modules $\mathcal{A}(X)$ to $[-1,t)$ and $\mathcal{A}(Y)$ to $(s,1]$ (with respect to the modules structures defined in the beginning of the proof). 
 By~\cite{Co,CG}, its factorization homology is the (derived) tensor product $\mathcal{A}(X) \otimes^{\mathbb{L}}_{\mathcal{A}(N\times D^1)}\mathcal{A}(Y)$. Indeed, the \v{C}ech complex $\check{C}\big(\mathcal{I},q_\ast(\mathcal{A})\big)$ is equivalent to the 
 two-sided Bar construction  $\mathop{Bar}\Big(\mathcal{A}(X), \mathcal{A}(N\times D^1), \mathcal{A}(Y) \Big)$ (strictly speaking after replacing the $E_1$-algebra and modules by differential graded associative ones).
\end{proof}

\section{Relationship with topological chiral homology}\label{S:TCH}

\subsection{Review of topological chiral homology \`a la Lurie}\label{S:ReviewTCH}

Let $A$ be an $E_n$-algebra and $M$ a manifold of dimension $m$ which is \emph{(stably) $n$-framed}, that is a manifold of dimension $m$ equipped with a framing of $M\times D^{n-m}$. The
\textbf{topological chiral homology} of $M$ with coefficients in $A$ was defined in~\cite{L-TFT,L-VI} and~\cite{Francis} and will be denoted $\int_{M} A$. Note that the above definition \emph{does} depend on the framing in general, even though it is not explicit in the notation\footnote{note that we also do not  write the factor $D^{n-m}$ in the notation}.
 Further $\int_M A$ is an $E_{n-m}$-algebra in general, see~\cite{L-TFT,L-VI}. 
We refer to the aforementioned references as well as to\cite{Francis,Andrade,AFT} for a precise definition. 
If $X$ is a manifold, let  $N(\text{Disj}(X))$ be the  $\infty$-category associated to the poset given by finite disjoint union of open sets in $X$ homeomorphic to an euclidean disk, ordered by inclusion.   According to Lurie~\cite[Remark 5.3.2.7]{L-VI} we have, roughly, that  
\begin{definition}\label{Def:TCHasColimit} Let $M$ be an $n$-framed manifold of dimension $m$ and $A$ an $E_n$-algebra. The \emph{topological chiral homology} $\int_M A$ is the colimit $\colim_{} \psi_M$ with $\psi_M:N(\emph{\text{Disj}}(M\times D^{n-m}))\to \hkmod$ the diagram given by the formula \begin{equation}\label{eq:DefTCH}\psi_M(V_1\cup\cdots\cup V_n)=\int_{V_1}A \otimes \cdots \otimes \int_{V_n} A\cong A\otimes \cdots \otimes A\end{equation} where $V_1,\dots, V_n$ are disjoint open sets homeomorphic to a ball (the latter equivalence follows from \cite[Example 5.3.2.8]{L-VI}). 
\end{definition}

For our purpose, among the properties satisfied by $\int_M A$, we will mainly need  the gluing property given in  Proposition~\ref{P:TCHpushout} below and the fact that $\int_{pt}A \cong A$.  Indeed, the gluing property of topological chiral homology comes from 
the fact that  topological chiral homology   defines an (extended) topological field theory in some appropriate monoidal $(\infty,n)$-category. In view of the cobordism hypothesis, the latter property can  actually be used as a definition of topological chiral homology~\cite[Theorem 4.1.24]{L-TFT}. 

\begin{remark}\label{R:Enoperad} The models for $E_n$-algebras that we are considering  are given by $E_n$-($\infty$-)operads as introduced in~\cite[Section 5.1]{L-VI} in the symmetric monoidal ($(\infty,1)$-)category $(\hkmod,\otimes)$.  The category of $E_n$-algebras is symmetric monoidal (\cite[Section 5.1.5]{L-VI}, \cite[Section 1.8]{L-III}) and, furthermore, 
 there is a commutative diagram of operads
\begin{equation}\label{E:EnCom} \xymatrix{E_1 \ar@{^{(}->}[r] \ar[rrd]_{j_1}&  E_2\ar@{^{(}->}[r] \ar[rd]^{j_2}& \dots \dots \ar@{^{(}->}[r]& E_n\ar@{^{(}->}[r] \ar[ld]^{j_n}& \dots \\
 & & \mathop{Com} & & }\end{equation}
where $\mathop{Com}$ is the operad of commutative (differential graded) algebras such that all maps are monoidals. Note that most models for $E_n$-algebras come with such monoidal properties and also as nested sequences (for instance, this is the case for the models based on the Barratt-Eccles operad~\cite{BF}).  In particular, a commutative differential graded algebra $A$ is naturally an $E_d$-algebra for any integer $d$. We write $j_d^*(A)$ for the $E_d$-algebra structure on $A$ induced by the map of operads $j_d:E_d\to \mathop{Com}$ whenever we want to put emphasis on this $E_d$-algebra structure. 

Likewise, any $E_n$-algebra $A$ is naturally an $E_d$ algebra for $d\leq n$.  Furthermore, we say that an $A$-module $M$ has a \emph{compatible} structure of $E_d$-algebras ($d\leq n$) if the structure maps of the module structure are maps of  $E_d$-algebras, where $A$ is equipped with its natural $E_d$-algebra induced by the diagram of operads~\eqref{E:EnCom}.  
\end{remark}

\begin{remark} 
With the exception of section \ref{SEC:En[M]-alg}, this paper only deals with a fixed $E_n$-algebra $A$, which, when $M$ is framed, is an example of a locally constant $N(\mathop{\mbox{Disk}}(M))$-algebra in the sense of Lurie~\cite{L-VI}, and for which topological chiral homology can be defined, too. In particular, we will show that $A$ also defines canonically a locally constant  factorization algebra on $M$ in the sense of Costello and Gwilliam~\cite{CG,Co}. This also means, that $\int_M A$ computes the global sections of a natural cosheaf defined on the Ran space of $M$ (see~\cite[Section 5.3.2]{L-VI}).
\end{remark}


One of the main consequence of  the interpretation~\cite{L-TFT} of topological chiral homology as an invariant produced by an (extended) topological field theory in some appropriate monoidal $(\infty,n)$-category is the following excision property.
\begin{proposition}[Gluing for topological chiral homology] \label{P:TCHpushout} Let  $M$ be an $n$-framed manifold (possibly with corners) of dimension $m$, (\emph{i.e.} $M\times D^{n-m}$ is framed).  Assume that  there is a codimension $1$ submanifold (possibly with corners) of $M$ of the form $N\times I^{m-1-j}$  (for some $0\leq j\leq m-1$) with a trivialization $N\times I^{m-j}$ of its neighborhood  and that $M$ is decomposable as $M=X\cup_{N\times I^{m-j}}Y$ where $X,Y$ are submanifolds (with corners) of $M$ glued along  $N \times I^{m-j}$.  We endow $X,Y$ and $N$ with the $n$-framing induced from $M$. Let $A$ be an $E_{n}$-algebra. Then 
\begin{itemize}
\item $\int_N A$ is an $E_{n-j}$-algebra.
\item $\int_M A$, $\int_X A$, and $\int_Y A$  are $E_{n-m}$-algebras. Further  $\int_X A$ and $\int_Y A$ are also modules over the $E_{n-j}$-algebra  $\int_N A$. 
\item The above module and algebra structures are compatible. Note that this uses the once and for all fixed telescopic sequence~\eqref{E:EnCom} of models for $E_n$-operads. 
\item There is a natural equivalence of $E_{n-m}$-algebras $$\int_X A\, \mathop{\otimes}\limits^{\mathbb{L}}_{\int_N\! A} \, \int_Y A \stackrel{\simeq}\longrightarrow \int_M A.$$  
\end{itemize}   
\end{proposition}
\begin{proof}
This is explained in~\cite[Section 5.3.4]{L-VI} and~\cite[Section 4.1]{L-TFT}, also see the proof of~\cite{Francis, AFT}. It is also an immediate consequence of Lemma~\ref{L:TCHpushout2} below, in the case where the manifolds are framed since, for a $n$-framed manifold $X$, an $E_n$-algebra yields canonically an $\mathbb{E}[X]$-algebra. 
\end{proof}
 {Let us explain roughly, how   the above Proposition can be seen in terms of extended topological field theory and the cobordism hypothesis. Consider the monoidal $(\infty,n)$-category $E_{\leq n}\textit{-}Alg$ which is fully dualizable~\cite{L-TFT}.  Each object $A$ is an $E_{n}$-algebra and defines a (unique up to equivalence) \emph{extended topological field theory} $\psi_A: \bordfr_{n} \to  E_{\leq n}\textit{-}Alg$ according to~\cite[Theorem 2.4.6]{L-TFT} and \cite[Section 4.1]{L-TFT}. The objects of the $(\infty,n-1)$-category $\mathop{Hom}_{E_{\leq n}\textit{-}Alg}(A,B))$ of morphisms between two $E_n$-algebras $A$, $B$ are $(A,B)$-bimodules equipped with a compatible $E_{n-1}$-algebra structure (in the sense of the telescopic sequence of $E_k$-operads~\eqref{E:EnCom}).  Note that the claims made in Section 4.1 in~\cite{L-TFT} actually essentially follows from~\cite[Theorem 1.2.2]{L-VI} and Sections 2 of~\cite{L-VI}.

The relationship between topological chiral homology with coefficients in $A$ and the field theory $\psi_A$ is as follows. Let $Z$ be an $n$-framed manifold  of dimension $d$ (framing of $Z\times D^{n-d}$). Then $Z$ defines a $d$-morphism in $\bordfr_{n}$, thus $\psi_A(Z)$ is an $E_{n-d}$-algebra and there is a natural equivalence 
 $$\psi_A(Z)\cong \int_Z A$$ see~\cite[Theorem 4.1.24]{L-TFT}.

Now assume  $M$ is a manifold following the assumption of Proposition~\ref{P:TCHpushout}. Thus  $M\times D^{n-m}$ is framed and we  have a decomposition  $M=X\cup_{N\times I^{m-j}}Y$  of $M$ where $X,Y$ are submanifolds of $M$ glued along a submanifold $N\times I^{m-1-j}$ of codimension $1$ in $M$. Then $N$ is a $j$-arrow in $\bordfr_{n}$ and $X,Y$ are $m$-arrows. Hence $\psi_A(N)\cong \int_N A$ is a $j$-arrow in $E_{\leq n}\textit{-}Alg$, hence an $E_{n-j}$-algebra. Similarly, $\psi_A(X)$ and $\psi_A(Y)$ are $E_{n-m}$-algebras. Since $N\times I^{m-1-j}$ is in the boundary of $X$ and $Y$, $\psi_A(X)$ and $\psi_A(Y)$ inherits $\psi_A(N)$-modules structures.  Further $M$ is equivalent to the composition of $X, Y$ along $N$ in the $(\infty,n)$-category $\bordfr_{n}$. Thus it follows that there is a natural equivalence
$\psi_A(X)\otimes_{\psi_A(N)}^{\mathbb{L}} \psi_A(Y)$ in $E_{\leq n}\textit{-}Alg$ since the composition in the $(\infty,n)$-category is induced by derived tensor products. }

\medskip

We finish this Section with the following Lemma.
\begin{lemma} \label{L:TCHcoproduct} Let $A$ be an $E_n$-algebra and $(M_i)_{i\in I}$ a family of $n$-framed manifolds of dimension $m$. There is a natural equivalence of $E_{n-m}$-algebras $$ \colim_{\small \begin{array}{l}F \subset J \\
   F \mbox{ finite}\end{array}
} \left(\bigotimes_{f \in F} \int_{M_f} A \right) \stackrel{\simeq} \longrightarrow \int_{\coprod_{i\in I} M_i} A  .$$
\end{lemma}
\begin{proof}
We set $M=\coprod_{i\in I} M_i$ and, for any finite subset $F$ of $I$, we denote $M_F:= \coprod_{f\in F} M_f$.  The inclusion $M_F\subset M$ yields a canonical map $N({\text{Disj}}(M_F\times D^{n-m})) \to N({\text{Disj}}(M\times D^{n-m}))$. Since an object in $N({\text{Disj}}(M\times D^{n-m}))$ is a \emph{finite} disjoint union of connected open sets in $M$,  every object in $N({\text{Disj}}(M\times D^{n-m}))$ lies in some $N({\text{Disj}}(M_F\times D^{n-m}))$ for a finite $F$. Hence we have an equivalence $N({\text{Disj}}(M\times D^{n-m}))\cong \colim_{F\text{ finite}}N({\text{Disj}}(M_F\times D^{n-m}))$ of $\infty$-categories and, by Definition~\ref{Def:TCHasColimit},  an natural equivalence
$$ \int_M A \cong \colim \psi_M \cong \colim_{F\text{ finite}} \psi_{M_F} \cong  \colim_{F\text{ finite}} \int_{M_F} A.$$ 
Now the lemma follows from~\cite[Theorem 5.3.3.1]{L-VI} which gives a natural equivalence
 $  \left(\bigotimes_{f \in F} \int_{M_f} A \right) \stackrel{\simeq}\to \int_{M_F} A$ for all finite $F$. 
\end{proof}

\subsection{Locality axiom and the equivalence of topological chiral homology with higher Hochschild functor for CDGAs} \label{S:locality}

In view of Proposition~\ref{P:TCHpushout} and Theorem~\ref{T:derivedfunctor}, Morse theory 
(or any triangulation) suggests the following result, which is the main result of this section.
\begin{theorem} \label{T:TCH=CH} Let $M$ be a manifold of dimension $m$ endowed with a framing of 
$M \times D^k$ and $A$ be a differential graded commutative algebra viewed as an $E_{m+k}$-algebra.
 Then, the topological chiral homology of $M$ with coefficients in $A$, denoted by $\int_M A$ is 
equivalent  to $CH^\com_M(A)$ viewed as an $E_k$-algebra (in other words to $j_k^*(CH^\com_M(A))$).
\end{theorem}
In particular, topological chiral homology $\int_M A$ with coefficient in a CDGA $A$ is always 
equivalent to a CDGA and is defined for \emph{any} manifold $M$. 

\smallskip

This theorem is similar to~\cite[Theorem 5.3.3.8]{L-VI} (with the difference that we assume $M$ to be smooth and keep track of the $E_k$-algebra structure). In this section, we wish to prove it 
by using a straightforward geometrical approach based on the gluing property. Indeed, the key idea 
to prove Theorem~\ref{T:TCH=CH} is to use handle decomposition which is very appropriate to deal 
with manifolds and the definition of topological chiral homology. However, note that, with respect
 to Hochschild chains, a representation of $M$ as a CW complex is already nice enough.

Before proving Theorem~\ref{T:TCH=CH}, we recall a few facts on the handle decompositions. Let $M$ be 
a smooth manifold of dimension $m$.  A \emph{handle decomposition} of $M$ is a 
sequence $\emptyset \subset M_0 \subset \cdots \subset M_m =M$, where each $M_j$ is obtained by 
attaching $j$-handles to $M_{j-1}$, see \cite{Mil}. That is gluing a copy 
of $H^j= D^j \times D^{m-j}$ using the attaching map $ S^{j-1} \times D^{m-j} \to \partial M_{j-1}$ 
which is assumed to be an embedding.  In particular all handles of same dimension are attached using 
diffeomorphisms.  

\smallskip

We can achieve such a handle decomposition for $M$ using Morse theory as follows. Let 
$f: M \to \mathbb R$ be a Morse function with critical points $p_1, \cdots, p_k$ numbered in a way 
that $f(p_1) < \cdots < f(p_k)$. Choose $a_0, \cdots, a_k$ such that $a_0 < f(p_1) < a_1 < \cdots < a_{k-1} < f(p_k) < a_k$. 
 Now it is sufficient to note that $f^{-1}([a_{i-1}, a_{i}])$ is diffeomorphic to attaching a 
$j$-handle to $f(a_{i-1}) \times [0, 1]$, where $j$ is the index of the critical point $p_i$, 
\emph{i.e.} the number of the negative eigenvalues of the Hessian of $f$ at that critical point.
 For example a torus with the height function is first given by attaching a $D^0 \times D^2$ to
 the empty set, then attaching a $D^1 \times D^1$ (think of it as a thin ribbon) to the boundary of
 the previous $D^2$. Then attaching another ribbon to the boundary of the previous ribbon, and the 
finally attaching a $D^2 \times D^0$ to what has been obtained along the boundary.

\begin{lemma}\label{L:TCHhandling}
Let $M$ be an $n$-framed manifold and $N$ be an $n$-framed manifold obtained from $M$ by attaching a countable sequence of handles $(H_i)_{i\in \N}$. For any $n\in \N$, we write $X_{k}$ for the result of attaching the first $k$-handles to $M$.  For any $E_n$-algebra $A$, there is a natural equivalence
$$ \colim_{k\in \N} \int_{X_k} A \stackrel{\simeq} \longrightarrow \int_N A.$$ 
\end{lemma}

\begin{proof} We may assume $\dim(M)=n$. The  maps $X_k\to X_{k+1}$ yield  
a diagram  \begin{equation} \label{eq:seqbfD} N({\text{Disj}}(M)) \to N({\text{Disj}}(X_1)) \to \cdots \to N({\text{Disj}}(X_k)) \to \cdots \to N({\text{Disj}}(N))\end{equation} of faithfull maps, hence  canonical maps  $\colim_{k\in \N} N({\text{Disj}}(X_k)) \to N({\text{Disj}}(N))$ and
\begin{equation}\label{eq:TCHunion}  \colim_{k\in \N} \int_{X_k} A\cong \colim_{k\in \N} \big(\colim_{} \psi_{X_k}\big) \to \colim_{} \psi_N\cong \int_N A\end{equation} (using the notation introduced in Definition~\ref{Def:TCHasColimit}). Note that $\colim_{k\in \N} \big(\colim_{} \psi_{X_k}\big)$ can be identified with the colimit $\colim \tilde{\psi}_N$ given informally by the diagram $$\tilde{\psi}_N (V_1^k\cup \cdots\cup V_j^k)=\int_{V_1^k}A \otimes \cdots \otimes \int_{V_j^k} A $$ where the $V_i^k$ are disjoint open subsets of $X_k$ homeorphic to a ball (here $k$ is not fixed).

 Since $A$ is a fixed $E_n$-algebra, it is in particular an $\mathbb{E}(X)$-algebra (in the sense of~\cite[Section 5.2.4]{L-VI}) for any $n$-framed manifold $X$.  In particular, by~\cite[Theorem 5.2.4.9]{L-VI} (also see Proposition~\cite[Proposition 5.3.2.13]{L-VI}), if $U\subset V$ are two open subsets of $X$ which are homeomorphic to a ball, then the induced map $\int_U A \to \int_V A $ is an equivalence. 

Now, let $V$ be an open subset of $N$ which is homeomorphic to a ball. Since $N=\colim_{k\in \N} X_k$, there exists a $k$ such that $V\cap X_k$ is a non-empty open subset of $X_k$, and thus contains an open ball $V^k$ in $X_k$ which is lying in $V$ too. In particular the natural map $\int_{V^k} A\to \int_{V} A$ is an equivalence. It follows that, given a finite set $V_1,\dots,V_n$ of open homeomorphic to a ball  in $N$, we can find an integer $k$ big enough and open sets  $V^k_1\subset V_1$, ...., $V^k_n\subset V_n$ in $X_k$ which are homeomorphic to a ball, yielding a natural equivalence $$\tilde{\psi}_N (V_1^k,\dots, V_n^k)=\int_{V_1^k}A \otimes \cdots \otimes \int_{V_n^k} A \cong \int_{V_1}A \otimes \cdots \otimes \int_{V_n} A=\psi_N (V_1,\dots, V_n).$$
This proves the cofinality of the functor $N(\widetilde{\mathop{\rm Disj}}(N))\hookrightarrow N(\mathop{\rm Disj}(N)) $ induced by the natural inclusion. Here we have denoted $\widetilde{\mathop{\rm Disj}}(N)$ the partially ordered set of open subsets of $N$ which are homeomorphic to $F\times R^n$  for a finite set $F$ and included in some $X_k$ (where $k$ is not fixed). Passing to colimits, we get that the map   $\colim_{k\in \N} \big(\colim_{} \psi_{X_k}\big)\cong \colim \tilde{\psi}_N \to \colim \psi_N$ is an equivalence and thus the canonical map~\eqref{eq:TCHunion} is an equivalence as well.
\end{proof}

\smallskip

\begin{proof}[Proof of Theorem~\ref{T:TCH=CH}] 
Let us sketch the key idea of the proof first: by the value on a point axiom, both topological chiral homology and higher Hochschild chains agree on a point and further on any disk $D^k$ (up to neglect of structure). Using handle decompositions, one can chop manifolds on disks which are glued along their boundaries. Since both topological chiral homology and higher Hochschild chains satisfy a similar gluing axiom (and also behave the same way under disjoint unions), one then can lift the natural equivalence for disks to any manifold using handle decompositions. We now make the above scheme precise.

\smallskip
 
Assume $M$ is compact.
Let us choose a generic Morse function on $M$ and the associated handle decomposition $\emptyset \subset M_0 \subset \cdots \subset M_m =M$ of $M$. Then $\emptyset \subset M_0\times  D^k \subset \cdots \subset M_m\times D^k =M \times D^k $ is a handle decomposition of $M� \times D^k$. That is $(M\times D^k)_j=M_j\times D^k$ where   we replace each $j$-handle $H^j=D^j\times D^{m-j}$ attached to $M_{j-1}$ by the $j$-handle $D^j\times D^{m+k-j}\cong D^j\times \big(D^{m-j}\times D^{k}\big)$ attached to $M_{j-1} \times D^{k}= (M\times D^k)_{j-1}$. The $(m+k)$-framing of $M$ induces an $(m+k)$-framing of each $M_j\times D^k =(M\times D^k)_j$.

\smallskip

 By homotopy invariance of Hochschild chains, one has an equivalence of CDGAs $CH_{M\times D^d}(A)\cong CH_{M}(A)$ (for any integer $d$). Further, from diagram~\eqref{E:EnCom} we deduce that, for any CDGA $B$, one has $j_{k}^*(B)\cong \iota_{d}^*(j_{k+d}^*(B))$ where $\iota_d: E_k\hookrightarrow E_{k+d}$ is the natural map. 
Since, for any $E_{m+k}$-algebra $B$, one has $\int_{D^m}B \cong B$ viewed as an $E_k$-algebra, the result of Theorem~\ref{T:TCH=CH} holds for all disks.

We now prove by induction that it holds for all ($(m+d)$-framed) spheres  $S^m$ and $E_{m+d}$-algebra $A$.
For $S^0=pt\coprod pt$, it follows from Theorem~\ref{T:derivedfunctor} and~\cite[Theorem 5.3.3.1]{L-VI} that $\int_{S^0}A \cong A\otimes A \cong CH_{S^0}(A)$ (as $E_{d}$-algebras). Now, assume the result for $S^{m-1}$ and $m\geq 1$. We have a decomposition of the $m$-sphere $S^m$ as $S^m\cong D^m \cup_{S^{m-1}} D^m$ as in the assumption of Proposition~\ref{P:TCHpushout}. Since this decomposition is also an homotopy pushout, it follows from the induction hypothesis, Proposition~\ref{P:TCHpushout} and Theorem~\ref{T:derivedfunctor}.(3) that there are natural equivalences $$\int_{S^m}A \;\cong \; A\!\mathop{\otimes}\limits^{\mathbb{L}}_{\int_{S^{n-1}}\! A}\! A \;\cong\; CH^\com_{D^m}(A)\!\mathop{\otimes}\limits^{\mathbb{L}}_{CH^\com_{S^{m-1}}(A)} \! CH^\com_{D^m}(A) \;\cong \; CH^\com_{S^m}(A) $$ of $E_d$-algebras which finishes the induction step.

\smallskip

Clearly, $M_0$ is a disjoint union $M_0=\coprod_{I_0} D^m$ of finitely many $m$-dimensional balls (here $I_0$ is the set indexing the various disks in $M_0$). Using again that, for any $E_{m+k}$-algebra $B$, one has $\int_{D^m}B \cong B$ viewed as an $E_k$-algebra, we get a natural equivalence of $E_k$-algebras $\int_{M_0} A \cong \bigotimes_{I_0} j_k^*(A)\cong j_k^*(\bigotimes_{I_0} A)$  since $\int_{M\coprod N} A\cong \int_M A \otimes \int_N A$ by~\cite[Theorem 5.3.3.1]{L-VI}, the set $I_0$ is finite  and $j_k$ is monoidal (that is commutes with the diagonals of the $E_k$-operads). Further $CH^\com_{D^m}(A) \cong CH^\com_{pt}(A)\cong A$ and, by Theorem~\ref{T:derivedfunctor}.(1) and (2), there is  a natural equivalence $CH^\com_{M_0}(A)\cong \bigotimes_{I_0} A$ of CDGAs. Hence the theorem is proved for $M_0$. 

\smallskip

By assumption $M_1$ is obtained by attaching finitely many $1$-handles $H^1_1,\dots, H^1_{i_1}$ (the sequence may be empty) to $M_0$. Choosing appropriate tubular neighborhoods for the image of $\partial{D^1}\times D^{m-1}$ in $M_0$ and gluing it with $\big(\partial{D^1}\times D^{m-1}\big)\times [0,\varepsilon]\cong\big(\partial{D^1}\times [0,\varepsilon]\times D^{m-1}\big)\subset H^1_1$, we can assume that the result $M_0\cup_{\partial D^1 \times D^{m-1}} H_1^1$ of  attaching $H^1_1$ to  $M_0$ satisfies the assumption of Proposition~\ref{P:TCHpushout}. Then, by Proposition~\ref{P:TCHpushout} we have a natural equivalence of $E_k$-algebras:
\begin{equation}\label{eq:attachhandle}
\int_{M_0} A\mathop{\otimes}\limits_{\int_{\partial D^1 \times D^{m-1}} A}^{\mathbb{L}} \int_{H^1_{1}} A\cong \int_{M_0\cup_{\partial D^1 \times D^{m-1}} H_1^1} A.
\end{equation}
By Theorem~\ref{T:derivedfunctor}.(3), we  also have a natural equivalence of CDGAs (and thus of the underlying $E_{k}$-algebras)
\begin{equation}\label{eq:attachhandleCH}
CH^\com_{M_0} (A)\mathop{\otimes}\limits_{CH^\com_{\partial D^1 \times D^{m-1}}(A)}^{\mathbb{L}} CH^\com_{H^1_{1}} (A)\cong CH^\com_{M_0\cup_{\partial D^1 \times D^{m-1}} H_1^1} (A).
\end{equation}
We have already seen that we have natural equivalences $CH^\com_{M_0}(A) \cong \int_{M_0}(A)$  of $E_k$-algebras and similarly for  $\partial D^1 \times D^{m-1}$ and $D^1\times D^{m-1}$ in place of $M_0$ (since those manifolds are disks).  Combining these equivalences with those given by the identities~\eqref{eq:attachhandle} and~\eqref{eq:attachhandleCH}, we get a natural equivalence
$$\int_{M_0\cup_{\partial D^1 \times D^{m-1}} H_1^1} A \cong CH^\com_{M_0\cup_{\partial D^1 \times D^{m-1}} H_1^1} (A).$$ Attaching more $1$-handles, we inductively get a natural equivalence $\int_{M_1} A \cong CH^\com_{M_1}(A)$ of $E_k$-algebras. 

\smallskip

We proceed the same for attaching $j\geq 2$-handles (by induction on $j$). The proof is identical to the $1$-handles case once we notice that there are also a natural equivalence $\int_{\partial{ D^j}\times D^{m-j}} A \cong CH^\com_{\partial{ D^j}\times D^{m-j}}(A)$. The later follows 
from the natural equivalences relating topological chiral homology and higher Hochschild chains of spheres (with value in $A$) proved above. Indeed, $\partial{ D^j}\times D^{m-j} \cong S^{j-1}\times D^{m-j}$ and there is a natural equivalence $\int_{S^{j-1}\times D^{m-j}} A\cong i_{m-j}^*\big( \int_{S^{j-1}} A \big)$ where $i_{m-j}: E_{k+1}\hookrightarrow E_{m+k-j+1}$ is the canonical map (neglecting part of the structure). Similarly, there are natural equivalences $CH^\com_{S^{j-1}}(A) \cong CH^\com_{S^{j-1} \times D^{m-j}}(A)$ of CDGAs. It follows, since $j_{k+m -j+1}^*\big(CH^\com_{S^{j-1}} (A)\big) \cong \int_{S^{j-1}} A$, that  we get a natural equivalence $$j_{k+1}^*(CH^\com_{S^{j-1}\times D^{m-j}} \big( (A)\big)\cong i_{m-j}^*\, j_{k+m-j+1}^*\big( CH^\com_{S^{j-1}\times D^{m-j}}(A)\big) \cong \int_{S^{j-1}\times D^{m-j}} A$$
which finishes the proof in the compact case.

\smallskip

If $M$ is non-compact, we still have a handle decomposition, but we may have to attach countable many handles to go from $M_i$ to $M_{i+1}$. In particular, we can find an increasing sequence of relatively compact open subsets $M_i=X_0\subset X_1\subset\cdots X_n\subset \cdots \subset \bigcup_{n\geq 0} X_n =M_{i+1}$ (for instance by choosing $\overline{X_n}$ to be the result of attaching the first $n$ $i$-handles  to $M_i$).  We wish to prove the result by induction on $i$. Note first that,
by definition of the Hochschild chain functor and Lemma~\ref{L:TCHcoproduct}, there is an equivalence (for the underlying $E_k$-algebras structures) $$CH^\com_{M_0}(A)\cong \colim_{\small \begin{array}{l}F_0 \subset I_0\\ 
F_0 \mbox{ finite} \end{array}
} \left(\bigotimes_{f\in F_0} CH^\com_{D^m}(A)\right) \cong \colim_{\small \begin{array}{l}F_0 \subset I_0\\ 
F_0 \mbox{ finite} \end{array}} \left(\bigotimes_{f\in F_0} \int_{D^m}(A)\right) \cong \int_{M_0} A$$
which proves the result for $M_0$. Now, assume we have an natural equivalence $CH^\com_{M_i}(A)\cong \int_{M_i} A$. Writing $M_i=X_0\subset X_1\subset\cdots X_n\subset \cdots \subset \bigcup_{n\geq 0} X_n =M_{i+1}$, by the above argument for the finite handles case,  we have natural equivalences 
$CH^\com_{X_n}(A) \cong \int_{X_n} A$ for all $n$ and thus a commutative diagram 
\begin{equation}
\xymatrix{\colim_{n\in \N} CH^\com_{X_n}(A) \ar[r]^{\simeq} \ar[d] & \colim_{n\in \N} \int_{X_n} A \ar[d] \\ CH^\com_{M_{i+1}}(A) \ar[r] &  \int_{M_{i+1}} A}.
\end{equation}
 By Lemmas~\ref{L:TCHhandling} and Lemma~\ref{L:CHhandling}, the vertical arrows are equivalences,hence the lower map is too, which finishes the induction.
\end{proof}

\begin{remark} A geometric intuition behind Theorem~\ref{T:TCH=CH} can be seen as follows. Let $M$ be a dimension $m$ manifold. Since a CDGA $A$ is an $E_{n}$-algebra for any $n$,  the topological chiral homology $\int_M A$ is defined for any $n$-framing of $M$, and is an $E_{n-m}$-algebra.  
Further, if $M$ is $n$-framed (hence we have chosen a trivialization of $M\times D^{n-m}$), then $M$ is also naturally $(n+k)$-framed for any  integer $k$. Since $A$ is a CDGA, it is an $E_{n+k}$-algebra as well and thus we could have used the trivialization of $M\times D^{n-m} \times D^k\cong M\times D^{n+k-m}$ as well to define $\int_M A$ as an $E_{n-m+k}$-algebra. 

\smallskip

It is well known that the transversality theorem implies  that two embeddings $\phi_1: M \to S^n$ and $\phi_2: M\to S^n$  of $M$ are isotopic if $n$ is large enough. In particular, for large $n$ the framing that comes from the embedding into $S^n$ is unique. This unique invariant of $M$ is called the stable normal bundle. Note that this implies that any two abstract framings of $M\times D^k$ and $M \times D^l$ are stably equivalent since, for example, for $M\times D^k$  we can make $M$ sit inside $\mathbb R^n$ and then the normal bundle of $M$ in $\mathbb R^n$ is the complement of the framing of the tangent bundle of $M\times D^k$ in $\mathbb{R}^n \times D^k$. 
\end{remark}

Building upon the last remark, we see that  the topological chiral homology of an $(m+k)$-framed manifold $M$ with value in  a CDGA should be equivalent (up to neglect of structure) to the topological chiral homology of $M$ equipped with the stable normal framing so that we get
\begin{corollary} Topological chiral homology with values in CDGAs is independent of the framing. In other words, if  $M_{1}$ and $M_{2}$ are diffeomorphic manifolds equipped respectively with an $(n+k_1)$-framing and $(n+k_2)$-framing, then there is a canonical equivalence $\int_{M_1} A \cong \int_{M_2} A$ of $E_{\min(k_1,k_2)}$-algebras.
\end{corollary}
{In particular, $\int_M A$ is naturally defined for every manifold $M$ when $A$ is a CDGA.}
\begin{proof}
 By Theorem~\ref{T:TCH=CH}, there are natural equivalences $\int_{M_1} A \cong CH_{M_1}^{\com}(A)$ of $E_{k_1}$-algebras and $\int_{M_2} A \cong CH_{M_2}^\com (A)$ of $E_{k_2}$-algebras. Since $M_1$ and $M_2$ are diffeomorphic, we have an equivalence $CH_{M_1}^\com (A)\cong CH_{M_2}^\com (A)$ as CDGAs.
\end{proof}

\subsection{Relation with the Blob complex}\label{S:blob}
It is  asserted in \cite[Example 6.2.10]{MW} that the \textbf{Blob complex} of Morrison-Walker is equivalent to \emph{an unoriented variant of} topological chiral homology. We briefly explain here how to deduce this equivalence from the proof of Theorem~\ref{T:TCH=CH} and several statements from~\cite{MW}. The main difference with \S~\ref{S:ReviewTCH} and \S~\ref{S:locality} to keep in mind is that we no longer assume our manifold to be \emph{framed}. Indeed, we are using the variant of topological chiral homology for (non-necessarily framed nor oriented) manifolds of dimension $n$. However this forces us to restrict our attention to  \emph{unoriented $E_n$-algebras}:
\begin{definition}\label{D:unorientedEn}
 The category of \emph{unoriented $E_n$-algebras}, denoted
$\mathbb{E}_n^{O(n)}\textit{-}Alg_\infty$ 
is defined as the ($(\infty,1)$-)category of symmetric monoidal functors
$$
\mathbb{E}_n^{O(n)}\textit{-}Alg_\infty := \text{\emph{Fun}}^\otimes(\text{\emph{Disk}}_n,\,\hkmod)
$$
where $\text{\emph{Disk}}_n$ is the category with objects the integers and morphism the spaces $\text{\emph{Disk}}_n(k,\ell):= \text{\emph{Emb}}(\coprod_{k} \R^n, \coprod_{\ell} \R^n)$  of smooth embeddings of $k$ disjoint copies of a disk $\R^n$ into $\ell$ such copies; the monoidal structure is induced by disjoint union of copies of disks.  
\end{definition}
Note that $\mathbb{E}_n^{O(n)}$-algebras are denoted $\mathcal{EB}_n$-algebras   in~\cite{MW}.
\begin{remark} The definition above is extracted from~\cite{L-VI,Francis}.
There is an natural action of the orthogonal group $O(n)$ on $E_n\textit{-}Alg_\infty$, see~\cite{L-TFT}; the category $(E_n\textit{-}Alg_\infty)^{hO(n)}$ of $O(n)$-homotopy fixed points is equivalent to  the $(\infty,1)$-category $\mathbb{E}_n^{O(n)}\textit{-}Alg_\infty$ of Definition~\ref{D:unorientedEn}
 In particular, \emph{any CDGA is an unoriented $E_n$-algebra}.

\smallskip

If one replaces the action of $O(n)$ by $SO(n)$,one recovers the notion of \emph{oriented} $E_n$-algebras which are commonly known as \emph{framed $E_n$-algebras} in the literature.
\end{remark}
Let $M$ be a dimension $n$ manifold. By~\cite[Example 6.2.10]{MW}, an $\mathbb{E}_n^{O(n)}$-algebra $A$  defines an $A_\infty$-$n$-category\footnote{this is, for $n=1$, an instance of the linear category with a single object associated to an associative algebra and, for general $n$, a slight variant   of the construction of an $(\infty, n)$-category $\mathcal{B}^n(A)$ associated to an $E_n$-algebra $A$ as in~\cite{Francis}} $\mathcal{C}^A$ and thus yields the \textbf{Blob complex} $\mathcal{B}_*(M,\mathcal{C}^A)$, see~\cite[Definition 7.0.1]{MW}. 
Similarly we can form the topological chiral homology $\int_M A$, see~\cite[Definition 3.15]{Francis} (or~\cite{L-VI,L-TFT}).
\begin{proposition}\label{P:blob} Let $M$ be a closed $n$-dimensional manifold (non necessarily framed nor oriented) and $A$ an $\mathbb{E}_n^{O(n)}$-algebra.
  There is an natural equivalence $$\int_{M} A\; \cong\; \mathcal{B}_*(M,\mathcal{C}^A)$$ in $\hkmod$. Further, if $A$ is a CDGA,  there is an natural equivalence $CH_M^\com(A) \cong \int_{M} A$ in $\hkmod$.
\end{proposition}
\begin{proof}
 According to~\cite[Theorem 7.2.1]{MW}, the Blob complex $\mathcal{B}_*(M,\mathcal{C}^A)$ satisfies the excision property for closed manifolds and can be computed using a colimit construction (\cite[\S~6.3]{MW}) similar to Definition~\ref{Def:TCHasColimit}. It also converts disjoint union of manifolds to tensor products~\cite[Property 1.3.2]{MW} and $\mathcal{B}_*(D^n,\mathcal{C}^A)\cong A$ by~\cite[Property 1.3.4]{MW} (and~\cite[Example 6.2.10]{MW}). The same properties hold for topological chiral homology as in Proposition~\ref{P:TCHpushout} (or Lemma~\ref{L:TCHpushout2}), the proofs and references to \cite{Francis,L-VI,L-TFT} in the unoriented case being essentially the same as in the framed one (using $O(n)$-homotopy fixed $E_n$-algebras). It follows that one can apply \emph{mutatis mutandis} the proof of Theorem~\ref{T:TCH=CH} to get the  equivalences stated in the proposition. 
\end{proof}
\begin{remark} Let  $A=C_\com(\Omega^{n} Y)$ be the $E_n$-algebra associated to the $n$-fold loop space of an $n$-connective pointed space $Y$.  The Blob complex  $\mathcal{B}_*(M,\mathcal{C}^{C_\com(\Omega^{n} Y)})$ for non-necessarily closed manifolds shall not be mistaken with the Blob complex associated to the $A_\infty$-$n$-category associated to the fundamental groupoid of the space $Y$ in~\cite[\S~6]{MW} (though the construction share some similarities). Indeed, it is claimed~\cite[Theorem 7.3.1]{MW} that the latter construction compute  the chains on the space of all maps $\textrm{Map}(M,Y)$, while the first one, by non-abelian Poincar\'e duality~\cite{L-VI}, $\int_{M} \Omega^{n} Y$ is the chains on the space of maps with compact support $\textrm{Map}_{c}(M,Y)$.
\end{remark}

\smallskip

Theorem~\ref{T:HF=TCH} below and Proposition~\ref{P:blob} suggest that the blob complex should be closely related to Factorization algebras as well. It would be interesting to relate system of fields, n-category and the blob complex in the sense of~\cite{MW} to factorization algebras (with extra properties) in the sense of~\cite{CG}. We plan to investigate these relationship in a future work.

\subsection{Topological chiral homology as a factorization algebra}\label{S:TCH=Fact}
In this section we give a precise relationship between factorization algebras, topological chiral homology for (stably) framed manifolds, and $E_n$-algebras.

\subsubsection{Topological chiral homology and factorization algebras for $n$-framed manifolds}
For any manifold $M$ of dimension $m$ which is $n$-framed (\emph{i.e.} $M\times D^{n-m}$ is framed) and $E_n$-algebra $A$, we can consider the topological chiral homology $\int_M A$ as well as $\int_U A$ for every open subset $U$ in $M$ (equipped with the induced framing). Further, if $U_1,\dots, U_k$ are pairwise disjoint open subsets of $V\in Op(M)$, there is a canonical equivalence (\cite[Theorem 3.5.1]{L-VI}) $$\int_{U_1} A \otimes \cdots \otimes \int_{U_k} A \;\stackrel{\simeq}{\longrightarrow}\int_{U_1\cup \cdots \cup U_k} A \;    $$   and a natural map $\int_{U_1\cup \cdots \cup U_k} A\to \int_V A$ (since any ball in $\bigcup U_i$ is a ball in $V$). Composing these two maps yield natural maps of $E_{m-n}$-algebras
\begin{equation}\label{eq:muTCHisFact}\mu_{U_1,\dots, U_k, V}: \int_{U_1} A \otimes \cdots \otimes \int_{U_k} A \; \longrightarrow \int_V A.\end{equation}

\begin{proposition}\label{P:TCHisFact} Let $M$ be an $n$-framed manifold of dimension $m$.
\begin{enumerate}
\item For any  $E_n$-algebra $A$, the rule $U\mapsto \int_U A$ (for $U$ open in $M$) together with the structure maps $\mu_{U_1,\dots, U_k,V}$~\eqref{eq:muTCHisFact} define a \emph{locally constant} factorization algebra $\mathcal{TC}_M(-,A)$ on $M$, such that $\mathcal{TC}_M(U,A)=\int_U A$ is canonically an $E_{n-m}$-algebra for any open $U$.     
\item The rule $A\mapsto \mathcal{TC}_M(-,A)$ defines  a functor
$\mathcal{TC}_M:{E_n}\textit{-}Alg_\infty\to \mathop{Fac}_M({E_{n-m}}\textit{-}Alg_\infty)$ which fits into the following commutative diagram
$$\xymatrix{{E_n}\textit{-}Alg_\infty \ar[rr]^{\int_M} \ar[d]_{\mathcal{TC}_M} & &{E_{n-m}}\textit{-}Alg_\infty  \\ 
\mathop{Fac}_M({E_{n-m}}\textit{-}Alg_\infty) \ar[rru]_{HF(-)} && } $$
\end{enumerate}
\end{proposition}
In other words topological chiral homology computes the factorization homology of $\mathcal{TC}_M$.

\smallskip

The idea behind Proposition~\ref{P:TCHisFact} is that for any submanifold $U$ of $M$ and $E_n$-algebra $A$, we can cover $U$ by a (coherent family of) open balls on which $A$ defines a locally constant factorization algebra. Gluing these data defines a  factorization algebra on $U$ whose homology can be computed from the balls by using the gluing/locality lemma given above (Lemma~\ref{L:Enmodule}). Since   the  topological chiral homology is equivalent to $A$ on balls and satisfy a similar locality axiom, they agree on $U$.
\begin{proof}[Proof of Proposition \ref{P:TCHisFact}] 
For any open subset $V$ of $M$, the topological chiral homology $\int_V A$ is the colimit $\colim_{} \psi_V$,  where $\psi_V:N(\emph{\emph{\text{Disj}}}(V))\to \hkmod$ is the diagram given by the formula $\psi_V(V_1\cup\cdots\cup V_n)=\int_{V_1}A \otimes \cdots \otimes \int_{V_n} A$ where $V_1,\dots, V_n$ are disjoint open sets homeomorphic to a ball (Definition~\ref{Def:TCHasColimit}). In particular, the structure maps $\mu_{U_1,\dots,U_n,V}$ are induced by a map of colimits and it is easy to check that they are natural with respect to open embeddings and thus defines a prefactorization algebra. Hence $\mathcal{TC}_M(-,A)$ is functorially (in $A$) a \emph{prefactorization} algebra on $M$ with value in $E_{n-m}$-algebras.

\smallskip

To prove that $\mathcal{TC}_M(-,A)$ is actually a \emph{factorization} algebra,  the idea is first to use a handle body decomposition to define another locally constant factorization algebra $\mathcal{F}_M$ on $M$ whose factorization homology $HF(\mathcal{F},M)$ is equivalent to $\int_M A$ and then to prove that this factorization algebra is indeed equivalent to  $\mathcal{TC}_M$. Note that by Lemma~\ref{L:EnFact}, if $\mathcal{F}_M$ is a locally constant factorization algebra on $M$ with value in $E_d$-algebras such that $\int_M A \cong HF(\mathcal{F}_M,M)$, then  we have a natural factorization algebra $\mathcal{F}_{M\times \R^d}$ on $M\times \R^d$  and further $HF(\mathcal{F}_{M\times \R^d},M\times \R^d)\cong \int_M A$ as an $E_d$-algebra. 

\smallskip

We start with the case of open balls.
By definition of topological chiral homology, for every manifold $B$ homeomorphic to an $m$-dimensional ball, there is a natural equivalence $\int_B A \cong A$ of $E_{n-m}$-algebras (where the $E_{n-m}$-algebra structure of $A$ is by restriction of structure), see~\cite{L-VI,L-TFT}. By a result of Lurie~\cite{L-VI} (also see \cite[Proposition 3.4.1]{Co} or Proposition~\ref{P:Fac=En}), there is a locally constant factorization algebra $\mathcal{F}_B$ on $B$ whose factorization homology is isomorphic to $A$.  

We now  prove that the factorization algebra $\mathcal{F}_B$ is equivalent to $\mathcal{TC}_B$, \emph{i.e.}, that there  are equivalences of prefactorization algebras $ \mathcal{F}_B(U) \cong \int_U A$ (for any open subset $U\subset B$). Note that $\mathcal{F}_{B}(U)\cong {\mathcal{F}_B}_{|U}(U)\cong HF({\mathcal{F}_B}_{|U}, U)$. To shorten notation, we denote $\mathcal{F}_U$ the factorization algebra ${\mathcal{F}_B}_{|U}$ induced on $U$ by restriction of $\mathcal{F}_B$ to $U$.
  One has $A^{\otimes n} \cong \int_{U_1} A \otimes \cdots \int_{U_l} A \cong \int_{U_1\cup \cdots \cup U_l} A$ for any pairwise disjoint open subsets $U_1,\dots,U_n$ of $B$ homeomorphic to a ball (\cite{L-VI}) and $$\mathcal{F}_B\Big(\coprod_{i=1}^l B_i\Big)\;\cong \;\bigotimes_{i=1}^l \mathcal{F}_{B}(B_i)\;\cong \;\bigotimes_{i=1}^l A$$ for any open balls $B_i\subset B$ (since $\mathcal{F}_B$ is locally constant). Hence $\mathcal{F}_B(U)\cong \int_U A$ when $U$ is a disjoint union of balls. 
In particular, this equivalence holds for any (framed) embedding of $S^0\times D^m$ in $B$. We now prove, by induction on $i$, that $$\mathcal{F}_B(S^{i}\times D^{m-i})\;\cong\; HF({\mathcal{F}_{S^i\times D^{m-i}}}, S^i\times D^{m-i})\;\cong\; \int_{S^i\times D^{m-i}} A$$ for any embedding of $S^{i}\times D^{m-i}$ in $B$ (where $i\geq 1$).   We have a decomposition of the sphere $S^i\times D^{m-i}$ as $S^i\times D^{m-i}\cong D^+ \cup_{S^{i-1}\times D^{i-m}\times D^1} D^-$ where $D^+$ and $D^-$ are homeomorphic to framed open subballs of $B$. By above, the factorization homology of $\mathcal{F}_{D^+}$ and  $\mathcal{F}_{D^-}$ are equivalent to $A$ and the factorization homology of  $\mathcal{F}_{S^{i-1}\times D^{i-m}\times D^1}$ is equivalent to $\int_{S^{i-1}\times D^{i-m}} A$ as an $E_{1}$-algebra by the induction hypothesis.  Now the equivalence $HF({\mathcal{F}_{S^i\times D^{m-i}}}, S^i\times D^{m-i})\cong \int_{S^i\times D^{m-i}} A$ follows from Lemma~\ref{L:Enmodule} and  Proposition~\ref{P:TCHpushout}. This completes the inductive step.  For general open submanifold  $U$ of $B$, we use a handle decomposition as in the proof of Theorem~\ref{T:TCH=CH} (using Lemma~\ref{L:Enmodule} instead of Theorem~\ref{T:derivedfunctor} and arguments  similar to what has just been explained above for spheres) to obtain, in a similar way, that $$\int_U A \;\cong\; HF\big(\mathcal{F}_U,U\big)\; \cong \; \mathcal{F}(U).$$
This proves that $\mathcal{TC}_B\cong \mathcal{F}_B$ and thus that $\mathcal{TC}_B$ is a factorization algebra. 

\medskip

 If $M=\coprod_{i=1}^l B_i$ is a disjoint union of balls of dimension $m$, then we also deduce that $\mathcal{TC}_{M}$ is a factorization algebra since 
 $\int_{U_1} A \otimes \cdots \otimes \int_{U_l} A \cong \int_{U_1\coprod \dots \coprod U_l} A$ for any  open subsets $U_i\subset B_i$.
In particular, this applies to the case of $S^0\times D^d$ ($d\geq 0$). Now let $M\cong S^m\times D^d $ ($d\geq 0$, $m\geq 1$) be framed (we do not assume it is embedded as an open set of an euclidean space). We work  by induction on $m$ so that we may assume the result of the proposition is known for $S^{m-1}\times D^l$. 

Assume we have a cover $U\cup V$ of a space $X$ and factorization algebras $\mathcal{B}_U$, $\mathcal{B}_V$, $\mathcal{B}_{U\cap V}$ on $U$, $V$ and $U\cap V$ with equivalences $\mathcal{B}_{U\cap V} \stackrel{\simeq}\longrightarrow {\mathcal{B}_U}_{|U\cap V}$ and $\mathcal{B}_{U\cap V} \stackrel{\simeq}\longrightarrow {\mathcal{B}_V}_{|U\cap V}$. Then we can glue these factorization algebras to define a factorization algebra on $X$, see~\cite{CG} (note that this descent property can be generalized to arbitrary covers). 
We wish to apply this to a decomposition of the $m$-sphere $S^m$ as $S^m\cong D^+ \cup_{S^{m-1}\times D^1} D^-$ where $D^+$ and $D^m$ are homeomorphic to framed open balls. By the above analysis, there are locally constant factorization algebras $\mathcal{F}_+$, $\mathcal{F}_{-}$ on $D^+\times D^d$ and $D^-\times D^d$ which are equivalent to $\mathcal{TC}_{D^{m+d}}(A)$. Restricting these equivalences to $S^{m-1}\times D^1 \times D^{d}$, we get an equivalence $${\mathcal{F}_{+}}_{|S^{m-1}\times D^1 \times D^{d}} \stackrel{\simeq}\longrightarrow {\mathcal{F}_{-}}_{|S^{m-1}\times D^1 \times D^{d}}.$$  Since $S^{m-1}\times D^1 \times D^d$ is the intersection of $D^+\times D^d$ with $D^{-}\times D^d$, we thus get a locally constant factorization algebra $\mathcal{F}$ on their union $S^m\times D^d$. 
It follows from Lemma~\ref{L:Enmodule} and Proposition~\ref{P:TCHpushout}, that $\mathcal{F}(S^m\times  D^d)\cong \int_{S^m\times D^d} A$ as an $E_d$-algebra. The equivalences $\mathcal{F}(U)\cong \int_U A$ follows for any open proper  subset of $S^m\times D^d$ by the above case for open balls (or by mimicking the proof). It follows that $\mathcal{TC}_{S^m\times D^d} \cong \mathcal{F}$ and thus is a factorization algebra.

\smallskip

The case of general $n$-framed manifolds $M$ is done similarly. Using a handle decomposition, induction and the descent property of factorization algebras we build a locally constant factorization algebra $\mathcal{F}_M$ on $M$ and then prove, by using handle decomposition of any open subset $U$ of $M$, that $\mathcal{F}_M(U) \cong \int_U M$; the argument is similar to the case of spheres (as in the proof of Theorem~\ref{T:TCH=CH}). 
 
 \medskip
 
 The factorization algebra $\mathcal{TC}_M(-,A)$ is locally constant by construction (since $\int_B A \cong A$ for any ball $B$) and its factorization homology is precisely the topological chiral homology of $M$ with value in $A$. Further $\mathcal{TC}_M(-,A)$ is functorial in $A$ since topological chiral homology is.      
\end{proof}

\begin{remark} \label{R:TCH=HF=CH}
Theorem~\ref{T:TCH=CH} follows easily from Proposition~\ref{P:TCHisFact} and Corollary~\ref{C:HFact=CH}.
\end{remark}

\subsubsection{Topological chiral homology and factorization algebras for $\mathbb{E}_n[M]$-algebras}\label{SEC:En[M]-alg}
We now go beyond the notion of $n$-framed manifolds $M$, and, more generally, consider locally constant algebras over an operad $\mathbb{E}_n[M]$, for which there might not exist a globally defined $E_n$-algebra.

\smallskip 

Following Lurie~\cite{L-VI}, topological chiral homology can also be defined for a (locally constant) family of $E_n$-algebras parametrized by the points in $M\times D^{m-n}$ even if $M$ is \emph{not}  $n$-framed. Such  families   objects are (locally constant) algebras over an \emph{($\infty$-)operad $\mathbb{E}_n[M]:=\mathbb{E}_{M\times D^{n-m}}^{\otimes}$, the operad of little $n$-cubes in $M\times D^{m-n}$}, see~\cite[Definition 5.2.4.1]{L-VI} (here $M$ is still of dimension $m$, and of course one can choose $m=n$). Note that there is a canonical map of $\infty$-operads from $\mathbb{E}_n[M]$ to the the operad $Comm$ governing CDGAs. Thus a CDGA is an  $\mathbb{E}_n[M]$-algebra in a canonical way.

\medskip

By~\cite[Theorem 5.2.4.9]{L-VI}, we can also describe an $\mathbb{E}_n[M]$-algebra as a locally constant  $N({\text{Disk}}(M\times D^{n-m}))$-algebra. Indeed, by~\cite[Remark 5.3.2.7]{L-VI} we can extend Definition~\ref{Def:TCHasColimit} to an $\mathbb{E}_n[M]$-algebra $\mathcal{A}$ as well by replacing the last equivalence in~\eqref{eq:DefTCH} by $\mathcal{A}(V_1)\otimes \cdots \otimes \mathcal{A}(V_n)$. That is, there is an equivalence 
\begin{equation}\label{eq:DefTCHnotframed} \int_M \mathcal{A} \; \cong \; \colim \int_{V_1}A \otimes \cdots \otimes \int_{V_n} A\cong \colim\mathcal{A}(V_1)\otimes \cdots \otimes \mathcal{A}(V_n)\end{equation} where $V_1,\dots, V_n$ are disjoint open sets homeomorphic to a ball.

\smallskip

\begin{lemma}[Lurie~\cite{L-VI}]\label{L:TCHpushout2} Let $M$ be a manifold and $\mathcal{A}$ be an $\mathbb{E}[M]$-algebra. Assume that  there is a codimension $1$ submanifold  $N$ of $M$  with a trivialization $N\times D^1$ of its neighborhood  such that $M$ is decomposable as $M=X\cup_{N\times D^1}Y$ where $X,Y$ are submanifolds  of $M$ glued along  $N \times D^1$. Then
\begin{enumerate}
\item $\int_{N\times D^1} \mathcal{A}$ is an $E_1$-algebra and $\int_X \mathcal{A}$ and $\int_{Y} \mathcal{A}$ are right and left modules over $\int_{N\times D^1} \mathcal{A}$.
\item The natural map $$\int_X \mathcal{A}\; \mathop{\otimes}_{\int_{N\times D^1} \mathcal{A}}^{\mathbb{L}} \; \int_Y \mathcal{A} \; \longrightarrow \; \int_{M} \mathcal{A}$$ is an equivalence.
\end{enumerate}
\end{lemma}
\begin{proof}
The lemma is explained in~\cite{L-VI} after Theorem 5.3.4.14 (where $\int_{N\times D^1} \mathcal{A}$ is simply denoted by $\int_N \mathcal{A}$). Let us give some more details. There is a homeomorphism $N\times D^1 \cong N\times (0,1)$ and similarly (since $X$ has a neighborhood homeomorphic to $N\times D^1$) a homeomorphism $X\cong X_0$ where $X_0=X\setminus N\times [0,1)$ (we also denote $Y_0=Y\setminus N\times (-1,0]$). Since $N\times (-1,0)$ and $N\times (0,1)$ are disjoint open sets in $N\times D^1$, we get a natural map $\int_{X_0} \mathcal{A} \otimes \int_{N\times (0,1)} \mathcal{A} \to \int_{X} \mathcal{A}$. Then the module structure  on $\int_{X} \mathcal{A}$ is given the composition
$$\int_{X} \mathcal{A} \otimes \int_{N\times D^1}\mathcal{A}\stackrel{\simeq} \longrightarrow  \int_{X_0} \mathcal{A} \otimes \int_{N\times (0,1)} \mathcal{A} \longrightarrow \int_{X} \mathcal{A} $$ and similarly for the module structure of $\int_{Y} \mathcal{A}$ and the $E_1$-structure of $\int_{N\times D^1}\mathcal{A}$.

For the second statement, we need to prove that  the canonical map
\begin{equation} \label{eq:L:TCHpushout2}
\colim \Big(\int_{X} \mathcal{A} \, \otimes \int_{N\times D^1} \mathcal{A} \, \otimes \int_{Y} \mathcal{A} \rightrightarrows \int_{X} \mathcal{A} \, \otimes \int_{Y} \mathcal{A}\Big) \longrightarrow \int_{M} \mathcal{A}
\end{equation} induced by the two module structures is an equivalence.
Using the equivalences of modules $\int_{X_0}\mathcal{A}\stackrel{\simeq}\to \int_X \mathcal{A}$,  $\int_{Y_0}\mathcal{A}\stackrel{\simeq}\to \int_Y \mathcal{A}$ and Definition~\ref{Def:TCHasColimit} (or more precisely the equivalence~\eqref{eq:DefTCHnotframed}), the left hand side of the map~\eqref{eq:L:TCHpushout2} is computed by the  colimit 
$$\colim_{\scriptsize \begin{array}{l} U_1,\dots, U_\ell \in \widetilde{\Ball} \\ \mbox{pairwise disjoint}\end{array} } \int_{U_1}A \otimes \cdots\otimes \int_{U_\ell} A $$
where $\widetilde{\Ball}$ is the set of open sets $U\in Op(M)$ homeomorphic to a ball  such that $U$ is either in $X_0$ or $Y_0$ or in $N\times D^1$. Thus, denoting $\Ball(M)$ the set of open sets homeomorphic to a ball in $M$, we are left to prove that the canonical map
 \begin{equation} \label{eq:LTCHpushout3} \colim_{\scriptsize \begin{array}{l} U_1,\dots, U_\ell \in \widetilde{\Ball} \\ \mbox{pairwise disjoint}\end{array} } \int_{U_1}A \otimes \cdots\otimes \int_{U_\ell} A \, \longrightarrow \hspace{-20pt}\colim _{\scriptsize \begin{array}{l} V_1,\dots, V_i \in \Ball(M) \\ \mbox{pairwise disjoint}\end{array} } \int_{V_1}A \otimes \cdots\otimes \int_{V_i} A \to \int_{V}A \end{equation} is an equivalence. 
For every family $B_1,\dots, B_k$ of pairwise disjoint open subsets of $M$ homeomorphic to a ball which intersects $N\times \{0\}$, choosing a point $x_j$ in each $B_j\cap \big(N\times \{0\}\big)$, we can find disjoint open subsets (homeomorphic to a ball)  $U_1,\dots, U_k$ included in $N\times D^1$, such that  $U_j$ contains  $x_j$ and is included in $B_j$.   The inclusion map $U_j\hookrightarrow B_j$ induces a canonical  morphism $ \int_{U_j} \mathcal{A} \to \int_{B_j} \mathcal{A}$ which is an equivalence since there is a natural equivalence $\int_B A \cong A$ for any ball.    It follows that the map  of (homotopy) colimits~\eqref{eq:LTCHpushout3}  is an equivalence.
\end{proof}

The generalization of Proposition~\ref{P:TCHisFact} to $\mathbb{E}_n[M]$-algebras is: 
\begin{theorem}\label{T:HF=TCH}Let $M$ be a manifold of dimension $m$ and $d \in \N$ an integer. 
\begin{enumerate}
\item The rule $\mathcal{A}\mapsto \Big(U\mapsto \int_U \mathcal{A} \Big)$ defines a functor of $(\infty,1)$-algebras $\mathcal{TC}_M: \mathbb{E}_d[M]\textit{-}Alg \to \mathop{Fac}^{lc}_{M}(E_d\textit{-}Alg)$ which fits into a commutative diagram
$$\xymatrix{\mathbb{E}_d[M]\textit{-}Alg \ar[rr]^{\int_M} \ar[d]_{\mathcal{TC}_M} & &{E_{d}}\textit{-}Alg_\infty  \\ 
\mathop{Fac}^{lc}_M({E_{d}}\textit{-}Alg_\infty) \ar[rru]_{HF(-)} && } $$
\item The functor $\mathcal{TC}_M: \mathbb{E}_d[M]\textit{-}Alg \to \mathop{Fac}^{lc}_{M}(E_d\textit{-}Alg)$ is an equivalence of $(\infty,1)$-categories. 
\end{enumerate}
\end{theorem}
In particular, any locally constant factorization algebra $\mathcal{F}$ on $M$ with values in $E_d$-algebras is equivalent to $\mathcal{TC}_M(\mathcal{A})$ for a unique (up to equivalences)  $\mathbb{E}_d[M]$-algebra $\mathcal{A}$, \emph{i.e.}, .  
 algebra over the operad of little cubes in $M\times D^d$.  Further, topological chiral homology of an open set $U$ with value in the associated  $\mathbb{E}_{M\times D^d}^{\otimes}$-algebra computes the (derived) sections of the  factorization algebra.  
 
 Also note that if $M$ is any manifold and $A$ is a CDGA, viewed as an $\mathbb{E}_d[M]$-algebra, then the factorization algebra $\mathcal{TC}_M(A)$ induced by Theorem~\ref{T:HF=TCH} is strongly constant in the sense of Section~\ref{S:Factorization}. Thus by  Corollary~\ref{C:HFact=CH}, there is an natural equivalence of ($E_{d}$-algebras) \begin{equation*}\int_M A \;\cong\; CH_M(A)\end{equation*}
 extending Theorem~\ref{T:TCH=CH} to non-framed manifolds.
\begin{proof}[Proof of Theorem~\ref{T:HF=TCH}] We first deal with assertion (1).
By Lemma~\ref{L:EnFact}, it is enough to prove that the rule $U\mapsto \int_U \mathcal{A}$, together with the structure maps 
\begin{equation} \label{eq:muFacloc} \int_{U_1}\mathcal{A} \otimes \cdots\otimes \int_{U_\ell}\mathcal{A} \stackrel{\sim}\longrightarrow  \int_{U_1\cup \cdots \cup U_\ell}\mathcal{A}\longrightarrow \int_{V}\mathcal{A} \end{equation} 
for ${U_i}$'s pairwise disjoint open subsets of $V\in Op(M\times D^d)$, defines a locally constant factorization algebra on $M\times D^d$, naturally in $\mathcal{A}\in \mathbb{E}_{M\times D^d}^{\otimes}\textit{-}Alg$.  Note that, by~\cite[Theorem 5.2.4.9]{L-VI}, the $\mathbb{E}_{M\times D^d}^{\otimes}$-algebra $\mathcal{A}$ satisfies that, for any ball $B$ which is a subset of a ball $B'$, the canonical map $\int_B \mathcal{A}\cong \mathcal{A}(B) \to \mathcal{A}(B')\cong \int_{B'}\mathcal{A}$ is an equivalence in $\hkmod$.  Now we can apply the same proof as the one of Proposition~\ref{P:TCHisFact} with $\mathcal{A}$ instead of $A$, using Lemma~\ref{L:TCHpushout2} instead of Proposition~\ref{P:TCHpushout}.

\smallskip

We now prove assertion (2). By~\cite[Theorem 5.2.4.9]{L-VI}, the canonical embedding $\theta: \mathbb{E}_{M\times D^d}^{\otimes}\textit{-}Alg\to N(\text{Disk}(M\times \R^d))\textit{-}Alg $ induces an equivalence between $\mathbb{E}_{M\times D^d}^{\otimes}\textit{-}Alg$ and locally constant $ N(\text{Disk}(M\times \R^d))$-algebras; we write $N(\text{Disk}^{lc}(M))\textit{-}Alg$ for the latter subcategory.  It is thus enough to   define a functor $\mathcal{EA}_M: \mathop{Fac}^{lc}_M({E_{d}}\textit{-}Alg_\infty)\to N(\text{Disk}^{lc}(M))\textit{-}Alg$ such that $\mathcal{TC}_M\circ \mathcal{EA}_M$ and  $\mathcal{EA}_M \circ \mathcal{TC}_M$ are respectively equivalent to the identity functors of $\mathop{Fac}^{lc}_M({E_{d}}\textit{-}Alg_\infty)$ and $N(\text{Disk}^{lc}(M))\textit{-}Alg$. Let $\mathcal{F}$ be in $\mathop{Fac}^{lc}_M({E_{d}}\textit{-}Alg_\infty)$. By Lemma~\ref{L:EnFact}, we can think of $\mathcal{F}$ as a locally constant factorization algebra on $M\times D^d$.  Let $B \in Op(M\times D^d)$ be homeomorphic to a ball. Then the restriction $\mathcal{F}_{|B}$ is a locally constant factorization algebra on $B\cong \R^n$, thus is equivalent to an $E_n$-algebra (which is canonically equivalent to $\mathcal{F}(B)$).
Further, for any finite family $B_1,\dots, B_\ell$ of pairwise disjoint open subsets homeomorphic to a ball and $U$ an open subset homeomorphic to a ball containing the $B_i$'s, the locally constant factorization algebra structure defines a canonical map $$\gamma_{B_1,\dots,B_l,U}:\mathcal{F}(B_1)\otimes \cdots \otimes \mathcal{F}(B_\ell) \longrightarrow \mathcal{F}(U) $$ which is an equivalence if $\ell=1$. The maps $\gamma_{B_1,\dots,B_l,U}$ are compatible in a natural way. This shows that the collection $\mathcal{F}(B)$ for all open sets $B\subset M\times D^d$ homeomorphic to a ball is a locally constant  $N(\text{Disk}(M\times \R^d))$-algebra, denoted $\mathcal{A}_\mathcal{F}$ and we define the functor $\mathcal{EA}_M$ to be defined by $\mathcal{EA}_M(\mathcal{F}):=\mathcal{A}_\mathcal{F}$ (for $M=pt$ it is the same as the functor $for$ defined in the proof of Proposition~\ref{P:Fac=En}). In other words, the functor $\mathcal{EA}_M(\mathcal{F})$ is simply induced by the composition $\text{Disk}(M\times \R^d)\hookrightarrow Op(M\times \R^d)\stackrel{\mathcal{F}}\to \hkmod$.   By abuse of notation, we also write  $\mathcal{A}_{\mathcal{F}}$ for the associated (well defined up to equivalences) $\mathbb{E}_{M\times D^d}^{\otimes}$-algebra.

Let $\mathcal{A}$ be an $\mathbb{E}_{M\times D^d}^{\otimes}$-algebra. By construction $\mathcal{EA}_M \circ \mathcal{TC}_M(\mathcal{A})$ is the (locally constant) $N(\text{Disk}(M\times \R^d))$-algebra given, on any (open set homeomorphic to an euclidean) ball $B$, by 
$$\Big(\mathcal{EA}_M \circ \mathcal{TC}_M(\mathcal{A})\Big)(B) = \int_B\mathcal{A}\, \cong \,  \mathcal{A}(B)$$ by definition of topological chiral homology~\cite[Example 5.3.2.8]{L-VI}. Hence there is a canonical equivalence $\mathcal{EA}_M \circ \mathcal{TC}_M(\mathcal{A}) \cong \mathcal{A}$ of locally constant $N(\text{Disk}(M\times \R^d))$-algebras and  thus of $\mathbb{E}_{M\times D^d}^{\otimes}$-algebras as well. 

It remains to prove that, there are natural equivalences $\mathcal{TC}_M(\mathcal{A}_{\mathcal{F}})\cong \mathcal{F}$ of factorization algebras, where $\mathcal{A}_{\mathcal{F}}$ is the $\mathbb{E}_{M\times D^d}^{\otimes}$-algebra $\mathcal{EA}_M(\mathcal{F})$ associated to $\mathcal{F}$ as above. Fixing a Riemannian metric on $M\times D^d$,  we can find a cover $\Ball^g(M\times D^d)$ of $M$ given by open sets in $M\times D^d$ which are geodesically convex. On every $U\in \Ball^g(M\times \R^d)$, the restrictions $\mathcal{TC}_{|U}(\mathcal{A}_{\mathcal{F}})$ and $\mathcal{F}_{|U}$ are naturally isomorphic by the above paragraph. In particular, for any set $U_I:= \bigcap_{i\in I} U_i$ and any subset  $J\subset I$, the following diagram 
$$ \xymatrix{\mathcal{TC}_{|U_I}(\mathcal{A}_{\mathcal{F}}) \ar[rr]^{\simeq} \ar[d] & & \mathcal{F}_{|U_I} \ar[d]\\  \mathcal{TC}_{|U_{I\setminus J}}(\mathcal{A}_{\mathcal{F}}) \ar[rr]^{\simeq}& & \mathcal{F}_{|U_{I\setminus J}}}$$ is commutative.
  Since $\mathcal{TC}_M(\mathcal{A}_{\mathcal{F}})$ and $\mathcal{F}$ are the factorization algebras obtained by descent from their restrictions on the cover $\Ball^g(M\times \R^d)$, on which they are naturally equivalent, it follows that  $\mathcal{F}$ is equivalent to $\mathcal{TC}_M(\mathcal{A}_{\mathcal{F}})$.
\end{proof}

\begin{example}
Since $S^2\times D^1$ embeds as an open set in $\R^3$, any $E_3$-algebra $A$ yields, by restriction, a (locally constant) factorization algebra $\mathcal{A}_{S^2}$ on $S^2$ (with values in $E_1\textit{-}Alg$) (Lemma~\ref{L:EnFact}). By Theorem~\ref{T:HF=TCH} and Proposition~\ref{P:TCHpushout}, decomposing the sphere as two disks glued along the equator, we  get that the factorization homology of $\mathcal{A}_{S^2}$ is given by 
$$ HF(\mathcal{A}_{S^2}) \; \; \cong \;\; A \mathop{\otimes}^{\mathbb{L}}_{CH_{S^1}(A)} A$$ as an $E_1$-algebra. Here $CH_{S^1}(A)$ is the usual Hochschild chain complex of the (underlying) $E_1$-algebra structure of $A$, which is naturally an $E_2$-algebra by~\cite[Theorem 5.3.3.11]{L-VI} and Proposition~\ref{P:TCHpushout}.


Similarly, any $E_2$-algebra $B$ yields a (translation invariant and locally constant) factorization algebra on $\R^2$, and thus a (locally constant) factorization algebra $\mathcal{B}_T$ on a torus $T=S^1\times S^1\cong \R^2/\Z^2$.  Since $T$ is framed, we can also define its topological chiral homology directly using the framing. It follows easily from the uniqueness statement in Theorem~\ref{T:HF=TCH}, that $\mathcal{B}_T$ is equivalent to $\mathcal{TC}_T(B)$ in $\mathop{Fac}_T(\hkmod)$. 

Note that the two canonical projections $p_1, p_2:\R^2\to \R$ define two locally constant factorization algebras ${p_1}_*(\mathcal{B})$, ${p_2}_*(\mathcal{B})$ on $\R$ and thus, two $E_1$-algebras $B_1$ and $B_2$. Now, cutting the torus along two meridian circles, we get two copies of $S^1\times D^1\cong \R^2/(\Z\oplus \{0\})$ glued along their boundaries.  By Theorem~\ref{T:HF=TCH} again, the topological chiral homology of  $S^1\times D^1$ is the same as the factorization algebra homology of the descent  factorization algebra $\mathcal{B}^{\Z\oplus \{0\}}$. Thus $\int_{S^1\times D^1} B $  is equivalent to the usual Hochschild chain complex $CH_{S^1}(B_1)$ and the later complex inherits an $E_1$-structure from the $E_2$-algebra structure of $\mathcal{B}$. 
From Proposition~\ref{P:TCHpushout}, we deduce a natural equivalence (in $\hkmod$)
$$HF(\mathcal{B}_T)  \cong  CH_{S^1}(B_1) \!\mathop{\otimes}^{\mathbb{L}}_{CH_{S^1}(B_1) \otimes (CH_{S^1}(B_1))^{op}} \!CH_{S^1}(B^1)  
 \cong  CH_{S^1}(CH_{S^1}(B_1)).
$$
Note that if $B$ was actually a CDGA, then the later equivalence follows directly from Corollary~\ref{L:CHYoCHX=CHXoCHZ}.
\end{example}

\subsection{Some applications}\label{SS:applications}

\subsubsection{Another construction of topological chiral homology for framed manifolds}
Let $M$ be an $m$-dimensional manifold which is $n$-framed. Given an $E_n$-algebra $A$, we can define the topological chiral homology $\int_M A$ of $M$ with values in $A$. By Proposition~\ref{P:TCHisFact} and Theorem~\ref{T:HF=TCH}, $\int_M A$ is the factorization homology of a factorization algebra on $M\times D^{n-m}$. We explain how to construct this factorization algebra directly.

\smallskip

Since $M$ is $n$-framed, there is a bundle isomorphism $\varphi: T(M\times D^{n-m}) \stackrel{\simeq}\longrightarrow \underline{\R^n}$ where $\underline{\R^n}$ is a trivial bundle over $M\times D^{n-m}$. 
Choosing a Riemannian metric on $M\times D^{n-m}$, we have, using the spray associated to the exponential map, canonical diffeomorphisms of (a basis of) open neighborhoods of any $x\in M\times D^{n-m}$ to open sets in the tangent space $T(M\times D^{n-m})_x$   of $M\times D^{n-m}$ at $x$. Composing with the map $\varphi$ induced by the framing, we get diffeomorphisms  $U\mapsto \psi(U_x)\in Op(\R^n)$ where $U$ is a (geodesically convex) open neighborhood of $x$.

Let $\mathcal{U}$ be the cover of $M\times D^{n-m}$ obtained by considering the $U$ above such that $\phi_x(U_x)\in \Ball(\R^n)$ is an euclidean ball. The cover $\mathcal{U}$ is a factorizing basis of open subsets of $M\times D^{n-m}$. 

To any $U\in \mathcal{U}$, we associate $\mathcal{A}(U)=A$, a (fixed) $E_n$-algebra. We wish to extend $\mathcal{A}$ into a factorization algebra. Since $A$ is an $E_n$-algebra, it defines a locally constant factorization algebra on $\R^n$ (see~\cite{Co,L-VI} and Proposition~\ref{P:Fac=En}), which we, by abuse of notation, again denote by $A$.
For any pairwise disjoint $U_1, \cdots, U_n \in \mathcal{U}$, and $V \in \mathcal{U}$ such that $U_i \subset V$ ($i=1\cdots n$), we define
the structure maps $\mu_{U_1,\dots, U_n, V}$ (see Section~\ref{S:Factorization})  by the following commutative diagram:
 \begin{equation*}\xymatrix{ \mathcal{A}({U_1})  \otimes \cdots \otimes \mathcal{A}({U_n}) \ar[d]_{\simeq} \ar[rr]^{\qquad \quad \mu_{U_1,\dots, U_n,V}}  & & \mathcal{A}({V})  \ar[d]^{\simeq}\\ A(\psi(U_1))\otimes \cdots \otimes A(\psi(U_n)) \ar[rr]_{} & & A(\psi(V))}\end{equation*}
where the lower arrow is given by the $E_n$-algebra structure of $A$. 
This yields a $\mathcal{U}$-factorization algebra  (in the sense of~\cite{CG} and \S~\ref{S:Factorization}) since $A$ is a factorization algebra on $\R^n$ and $M\times D^{n-m}$ is framed. By~\cite[Section 3]{CG}, we can now extend $\mathcal{A}$ to a factorization algebra on $M\times D^{n-m}$.

\begin{corollary}\label{C:FactforframedTCH}
There is an equivalence of $E_{n-m}$-algebras $$\int_M A \; \cong \; HF(M,\mathcal{A}).$$
\end{corollary}
\begin{proof}
For any ball $U$, we have a natural equivalence $\int_U A \cong A \cong \mathcal{A}(U)$. Now the result follows from Theorem~\ref{T:HF=TCH} (and its proof) after taking global sections. 
\end{proof}
Note that  topological chiral homology $\mathcal{TC}_M(A)$ is independent of the Riemannian metric, hence the factorization algebra $\mathcal{A}\in \mathop{Fac}_{M\times D^{n-m}}^{lc}(\hkmod)$ on $M$ thus obtained  is also independent of the Riemannian metric.

\subsubsection{Interpretation of topological chiral and higher Hochschild in terms of mapping spaces}

As we have already noticed, higher Hochschild chains behave much like mapping spaces (and thus so do $\int_M A$ for CDGAs $A$). Indeed,
\begin{corollary} Let $A=\Omega^\ast N$ be the de Rham forms on a $d$-connected manifold (with its usual differential). Then for any manifold  $M$ of dimension $m\leq d$, there is a natural quasi-isomorphism $\int_M A \cong \Omega^\ast (N^M)$, the space of (Chen) de Rham forms of the mapping space $N^M= \mathop{Map} (M,N)$. 
\end{corollary}
In other words, topological chiral homology of $M$ with value in $\Omega^\ast N$ calculates the mapping space $N^M$ (if $N$ is sufficiently connected).
\begin{proof}
 By Theorem~\ref{T:TCH=CH}, we are left to a similar statement for $CH_{M}(\Omega^\ast N)$. Since $M$ is $m$-dimensional it has a simplicial model with no non-degenerate simplices in dimensions above $m$. Now the result follows from~\cite[Proposition 2.5.3 and Proposition 2.4.6]{GTZ}.
\end{proof}

\begin{remark}
By~\cite[Section 2.4]{GTZ}, there is a canonical map $\int_M \Omega^\ast N \to \Omega^\ast (N^M)$. Further, it is possible to replace $N$ by any nilpotent space (by mimicking the proof of \cite[Propositions 2.5.3 and 2.4.6]{GTZ}) and $\Omega^\ast N$ by a Sullivan model of $N$.
\end{remark}

We now give a (derived/homotopical) algebraic geometry statement. Recall that $k$ denotes a field of characteristic zero and let $\mathbf{dSt}_k$ be the (model) category of \emph{derived stacks} over $k$ described in details in~\cite[Section 2.2]{ToVe} (which is a derived enhancement of the category of stacks over $k$). This category admits internal Hom's that we denote by $\mathbb{R}\mathop{Map}(\mathfrak{X},\mathfrak{Y})$ following~\cite{ToVe,ToVe2}. To any simplicial set $X_\com$, we associate the constant simplicial presheaf $k\textit{-}Alg \to \sset$ defined by $R\mapsto X_\com$ and we denote $\mathfrak{X}$ the associated stack. For a (derived) stack $\mathfrak{Y}$, we denote $\mathcal{O}_{\mathfrak{Y}}$ its functions~\cite{ToVe} (\emph{i.e.}, $\mathcal{O}_{\mathfrak{Y}}:=\mathbb{R}\underline{Hom}(\mathfrak{Y},\mathbb{A}^1)$).
\begin{corollary}\label{C:mappingstack} Let $\mathfrak{R}=\mathbb{R}\mathop{Spec}(R)$ be an affine derived stack (for instance an affine stack)~\cite{ToVe}. Then the Hochschild chains over $X_\com$ with coefficient in $R$ represent the mapping stack $\mathbb{R}\mathop{Map}(\mathfrak{X}, \mathfrak{R})$. That is, $$\mathcal{O}_{\mathbb{R}\mathop{Map}(\mathfrak{X},\mathfrak{R})}\; \cong \; CH_{X_\com}^{\com}(R).$$
\end{corollary}
\begin{proof}
The bifunctor $(\mathfrak{X},\mathbb{R}\mathop{Spec}(R))\mapsto \mathbb{R}\mathop{Map}(\mathfrak{X}, \mathbb{R}\mathop{Spec}(R))$ is contravariant in $\mathfrak{X}$ and $R \in \cdga^{\leq 0}$.  Thus, $\mathcal{O}_{\mathbb{R}\mathop{Map}(\mathfrak{X},\mathfrak{R})}$ defines a covariant bifunctor. Since $\mathbb{R}\mathop{Map}(-,\mathfrak{R})$ sends (homotopy) limits to (homotopy) colimits, it follows from Theorem~\ref{T:deriveduniqueness} (also see Remark~\ref{R:cdga-}) that   $\mathcal{O}_{\mathbb{R}\mathop{Map}(\mathfrak{X},\mathfrak{R})}$ is equivalent to $CH_{X_\com}^{\com}(R)$.
\end{proof}
\begin{example}
Let $B_\com \mathbb{Z}$ be the nerve of $\mathbb{Z}$ and $\mathfrak{B}\mathbb{Z}$ its associated stack. Recall that there is an homotopy equivalence  $S^1\to |B_\com \mathbb{Z}|$ (actually induced by a simplicial set map, see~\cite{L}). From Corollary~\ref{C:mappingstack} we recover that the derived loop stack $L\mathfrak{R}:=\mathbb{R}\mathop{Map}(\mathfrak{B}\mathbb{Z},\mathfrak{R})$ is represented by $CH_{B_\com \mathbb{Z}}^{\com}(R) \stackrel{\simeq}\longleftarrow CH_{S^1_\com}^{\com}(R)$ the standard Hochschild chain complex of $R$ as was proved in~\cite{ToVe2}. Similarly,  the derived torus mapping stack $\mathbb{R}\mathop{Map}(\mathfrak{B}\mathbb{Z}\times \mathfrak{B}\mathbb{Z},\mathfrak{R})$ is represented by $CH_{S^1\times S^1}^\com(R)$ and the secondary cyclic homology in the sense of~\cite{ToVe2} is represented by the homotopy fixed points $CH_{S^1\times S^1}^\com(R)^{h (S^1\times S^1)}$ with respect to the induced action of the simplicial group $B_\com \mathbb{Z}\times B_\com \mathbb{Z}$ on the derived mapping space.
\end{example}
\begin{remark}
Sheafifying (or rather stackifying) the higher Hochschild derived functor, it seems possible to extend Corollary~\ref{C:mappingstack} to  general derived schemes.  
\end{remark}

\subsubsection{Topological chiral Homology and homology spheres}

Topological chiral homology of CDGAs is a homology invariant, and thus, in particular, a homotopy invariant. Indeed, we have the following corollary.
\begin{corollary} Let $f:M \to N$ be smooth map between two manifolds inducing isomorphisms on homology and $A$ be a \cdga, then $\int_M A\cong \int_N  A$. 
\end{corollary}
\begin{proof} This follows  from Theorem~\ref{T:TCH=CH} and the quasi-isomorphism invariance of $CH_{(-)}^{\com}(A)$ (Proposition~\ref{P:homologyinvariance}).
\end{proof}

\begin{example}  
The composition  $S^3 \to SO(3) \to SO(3)/I$, where $I$ is the icosahedral group,  
 induces an isomorphism on homology. To see this note that the fundamental group of $SO(3)/I$ 
is the binary icosahedral group $\tilde{I}$ which is a perfect group and therefore $H_1(SO(3)/I)=0$. 
The result $SO(3)/I$ is the Poincar\'e homology sphere and has thus the same topological chiral homology 
with value in any CDGA as $S^3$. 
\end{example}

\begin{remark}
Note that we study topological chiral homology in the framework of chain complexes, \emph{i.e.} 
 we have fixed the $(\infty,1)$-category of chain complexes as our \lq\lq{}ground\rq\rq{} 
 monoidal $(\infty,1)$-category. If one works in some other framework (such as topological spaces), 
 one can expect to have more refined invariants.
\end{remark}

\begin{remark} 
Note that $S^1$ has two diffeomorphic $1$-framings (specified by a choice of orientation). 
This accounts for the fact that classically there is only one Hochschild complex for 
associative algebras. 
There are countably many $2$-framings for the circle, one for each integer, 
giving rise to equivalent topological chiral homologies when the integrand is a \cdga. 
It would therefore be meaningful to look for an explicit $E_2$-algebra that distinguishes these 
framings from one another, if such exists.  Similarly, it would be interesting to find an explicit $E_3$-algebra 
that distinguishes the two $3$-framings of $S^1$.
\end{remark}

\subsubsection{Fubini formula for topological chiral homology}
The exponential law for Hochschild chains (Proposition~\ref{P:product}) has an analogue for topological chiral homology.
\begin{corollary} \label{C:FubiniTCH} Let $M$, $N$ be manifolds and $\mathcal{A}$ be an $\mathbb{E}_d[M\times N]$-algebra. Then, $\int_N\mathcal{A}$ has a canonical lift as an $\mathbb{E}_d[M]$-algebra and further, there is an equivalence of $E_d$-algebras
$$\int_{M\times N} \mathcal{A} \; \cong \; \int_M\Big(\int_N \mathcal{A}\Big). $$
\end{corollary}
\begin{proof} Replacing $M$ by $M\times \mathbb{R}^d$ and using Lemma~\ref{L:EnFact}, 
it is enough to prove the result for $d=0$. Since the homology of a factorization algebra on $X$ is given by the pushforward along the canonical map $p:X\to pt$, 
by Theorem~\ref{T:HF=TCH}, one has 
 \begin{equation}\label{eq:FubiniTCH}\int_{M\times N}\mathcal{A}
 \cong p_*\big( \mathcal{TC}_{M\times N}(\mathcal{A})\big) 
\cong p_*\big(\pi_*\big(\mathcal{TC}_{M\times N}(\mathcal{A}) \big)\big)\end{equation} 
where $\pi:M\times N\to M$ is the canonical projection. Since $\mathcal{TC}_{M\times N}(\mathcal{A})$
 is locally constant, $\pi_*\big(\mathcal{TC}_{M\times N}(\mathcal{A})$ 
is also locally constant whose value on an open ball $D\subset M$ is given by 
$\pi_*\big(\mathcal{TC}_{M\times N}(\mathcal{A})(U) \cong \mathcal{TC}_{M\times N}(\mathcal{A})(U\times N)
\cong \int_N \mathcal{A}$.  This defines the canonical  $\mathbb{E}_d[M]$-algebra structure on 
$ \int_N \mathcal{A}$ and the result now follows from the equivalence~\eqref{eq:FubiniTCH}.
\end{proof}
\begin{example}
Let $M$, $N$ be $m+k$-framed and $n+\ell$-framed manifolds of respective dimension $m$, $n$ and $A$
 be an ${E}_{n+k+m+\ell}$-algebra. Then, the product $M\times N$ is canonically $m+n+k+\ell$-framed 
and $A$ is an $\mathbb{E}_{k+\ell}$-algebra. Then,  Corollary~\ref{C:FubiniTCH} yields an equivalence of
 $E_{k+\ell}$-algebras: 
$$\int_{M\times N} A \; \cong \; \int_{M} \Big(\int_{N} A\Big). $$
In particular, if $A$ is a CDGA, then Corollary~\ref{C:FubiniTCH} reduces to the exponential 
law for Hochschild chains (Proposition~\ref{P:product})
 under the equivalence between topological chiral homology and derived Hochschild chains (Theorem~\ref{T:TCH=CH}).
\end{example}

\end{document}